\documentclass[11pt]{article}
\RequirePackage[OT1]{fontenc}
\RequirePackage{amsthm,amsmath}
\RequirePackage{hyperref}
\usepackage{amsfonts}
\usepackage{amssymb}
\usepackage{amstext}
\usepackage{parskip}
\usepackage{fullpage}
\usepackage{graphicx}
\usepackage{subcaption}

\RequirePackage[numbers]{natbib}
\usepackage{amsfonts,amssymb,amsthm,amsmath,enumerate}

\newcommand{\prob}[1]{\ensuremath{\mathbb P}\left(#1\right)}

\newcommand{\R}{\ensuremath{\mathbb R}}

\newcommand{\Z}{\ensuremath{\mathbb Z}}

\newcommand{\size}[1]{\ensuremath{\left|#1\right|}}

\newcommand{\argmin}{\operatorname{argmin}}

\newcommand{\e}{\epsilon}
\newcommand{\ve}{\varepsilon}

\newcommand{\half}{\ensuremath{\frac{1}{2}}}
\newcommand{\silent}[1]{}

\newcommand{\D}{{\mathbb D}}
\newcommand{\Ball}{{\mathbf B}}

\newcommand{\C}{{\mathcal C}}

\newcommand{\M}{{\mathcal M}}

\newcommand{\cov}{\textsf{Cov}}

\newcommand{\gap}{\textsf{gap}}

\newcommand{\inv}[1]{\frac{1}{#1}}
\newcommand{\abs}[1]{\left\lvert#1\right\rvert}
\newcommand{\twonorm}[1]{\left\lVert#1\right\rVert_2}
\newcommand{\norm}[1]{\left\lVert#1\right\rVert}
\newcommand{\ldot}{\ldots}

\def\cpk{{\mathcal P}_k}
\def\Z{{\mathbb Z}}
\def\E{{\mathbb E}}

\def\M{{\mathcal M}}

\def\supp{\mathop{\text{supp}\kern.2ex}}
\def\argmin{\mathop{\text{arg\,min}\kern.2ex}}
\let\hat\widehat
\let\tilde\widetilde

\def\qed{\hskip1pt $\;\;\scriptstyle\Box$}

\newcommand{\vecone}{{\bf 1}}
\newcommand{\vc}{{\bf c}}

\newcommand{\BW}{\ensuremath{\mathbb{W}}}
\def\argmax{\mathop{\text{arg\,max}\kern.2ex}}

\def\fatnorm#1{|\kern-.2ex|\kern-.2ex| #1 |\kern-.2ex|\kern-.2ex|}

\newcommand{\shnorm}[1]{\lVert#1\rVert}
\newcommand{\fnorm}[1]{\left\lVert#1\right\rVert_F}
\newcommand{\infonenorm}[1]{\left\lVert#1\right\rVert_{\infty \to 1}}
\newcommand{\onenorm}[1]{\left\lVert#1\right\rVert_{1}}

\def\SDP{\mathop{\text{SDP}\kern.2ex}}
\newcommand{\func}[1]{\ensuremath{\mathrm{#1}}}
\newcommand{\diag}{\func{diag}}

\newcommand{\offd}{\func{offd}}

\newcommand{\Sp}{\mathbb{S}}
\newcommand{\ip}[1]{\;\langle{\,#1\,}\rangle\;}

\newcommand{\beq}{\begin{equation}}
\newcommand{\eeq}{\end{equation}}
\newcommand{\ben}{\begin{eqnarray}}
\newcommand{\een}{\end{eqnarray}}
\newcommand{\bnum}{\begin{enumerate}}
\newcommand{\enum}{\end{enumerate}}
\newcommand{\bit}{\begin{itemize}}
\newcommand{\eit}{\end{itemize}}
\newcommand{\bens}{\begin{eqnarray*}}
\newcommand{\eens}{\end{eqnarray*}}

\def \D {\mathbb{D}}

\DeclareMathOperator*{\Span}{span}

\newcommand{\tr}{{\rm tr}}

\newcommand{\mvec}[1]{\rm{vec}\left\{\,#1\,\right\}}
\newcommand{\opt}{\ensuremath{\text{opt}}}
\newcommand{\CC}{{\mathcal C}}
\newcommand{\Net}{{\mathcal N}}

\newtheorem{theorem}{Theorem}[section]

\newtheorem{fact}[theorem]{Fact}
\newtheorem{lemma}[theorem]{Lemma}
\newtheorem{proposition}[theorem]{Proposition}

\newtheorem{definition}[theorem]{Definition}

\newtheorem{corollary}[theorem]{Corollary}

\def\qed{\hskip1pt $\;\;\scriptstyle\Box$}

\newenvironment{proofof}[1]{\hspace*{20pt}{\it Proof}{ of #1}.\hskip10pt}{\qed\vskip5pt}
\newenvironment{proofof2}{\hskip10pt}{\qed\vskip5pt}

\begin{document}

\title{Semidefinite programming on population clustering: a global analysis}

  \author{Shuheng Zhou\\
    University of California, Riverside, CA 92521}

\date{}

\maketitle

\begin{abstract}
In this paper, we consider the problem of 
partitioning a small data sample of size $n$ drawn from a mixture of $2$ sub-gaussian
distributions.
Our work is motivated by the application of 
clustering individuals according to their population of 
origin using markers, when the divergence between the two populations is small.
We are interested in the case that individual features are of low average 
quality $\gamma$, and we want to use as few of them as 
possible to correctly partition the sample.
We consider semidefinite relaxation of an integer quadratic program 
which is formulated essentially as finding the maximum cut on a graph 
where edge weights in the cut represent dissimilarity scores 
between two nodes based on their features.
A small simulation result in Blum, Coja-Oghlan, Frieze and Zhou (2007,
2009) shows that  even when the sample size $n$ is small, by
increasing $p$ so that
$np= \Omega(1/\gamma^2)$,  one can classify a mixture of two product
populations using the spectral method therein with success rate reaching an ``oracle'' curve.
There the ``oracle'' was computed  assuming that distributions were known, 
where success rate means the ratio between correctly classified 
individuals and the sample size $n$.
In this work, we show the theoretical underpinning of this observed
concentration of measure phenomenon in high dimensions,
simultaneously for the semidefinite optimization goal and the spectral
method, where the input is based on the gram matrix computed from centered data.
We allow a full range of tradeoffs between the sample
size and the number of features such that the product of
these two is lower bounded by $1/{\gamma^2}$ so long as the number of
features $p$ is lower bounded by $1/\gamma$.

\end{abstract}

\section{Introduction}
\label{sec:intro}
We explore a type of classification problem that arises in the context
of computational biology.  The problem is that we are given a small
sample of size $n$, 
e.g., DNA of $n$ individuals (think of $n$ in the hundreds or thousands),
each described by the values of $p$ {\em features} or {\em markers},
e.g., SNPs (Single Nucleotide Polymorphisms, 
think of $p$ as an order of magnitude larger than $n$).
Our goal is to use these features to classify the individuals according to
their population of origin.  
Features have slightly different probabilities 
depending on which population the individual belongs to.
Denote by $\Delta^2$ the $\ell_2^2$ distance between two 
population centers (mean vectors), namely, $\mu^{(1)}, \mu^{(2)} \in \R^p$.
We focus on the case where $p >n$, although it is not needed.
 Note that $\Delta$ measures the Euclidean distance between
 $\mu^{(1)}$ and $\mu^{(2)}$ and thus represents
 their separation.
  
The objective we consider is to minimize the total data size $D=n p$
needed to correctly classify the individuals in the sample as a
function of the ``average quality'' $\gamma$ of the features:
\ben
\label{eq::Delta}
 \gamma := \Delta^2/p, \; \text{ where} \; \; \Delta^2  :=
 \sum_{k=1}^p (\mu^{(1)}_k - \mu^{(2)}_k)^2 \;
\text{ and }  \mu^{(i)} = (\mu^{(i)}_1, \ldots, \mu^{(i)}_p) \in \R^p, i =1, 2.
\een
Suppose we are given a data matrix $X \in \R^{n \times p}$ with
samples from two populations $\C_1, \C_2$, such  that 
\ben 
\label{eq::Xmean}
\forall i \in  \C_g, \; \; \E (X_{i j}) =  \mu^{(g)}_{j}  \; \; \;  
g =1,  2, \forall j =1, \ldots, p. 
\een 
Our goal in the present work is to estimate the group membership
vector $\bar{x} \in \{-1,1\}^n$ such that
\ben
\label{eq::xbar}
\bar{x}_i = 1 \; \text{ for } \; i \in \CC_1 \;  \text{ and } \;
\bar{x}_i =- 1 \; \text{ for } \; i \in \CC_2,
\een
where the sizes of clusters $\abs{\C_j} =: n_j, \forall j$ may not be
the same. Our ultimate goal is to estimate the solution to the discrete optimization problem:
\ben
\label{eq::quadratic}
\text{maximize} \; \; x^T \bar{A} x \quad \text{subject to} \; \; x 
\in  \{-1, 1\}^n 
\een
where $\bar{A}$ is a static reference matrix to be specified.
It was previously shown that, in expectation, among all balanced cuts
in the complete graph formed among $n$ vertices (sample points), the cut of maximum weight corresponds to
the correct partition of the $n$ points according to their distributions in
the balanced case ($n_1 = n_2  = n/2$).
Here the weight of a cut is 
the sum of weights across all edges in the cut, and the edge weight 
equals the Hamming distance between the bit vectors of the two
endpoints~\cite{CHRZ07,Zhou06}.
Under suitable conditions, the statement above also holds with high
probability (w.h.p.).

In other words, in the context of population clustering,
it has been previously shown one can use a random instance of the integer quadratic program: 
\ben
\label{eq::graphcut}
\text{maximize} \; \; x^T A x \quad \text{ subject to} \quad x \in  \{-1, 1\}^n
\een
to identify the correct partition of nodes according to their population
of  origin w.h.p., so long as the data size $D$ is sufficiently large
and the separation metric is at the order of $\Delta^2 = \Omega(\log n)$.
The analyses focused on the high dimensional setting, where $p \gg
n$~\cite{CHRZ07,Zhou06}.
Here $A =(a_{ij})$ is an $n \times n$ symmetric matrix
where for $1\le i, j \le n$, $a_{ij} = a_{ji}$ denotes the edge weight
between nodes $i$ and $j$, computed from the individuals’ bit vectors.
This result is structural, rather than algorithmic.
The integer quadratic program~\eqref{eq::quadratic}
(or~\eqref{eq::graphcut}) is NP-hard.
In a groundbreaking paper~\cite{GW95}, Goemans and Williamson
show that one can use semidefinite program (SDP) as relaxation to solve
these approximately.

In this paper, we propose a semidefinite relaxation framework, inspired
by~\cite{GV15}, where we design and analyze computational efficient
algorithms  to partition data into two groups approximately according
to their population of origin.
More generally, one may consider semidefinite relaxations for
the following sub-gaussian mixture model with $k$ centers
(implicitly, with rank-$k$ mean matrix embedded), where
\ben
\label{eq::model}
X_i = \mu^{(\psi_i)} + \Z_i
\een
where $\Z_1, \ldot, \Z_n \in \R^{p}$ are independent, sub-gaussian, mean-zero, random vectors and
$\psi_i: i \to \{1, \ldots, k\}$ assigns node $i$ to a group $\C_j$ with the 
mean $\mu^{(j)} \in \R^{p}$ for some $j \in [k]$.
Here we denote by $[k]$ the set of integers $\{1, \ldots, k\}$.
Here, each row vector of $X$ is a $p$-dimensional
sub-gaussian random vector and we assume rows are independent.
We will consider a flexible model of parametrization for $\Z_j, j \in [n]$ in 
Section~\ref{sec::theory}. 
In particular, we allow each population to have distinct covariance
structures, with diagonal matrices as special cases.
The analysis framework for the semidefinite relaxation by Guédon and Vershynin~\cite{GV15}
was set in the context of community detection in sparse networks, 
where $A$ represents the adjacency matrix of a random graph.
In other words, they study the semidefinite relaxation of the integer program 
\eqref{eq::graphcut}, where an $n \times n$ symmetric 
random adjacency matrix $A$ (observed) is used to replace the hidden
static matrix $\bar{A}$ in the original problem~\eqref{eq::quadratic}
such that $\E (A) = \bar{A}$.

The innovative proof strategy of~\cite{GV15} is to apply the Grothendieck's
inequality for the random error $A -\bar{A}$ rather than the original matrix $A$ as considered in the 
earlier literature. We call this approach the global analysis,
following~\cite{CCLN21}. With proper adjustments, we apply this
methodology to our settings and prove the first main
Theorem~\ref{thm::SDPmain} regarding the partial recovery of the group memberships based on $n$ sequences of
$p$ features following the mean model~\eqref{eq::Xmean}.
The important distinction is: here, we replace the random adjacency matrix $A$ arising from stochastic block models as considered in~\cite{GV15}
with an instance of symmetric matrix $A$, cf.~\eqref{eq::defineAintro},
computed from the centered data matrix which we now elaborate.
Let $\diag(A)$ and $\offd(A)$ be the diagonal and the off-diagonal 
part of matrix $A$ respectively.

\noindent{\bf Estimators.}
We propose the following estimators.
As in many statistical problems,
one simple but crucial step is to first obtain the centered data.
Let $\vecone_n = [1, \ldots, 1] \in \R^n$ denote a vector of all $1$s.
Let $X$ be a data matrix with row vectors $X_1, \ldots, X_n$ as defined
in~\eqref{eq::Xmean}.
Denote by
\ben
\label{eq::defineY}
Y & = & X - \vecone_n
\otimes \hat{\mu}_n = X - P_1 X, \quad \text{ where}
\quad P_1 = \inv{n} \vecone_n \vecone_n^T \; \; \text{ and }  \\
\label{eq::muhat}
\hat{\mu}_n & = & \inv{n}\sum_{i} X_i \; \text{is the average
  over $n$ random vectors in $\R^p$.}
\een
Loosely speaking, this procedure is 
called ``global centering'' in the statistical literature, for 
example, see~\cite{horns19}.
To estimate the group membership vector $\bar{x} \in \{-1, 1\}^n$,
we use the following adjusted $A$:
\ben
\label{eq::defineAintro}
A & := & Y Y^T -\lambda (E_n - I_n), \;\; \text{ where} \;
\lambda  = \frac{2}{n(n-1)}\sum_{i<j} \ip{Y_i, Y_j},
\\
\label{eq::defineLT}
E_n  & := & \vecone_n \vecone_n^T, \; \text{ and } \; I_n \text{ denotes the 
  identity matrix,}
\een
and consider the following semidefinite optimization problem: 
\ben 
\label{eq::sdpmain}
\text{SDP: } && \text{maximize}
\; \;  \ip{A, Z} \;\; \text{ subject to} \quad Z \succeq 0, 
\; I_n \succeq \diag(Z),
\een 
where $Z \succeq 0$ indicates that the matrix $Z$ is constrained to be positive semidefinite and 
$A \succeq B$ means that $A - B \succeq 0$; moreover, the inner 
product of matrices $\ip{A, B} = \tr(A^T B)$.
Here and in the sequel, denote by 
$\M^{+}_{G} := \{Z: Z \succeq 0,  I_n \succeq \diag(Z)\}$
the set of postive semidefinite matrices whose entries 
are bounded by 1 in absolute value.
More explicitly, the optimization problem  SDP \eqref{eq::sdpmain} is equivalent to:
\ben 
\label{eq::hatZintro}
\text{SDP2: } && \text{maximize}
\; \;  \ip{YY^T - \lambda E_n,  Z} \;\; \text{ subject to} \quad 
Z \succeq 0,  \diag(Z) = I_n.
\een
In our setting, centering the data $X$ plays a key role in the statistical 
analysis and in understanding the roles of sample size lower bounds 
for partial recovery of the clusters. We mention in passing that our
probabilistic analysis in terms of covariance estimation,
cf. Theorems~\ref{thm::YYcrossterms} and~\ref{thm::YYcovcorr},
can be readily applied to the rank-$k$ model (or $k$-means) settings
as well.  Before we continue, some
definitions and notations. Let $\Ball_2^n$ and $\Sp^{n-1}$ be the unit Euclidean ball and the unit sphere in $\R^n$ respectively.
\begin{definition}
  Recall for a random variable $X$, the $\psi_2$-norm of $X$,
  denoted by $\norm{X}_{\psi_2}$, is $\norm{X}_{\psi_2} = \inf\{t > 0\; : \; \E \exp(X^2/t^2) \le 2 \}$.  A random vector $W \in \R^m$ is called sub-gaussian if the one-dimensional marginals $\ip{W, h}$ are sub-gaussian random 
variables for all $h \in \R^{m}$:
(1) $W$ is called isotropic if for every $h \in \R^m$,
$\E \abs{\ip{W, h}}^2= \twonorm{h}^2$, where $\twonorm{h}^2 = \sum_{j=1}^m h_j^2$;
(2)  $W$ is $\psi_2$ with a constant $C_0$ if for every $h \in
  \R^m$,  $\norm{\ip{W, h}}_{\psi_2} \le C_0  \twonorm{h}$.
  The sub-gaussian norm of $W \in \R^m$ is denoted by
  \ben
  \label{eq::Wpsi}
  \norm{W}_{\psi_2}  := \sup_{h \in \Sp^{m-1}} \norm{\ip{W, h}}_{\psi_2}.
  \een
\end{definition}
Throughout this paper, we use $Z =(Z_{ij})$ to denote the positive 
semidefinite matrix in SDP objective functions. 
We use $\Z = (z_{ij})$ to denote the mean-zero random 
matrix with independent, mean-zero, sub-gaussian row vectors $\Z_1,
\ldots, \Z_n \in \R^p$ as considered in~\eqref{eq::model}, where for a constant $C_0$ ,
\ben
\label{eq::covZ1}
\forall j=1, \ldots, n, \; 
\norm{\ip{\Z_j, x}}_{\psi_2} & \le & C_0 \norm{\ip{\Z_j, x}}_{L_2}
\text{ for any } x \in  \R^{p}  \\
\label{eq::covZ2}
\text{ where } \; \norm{\ip{\Z_j, x}}^2_{L_2} &:=& \E \ip{\Z_j, x}^2 = x^T \E
(\Z_j \Z_j^T) x =: x^T \cov(\Z_j) x.
\een
Examples of random vectors with sub-gaussian marginals include the 
multivariate normal random vectors $\Z_j \sim N(0, \Sigma)$
with covariance $\Sigma \succ 0$, and vectors with independent Bernoulli random variables, where the mean parameters $p_j^i := \E (X_{ij})$ for all $i
\in [n]$ and $j \in [p]$ are assumed to be bounded away from 0 or 1;
See for example~\cite{CHRZ07,BCFZ09}.
For a symmetric matrix $A$, let $\lambda_{\max}(A)$ and 
$\lambda_{\min}(A)$ be the largest and the smallest 
eigenvalue of $A$ respectively. 
The operator norm $\twonorm{A}$ is defined to be 
$\sqrt{\lambda_{\max}(A^T A)}$.

\noindent{\bf Signal-to-noise ratios and sample lower bounds.}
Our work is inspired by the two threads of work in combinatorial
optimization and in community detection, and particularly
by~\cite{GV15} to revisit the max-cut problem~\eqref{eq::graphcut} and to
formulate the SDP~\eqref{eq::sdpmain}.
Our focus is on the sample size lower bound, similar to the earlier work 
of the author~\cite{CHRZ07,Zhou06}.
Moreover,  we adopt the following notion of signal-to-noise
ratio when the sub-gaussian random vectors $\Z_i$ in \eqref{eq::model} are isotropic:
\ben
\label{eq::SNR}
\text{\bf SNR isotropic: } \quad s^2 = ({\Delta^2}/{C_0^2})
\wedge ({n p \gamma^2}/{C_0^4})
\een
where $C_0$ is $\psi_2$-constant of the high dimensional vectors $\Z_i
\in \R^{p}$; This notion of SNR appears in~\cite{GV19}, which
can be properly adjusted when coordinates in $\Z_i$ are
dependent in view of~\eqref{eq::covZ1} and \eqref{eq::covZ2}:
\ben
\label{eq::SNR2}
\text{\bf SNR anisotropic: } \quad s^2 =\frac{\Delta^2}
{C_0^2 \max_{j}\twonorm{\cov(\Z_j)}} \wedge \frac{n p \gamma^2}
{C_0^4 \max_{j}\twonorm{ \cov(\Z_j)}^2}.
\een
We can rewrite the separation condition that is implicit in
Theorem~\ref{thm::SDPmain} as follows: 
  \ben
  \label{eq::miles}
  \Delta^2=  p \gamma \ge
   C_0^2 \max_{i} \twonorm{\cov(\Z_i)} \left(\inv{\xi^2} \vee \sqrt{\frac{p}{n 
        \xi^2}} \right) \; \text{ for some } \; 0< \xi <1/2.
  \een
We obtain in Theorem~\ref{thm::SDPmain}
misclassification error that is inversely proportional to the square root of
the SNR parameter $s^2$ as in~\eqref{eq::SNR}
(resp.~\eqref{eq::SNR2}) for isotropic $\Z_i$ (resp. for $\Z_i, i\in [n]$ with
covariance structures), assuming that it is lower bounded.
In the settings of Theorem~\ref{thm::SDPmain} and
Lemma~\ref{lemma::twogroup}, we are able to prove that the error
decays exponentially with respect to the SNR $s^2$ in
Theorem~\ref{thm::exprate}.
The implication of such an exponentially decaying error bound is: when $s^2
=\Omega(\log n)$, perfect recovery of the cluster structure is
accomplished. This result is in the same spirit as that
in~\cite{GV19}; See also~\cite{Royer17,FC18,FC21} and references
therein. Due to its significant length, we defer its proof to another paper.
We compare with~\cite{FC18,FC21,GV19} in the sequel.
Also closely related is the work of~\cite{BCFZ09}.

In more details, spectral algorithms in~\cite{BCFZ09} partition
samples based on the top few eigenvectors of the
gram matrix $XX^T$, following an idea which goes back at least to~\cite{Fie1973}.
In~\cite{BCFZ09}, the two parameters $n, p$ are 
assumed to be  roughly at the same order, hence not allowing a full
range of tradeoffs between the two  
dimensions as considered in the present work; cf~\eqref{eq::miles}.
The spectral analysis in this paper is based on $YY^T$, which
will directly improve the results in~\cite{BCFZ09} in the
sense that  we remove the lower bound on $n$, concerning spectral
clustering for $k=2$; cf. Theorem~\ref{thm::SVD}.
Such a lower bound on $n$ was deemed to be unnecessary given the
empirical evidence in~\cite{BCFZ09}; See  ``summary and future
direction'' in~\cite{BCFZ09}.

\subsection{Contributions}
In summary, we make the following theoretical contributions in this paper:
(a) We construct the estimators in~\eqref{eq::sdpmain}, which
crucially exploit the geometric properties of the two mean vectors, as
we show in Section~\ref{sec::estimators};
(b) Moreover, we use $YY^T$ (and the corresponding 
$A$~\eqref{eq::defineAintro}) instead of the gram matrix $X X^T$ as 
considered in~\cite{BCFZ09}, as the input to our optimization
algorithms, ensuring both computational efficiency and statistical
convergence, even in the low SNR case ($s^2= o(\log n)$);
(c) This approach allows a transparent and unified 
global and local analysis framework for the semidefinite  
programming optimization problem~\eqref{eq::sdpmain}, as given  
in Theorems~\ref{thm::SDPmain} and~\ref{thm::exprate} 
respectively; 
(d)  With the new results on concentration of measure  bounds on 
$\norm{YY^T -\E YY^T}$, we can simultaneously analyze the 
SDP~\eqref{eq::sdpmain} as well as spectral
algorithms based on the leading eigenvector of $YY^T$.
Here and in the sequel, we use $\norm{\cdot}$ to indicate either an 
operator or a  cut norm; cf. Definition~\ref{def::cutnorm}.

In Section~\ref{sec::kmeans}, we make further connections to the
existing semidefinite relaxations of the  $k$-means clustering
problems, which include the baseline spectral algorithm based on the singular
value decomposition (SVD) of $YY^T$. This allows even faster
computation. We analyze a simple spectral algorithm in Theorem~\ref{thm::SVD} through the
Davis-Kahan Perturbation Theorem, where we obtain error bounds 
similar to Theorem~\ref{thm::SDPmain}.
There, we further justify the global centering approach taken in the 
current paper. 
We compare numerically the two algorithms, namely, based on SDP
and spectral clustering respectively and show they indeed have similar
trends as predicted by the signal-to-noise ratio parameter.

\subsection{Notations and organizations}
\label{sec::notation}
Let $e_1, \ldots, e_n$ be the canonical basis of $\R^n$.
For a set $J \subset \{1, \ldots, n\}$, denote
$E_J = \Span\{e_j: j \in J\}$.
Let $P_1 =\vecone_n \vecone_n^T/n$ and $E_n = n P_1$.
Denote by $E_{n \times m} \subset \R^{n \times m}$ a matrix of all ones.
For a vector $v \in \R^n$, we use $v_{J}$ to denote the subvector
$(v_j)_{j \in J}$.
For a vector $x$, $\norm{x}_{\infty} := \max_{j} \abs{x_j}$, $\onenorm{x} := \sum_{j} \abs{x_j}$, 
and  $\twonorm{x} := \sqrt{\sum_{j} x^2_j}$; $\diag(x)$ denotes the diagonal matrix whose 
main diagonal entries are the entries of $x$.
For a matrix $B \in \R^{n \times n}$, $\tr(B) = \sum_{i=1}^n B_{ii}$. 
For a matrix $A = (a_{ij})$ of size $n \times m$,
let $\mvec{A}$ be formed by concatenating columns of matrix 
$A$ into a long vector of size $nm$; 
we use $\onenorm{A} =  \sum_{i=1}^n \sum_{j=1}^m |a_{ij}|$ denote
the $\ell_1$ norm of $\mvec{A}$, and $\fnorm{A} = (\sum_{i, j}
a_{ij}^2)^{1/2}$ the $\ell_2$ norm of $\mvec{A}$, which is also known as
the matrix Frobenius norm.
For a matrix $A$, let  $\norm{A}_{\infty} = \max_{i} \sum_{j=1}^n |a_{ij}|$ denote the
maximum absolute row sum;
Let $\norm{A}_{\max} = \max_{i,j} |a_{ij}|$ denote the component-wise
max norm.
For two numbers $a, b$, $a \wedge b := \min(a, b)$, and 
$a \vee b := \max(a, b)$.
We write $a \asymp b$ if $ca \le b \le Ca$ for some positive absolute
constants $c,C$ which are independent of $n, p$, and $\gamma$.
We write $f = O(h)$ or $f \ll h$ if $\abs{f} \le C h$ for some absolute constant
$C< \infty$ and $f=\Omega(h)$ or $f \gg h$ if $h=O(f)$.
We write $f = o(h)$ if $f/h \to 0$ as $n \to \infty$, where the
parameter $n$ will be the size of the matrix under consideration.
In this paper, $C, C_1, C_2, C_4, c, c', c_1$, etc,
denote various absolute positive constants which may change line by line.

The rest of the paper is organized as follows.
In Section~\ref{sec::theory}, we present the main theoretical results
of the paper. In Section~\ref{sec::estimators}, we present the proof outline for  
Theorem~\ref{thm::SDPmain} on the  
semidefinite program~\eqref{eq::sdpmain} and concentration of measure 
bounds on $B$ in Theorem~\ref{thm::reading}.
In Section~\ref{sec::kmeans},
we discuss various forms of semidefinite relaxations that have been 
considered in the literature,
and highlight the connections and main differences 
with the current work.
Section~\ref{sec::proofreadingmain} gives an outline of the arguments 
for proving Theorem~\ref{thm::reading}, highlighting concentration
bounds on $\norm{YY^T-\E YY^T}$ in Theorems~\ref{thm::YYaniso}
and~\ref{thm::YYnorm}.
In Section~\ref{sec::YYcutnorm}, we present main ideas in proving
Theorem~\ref{thm::YYnorm} with regards to the operator and cut norm
using independent design. In Section~\ref{sec::finalYYaniso}, we discuss correlated design and their concentration of measure bounds concerning
Theorem~\ref{thm::YYaniso}, one of the most technical results of this paper.
Section~\ref{sec::experiments} shows numerical results that validate
our theoretical predictions. We conclude in Section~\ref{sec::conclude}.
We defer all technical proofs to the supplementary material.

\section{Theory}
\label{sec::theory}
We will first construct a matrix $Y$ such that we
subtract the sample mean $\hat{\mu}_n \in \R^{p}$ as computed from \eqref{eq::muhat} from each
row vector $X_i$ of the data matrix. A straight-forward calculation leads to the expression of the
reference matrix $R$ in view of \eqref{eq::EYpre}, and hence $\E(Y) \E(Y)^T = R$ for $R$ as in~\eqref{eq::Rtilt}.
\begin{definition}{\textnormal{ \bf (The estimators)}}
\label{def::estimators}
Let $Y$ be as in~\eqref{eq::defineY} and  $X$ as in~\eqref{eq::Xmean}.
Denote by $\Z_i = X_i - \E (X_i)$ and
\ben
\label{eq::muvector}
\hat{\mu}_n - \E \hat{\mu}_n & := & \inv{n}
\sum_{i=1}^n \Z_i = \inv{n} \sum_{i=1}^n  X_i - \E (X_i),
\een
where $\mu_n$ is as in \eqref{eq::muhat}.

Clearly, by linearity of expectation, $\E \hat{\mu}_n = w_1 \mu^{(1)}+
w_2 \mu^{(2)}$, where $w_j = {\abs{C_j}}/{n}$ for $j = 1, 2$.  Hence
we have for $w_i = \abs{\C_i} /n, i =1, 2$
\ben
\label{eq::EYpre}
\E (Y_i) & := &
\left\{
  \begin{array}{rl} w_2 (\mu^{(1)} - \mu^{(2)})  & \text{ if }  \; i \in \C_1; \\
    w_1 (\mu^{(2)} - \mu^{(1)}) &  \text{ if }  \; i \in \C_2.
  \end{array}\right.
\een
\end{definition}

\begin{definition}{\textnormal{ \bf (The reference matrix)}}
\label{def::reference}
Denote by $n_1 = w_1 n$ and $n_2 = w_2 n$. For $Y$ as defined in
\eqref{eq::defineY} (cf. Definition~\ref{def::estimators}), and
$\Delta^2 =p \gamma$ as
in~\eqref{eq::Delta}, we have
\ben
 R= \E(Y) \E(Y)^T
    &=: &
\label{eq::Rtilt}
p \gamma 
\left[
\begin{array}{cc}
 w_2^2 E_{n_1} &- w_1 w_2E_{n_1 \times n_2} \\
 - w_1 w_2 E_{n_2 \times n_1} & w_1^2 E_{n_2} 
\end{array}
\right].
\een
\end{definition}

\begin{definition}{\textnormal{\bf (Data generative process.)}}
  \label{def::WH}
  Suppose that random matrix $\BW =(w_{jk})\in \R^{n \times m}$ has  $W_1, \ldots, W_n \in   \R^{m}$
 as row vectors, where $W_j, j \in [n]$ are independent, mean-zero, 
isotropic sub-gaussian random vectors  with independent entries satisfying 
 %whose $\psi_2$ norm 
\ben 
\label{eq::Wpsi2}
\forall j \in [n], \quad \cov(W_j) :=\E (W_j W_j^T) = I_m, \; \; 
\E[w_{jk}] = 0, \forall j, k, \quad \text{and } \; \; \max_{jk} \norm{w_{jk}}_{\psi_2} \le C_0. 
\een
%Clearly,  each vector $W_j$ has  sub-gaussian norm 
%being bounded
Suppose that we have for row vectors of $\Z \in \R^{n \times p}$,   $\forall j =1, \ldots, n$,
 \bens 
  \Z^T_j & = & W^T_j H_j^T \; \text{ where} \; H_j \in \R^{p \times
    m}, \; \; 0< \twonorm{H_j} < \infty,
  %\text{and } 
\eens
and $H_j$ is allowed to repeat, for example, across rows from the  same cluster $\C_i$ for
some $i =1, 2$. 
\end{definition}
Throughout this paper, we assume that $m \ge p$ to simplify our 
exposition, although this is not necessary.
First, Lemma~\ref{lemma::twogroup} characterizes the two-group design 
matrix variance and covariance structures to be considered in 
Theorems~\ref{thm::SDPmain} and~\ref{thm::exprate}.
It is understood that when $H_i$ is a symmetric square matrix, it can be
taken as the unique square root of positive semidefinite covariance
matrix, denoted by $\cov(\Z_j) := H_i H_i^T =:\Sigma_i \succeq 0$ for all $j \in
\C_i$.
\begin{lemma}\textnormal{\bf (two-group sub-gaussian mixture model)}
  \label{lemma::twogroup}
Denote by $X$ the two-group design matrix as considered in
\eqref{eq::Xmean}.
Let $W_1, \ldots, W_n \in \R^{m}$ be independent, mean-zero, 
isotropic, sub-gaussian random vectors satisfying~\eqref{eq::Wpsi2}. 
Let $\Z_j = X_j - \E X_j = H_i W_j$, for all $j \in \C_i$, where $H_i
\in \R^{p \times m}$, and $0< \twonorm{H_i} < \infty$, for $i \in \{1,
2\}$. Then $\Z_1, \ldots, \Z_n$ are independent sub-gaussian random vectors
with $\cov(\Z_i)$ satisfying \eqref{eq::covZ1} and \eqref{eq::covZ2}, where
\ben 
\label{eq::rowcov}
\forall j \in \C_i, \quad
\cov(\Z_j) := \E (\Z_j \Z_j^T) 
& = &  \E (H_i W_j W_j^T H_i^T)  =  H_i  H_i^T \text{ and } \;
V_i := \fnorm{H_i}^2 = \tr(\Sigma_i).
\een
\end{lemma}

\subsection{Main results}
Throughout this paper, we use $n_{\min}  := n w_{\min}$ and $n_{\max}
:= n w_{\max}$ to represent the size of the smallest and the largest 
clusters respectively.
Denote by $w_{\min} := \min_{j =1, 2} w_j$, where $w_j = n_j/n$.
We first make the following assumptions (A1) and (A2), assuming random
matrix $\Z$ has independent sub-gaussian entries, matching
the separation (and SNR) condition~\eqref{eq::kilo}.
As a baseline, we state in Theorem~\ref{thm::SDPmain} our 
first main result under (A1) and (A2).
However, the conclusions of Theorem~\ref{thm::SDPmain} hold for
the general two-group model so long as (A2) holds, upon
adjusting~\eqref{eq::kilo}. \\
\noindent{\bf (A1)}
Let $\Z = X - \E X = (z_{ij})$. Let $\Z_i = X_i - \E X_i, i =1, \ldots, n$ be 
  independent, mean-zero, sub-gaussian random vectors with independent coordinates such that for all $i,j$, $\norm{z_{ij} }_{\psi_2} := \norm{X_{ij} - \E 
    X_{ij}}_{\psi_2} \le C_0$. \\
\noindent{\bf (A2)}
  The two distributions have variance profile discrepancy bounded in the
  following sense:
\ben
\nonumber
\abs{V_1 - V_2}  & \le & \inv{3} \xi 
n p \gamma \; \; \text{ for some } \; \; 1> 2 \xi =
\Omega(1/n_{\min}), \; \text{ where} \;\\
\label{eq::Varprofile}
V_1 & = & \E \ip{\Z_j, \Z_j}  \quad \forall j \in \C_1 \;\; \text{and}\;\; 
V_2  =  \E \ip{\Z_j, \Z_j} \quad \forall j \in \C_2.
\een
\begin{theorem}
  \label{thm::SDPmain}
  Let $1> \delta = \Omega(1/n)$.
  Let $\C_j \subset [n]$ denote the group membership, with $\abs{\C_j} =
  n_j$ and $\sum_{j} n_j = n$.
  Suppose that for $j \in \C_i$, $\E X_j = \mu^{(i)}$, where $i =1,
  2$.    Let $\hat{Z}$ be a solution of the SDP~\eqref{eq::sdpmain}.
Suppose that (A1) and (A2) hold and for some absolute constants $C, C'$,
  \ben 
  \label{eq::kilo}
 p \gamma = \Delta^2 \ge \frac{C' C_0^2}{\xi^2}  \; \;  \text{ and } \; 
  p n \ge \frac{C C_0^4}{\xi^2 \gamma^2} , \; \text{ where  $\xi$ is
    the same as in~\eqref{eq::Varprofile}}.
  \een
Then with probability at least $1-2 \exp(-c n)$, we have
\ben
\label{eq::hatZFnorm}
\onenorm{\hat{Z} - \bar{x} \bar{x}^T} /n^2 & =: &   \delta \le {2 K_G
  \xi}/{w_{\min}^2}   \; \; \text{and} \; \;
\fnorm{\hat{Z} - \bar{x} \bar{x}^T}^2 / n^2\le {4 K_G \xi}/{w_{\min}^2},
\een
where $\bar{x}$ is as in \eqref{eq::xbar}.
The same error bounds \eqref{eq::hatZFnorm} also hold for the more general two-group sub-gaussian mixture model as considered
in~Lemma~\ref{lemma::twogroup},
upon adjusting~\eqref{eq::kilo}, so that~\eqref{eq::miles} holds.
\end{theorem}

\noindent{\bf Discussions.}
We give a proof outline of Theorem~\ref{thm::SDPmain} in
Section~\ref{sec::estimators} for completeness. Our proof covers both isotropic and 
anisotropic cases. See~\cite{Royer17,GV19} for justifications of
(A2). 
Our analysis shows the surprising result that 
Theorem~\ref{thm::SDPmain} does not depend on the clusters being 
balanced,  nor does it require identical variance profiles, so long as
(A2) holds.
Let us also choose a convex subset $\M_{\opt}$:
\ben
\label{eq::moptintro}
\M_{\opt} = \{Z: Z \succeq 0,  \diag(Z) = I_n\} \subset \M^{+}_{G}.
\een
Our proof follows the sequence of arguments in~\cite{GV15}, which were
specified for the stochastic block model. However, when adapting to
our setting, we crucially use the sub-gaussian concentration of
measure bounds as given in Theorems~\ref{thm::reading}
and~\ref{thm::YYnorm}, as well as verifying a non-trivial global
curvature of the excess risk $\ip{R,  Z^{*} - \hat{Z}}$ for the
feasible set $\M_{\opt}$ at the maximizer $Z^{*} = \argmax_{Z \in
  \M_{\opt}} \ip{R, Z} = \bar{x} \bar{x}^T$, cf. Lemma~\ref{lemma::onenorm}.
In order to control the  misclassification error using the
global approach, the parameters $(\delta, \xi)$ must satisfy the
following: in view of~\eqref{eq::kilo} and~\eqref{eq::hatZFnorm},
\ben
\label{eq::trend}
\xi^2 \asymp \frac{C_0^2}{p \gamma} \vee \frac{C_0^4}{ n p \gamma^2} =
1/s^2
\; \; \text{ and } \; \;\delta \le  \frac{2 K_G \xi}{w_{\min}^2}.
\een
Here the parameter $0< \xi^2 <1/4$ is understood to be chosen to be inversely
proportional to the SNR parameter $s^2$, so that
%the inequalities in~\eqref{eq::kilo} hold and
with probability at least $1-\exp(-c n)$,
\bens
\twonorm{YY^T - \E YY^T} \le C (C_0^2 (\sqrt{pn} \vee n) \vee C_0 n \sqrt{p \gamma} ) \asymp \xi n p \gamma 
\eens
as we will show in Theorems~\ref{thm::YYnorm} and \ref{thm::YYaniso}.

Clearly, the larger separation $\Delta^2$, the larger sample size $n$, and
the larger $s^2$, the easier it is for (A2) to be satisfied, since by definition and~\eqref{eq::kilo},
\bens
\xi n p \gamma \ge n/(\xi) \vee \inv{ \xi \gamma} \asymp \sqrt{s^2} (n
\vee \inv{\gamma})
\eens
Hence so far, the misclassification error rate $\delta \asymp
\xi/w_{\min}^2$ is bounded to be inversely proportional
to the square root of $s^2$. More explicitly, we have
Corollary~\ref{coro::misclass}.
\begin{corollary}{\textnormal{(Clustering with $o(n)$ misclassified vertices)}}
  \label{coro::misclass}
  Let $\hat{x}$ denote the eigenvector of $\hat{Z}$
  corresponding to the largest eigenvalue, with $\twonorm{\hat{x}}
  = \sqrt{n}$.
  Then in both settings of Theorem~\ref{thm::SDPmain}, we
  have with probability at least $1-2 \exp(-c n)$,
  \ben
  \label{eq::definedelta}
  \min_{\alpha = \pm 1} \twonorm{\alpha \hat{x} - \bar{x}}^2 \le
\delta n
  = \delta  \twonorm{\bar{x}}^2, \; \text{ where} \; \delta \le {32  K_G \xi}/{w_{\min}^2};
  \een
Moreover,  the signs of the coefficients of $\hat{x}$ correctly 
  estimate the partition of the vertices into the two clusters, up to 
  at most $\delta n$ misclassified vertices. 
  \end{corollary}
Next, we present in Theorem~\ref{thm::exprate}
(resp. Corollary~\ref{coro::misexp}) an error bound
\eqref{eq::hatZFnorm2} (resp.~\eqref{eq::eigenconv}), which decays
exponentially in the SNR parameter $s^2$ as defined in~\eqref{eq::SNR2}.
The settings as considered in Theorem~\ref{thm::exprate} include that of 
Theorem~\ref{thm::SDPmain} as a special case, which we elaborate 
in Section~\ref{sec::covest}.  We prove Theorem~\ref{thm::exprate} in a 
concurrent paper. Corollaries~\ref{coro::misclass} and~\ref{coro::misexp} 
follow from the Davis-Kahan Theorem, Theorems~\ref{thm::SDPmain}
and~\ref{thm::exprate} respectively, which we prove
in the supplementary Section~\ref{sec::proofofmisclass}.
\begin{theorem}
  \label{thm::exprate}
  Let $W_1, \ldots, W_n \in \R^{m}$ be independent, mean-zero, 
isotropic, sub-gaussian random vectors satisfying~\eqref{eq::Wpsi2}. 
   Suppose the conditions in Theorem~\ref{thm::SDPmain} hold,
   except that instead of (A1), we assume that the noise matrix $\Z = X-\E(X)$ is
generated according to Definition~\ref{def::WH}:
\bens
 \forall  j \in \C_i, \quad
\Z_j = H_i W_j  \; \text{ for } \; i \in \{1, 2\} \; \;  \text{ and } \; H_i \in \R^{p \times m},  
\eens
where $0< \twonorm{H_i} < \infty$.
Suppose that  for some absolute constant $C, C_{1}$,
\ben
\label{eq::NKlower}
p \gamma \ge \frac{C C_0^2 \max_{j}\twonorm{\cov(\Z_j)} }{w_{\min}^4} \;
\;\text{ and } \; \; np \ge \frac{C_{1} C_0^4 \max_{j}\twonorm{\cov(\Z_j)}^2}{\gamma^2 w_{\min}^4}.
\een
Let $s^2$ be as defined in~\eqref{eq::SNR2}.
Then with probability at least $1-2 \exp(-c_1 n) - c_2/n^2$,
\ben 
\label{eq::hatZFnorm2}
\onenorm{\hat{Z} - \bar{x} \bar{x}^T}/n^2 \le \exp(-c_0 s^2 w_{\min}^4) 
\een
for $\hat{Z}$ as in~\eqref{eq::sdpmain},  for some absolute constants $c, c_0, c_1$.
\end{theorem}

\begin{corollary}{\textnormal{(Exponential decay in $s^2$)}}
  \label{coro::misexp}
  Denote by $\theta_{\SDP} = \angle(\hat{x}, \bar{x})$, the angle
  between $\hat{x}$ and $\bar{x}$, where recall $\bar{x}_j = 1$ if $j
  \in \C_1$ and $\bar{x}_j = -1$ if  $j \in \C_2$, and $\hat{x}$ is as
  in Corollary~\ref{coro::misclass}.
  In the settings of Theorem~\ref{thm::exprate},
  with probability at least $1-2 \exp(-c n) - 2/n^2$,
for some absolute constants $c, c_0, c_1$,
\ben
\nonumber
 \sin(\theta_{\SDP}) & \le & 2 \twonorm{\hat{Z} - \bar{x}
   \bar{x}^T}/{n}  \le \exp(- c_1 s^2 w_{\min}^4) \; \;  \text{ and } \\
 \label{eq::eigenconv}
\min_{\alpha = \pm 1}  \twonorm{(\alpha \hat{x} - \bar{x})/\sqrt{n}}
  & \le &
  {2^{3/2} \twonorm{\hat{Z} - \bar{x} \bar{x}^T}}/{n}  \le
 4 \exp(-c_0 s^2 w_{\min}^4/2) 
 \een 
\end{corollary}

\subsection{Covariance estimation}
\label{sec::covest}
\noindent{\bf Remarks on covariance being diagonal.}
In Theorem~\ref{thm::SDPmain}, each noise 
vector $\Z_j, \forall j \in [n]$ has independent, mean-zero,
sub-gaussian coordinates with uniformly bounded $\psi_2$ norms.
Suppose that we generate two clusters according
to Lemma~\ref{lemma::twogroup},
with diagonal $H_1$ and $H_2$ respectively, where
\bens
H_j H^T_j =  \diag(\sigma_{1 j}^2, \ldots, \sigma_{p j}^2) \quad
\forall  j \in \{1, 2\},
\eens
Let  $\sigma_{\max}  :=  \max_{i} \twonorm{H_i} =
(\max_{j, k} \E (z_{j k}^2) )^{1/2}$.
Then for each row vector  in $\C_i$, we have
\bens 
\cov(\Z_j) & = & H_i H_i^T \; \text{  and hence }\; \; V_i  := \E
\ip{\Z_j, \Z_j} = \tr(H_i H_i^T) = \sum_{k=1}^p \sigma_{k i}^2
\eens
 by Lemma~\ref{lemma::twogroup}, where $V_i$ is the common variance
 profile for nodes $j \in \C_i$. 
Now, we have by independence of coordinates of $\Z_j$ and by definition of~\eqref{eq::Wpsi} 
\bens
\norm{\Z_j}_{\psi_2} & := & \sup_{h \in \Sp^{p-1}}\norm{\ip{\Z_j, 
    h}}_{\psi_2} \le C \max_{k \le p} \norm{z_{j k}}_{\psi_2}
\le  C C_0 (\max_{i} \twonorm{H_i})
\eens
where $\Sp^{p-1}$ denotes the sphere in $\R^{p}$, and  we
use \eqref{eq::Wpsi2} and the fact that
\bens
\max_{k \le p} \norm{z_{j k}}_{\psi_2}
\le  \sigma_{\max} (\max_{j,k}\norm{w_{j k}}_{\psi_2}) 
\le  \sigma_{\max} C_0 \; \text{ since} \;  \norm{w_{jk}}_{\psi_2} \le C_0, \forall j, k.
\eens
Then clearly, $\max_{j} \twonorm{\cov(\Z_j)} =  \max_{i} \twonorm{H_i H_i^T} =
\sigma_{\max}^2$ and hence \eqref{eq::NKlower} implies
that~\eqref{eq::kilo} holds; cf.

\noindent{\bf Remarks on more general covariance.}
When we allow each population to have distinct
covariance structures following Theorem~\ref{thm::exprate},
we have for some  universal constant $C$, and for all $j \in \C_i$, 
\ben
\label{eq::Zpsi2}
\norm{\Z_j}_{\psi_2}  :=  \sup_{h \in \Sp^{p-1}}\norm{\ip{\Z_j,
    h}}_{\psi_2}  &\le & \norm{W_j}_{\psi_2} \twonorm{H_i} \le C C_0 \max_{i} \twonorm{H_i} \; \text{ since } \\
\forall h \in \Sp^{p-1}, \quad \norm{\ip{\Z_j, h}}_{\psi_2} & = &
 \norm{\ip{H_i W_j,  h}}_{\psi_2} \le  \norm{W_j}_{\psi_2}  \twonorm{H_i^T h}
\een
 where $\norm{W_j}_{\psi_2}  \le   C C_0$  by definition
 of~\eqref{eq::Wpsi2}.
 Without loss of generality (w.l.o.g.), one may assume that $C_0 = 1$, as one can 
 adjust $H_i$  to control the upper
 bound in~\eqref{eq::Zpsi2} through $\twonorm{H_i}$.
As we will show in Theorems~\ref{thm::YYcrossterms} and \ref{thm::YYcovcorr},
with probability at least $1- 2 \exp(-c n)$, for $\Z$ as in Definition~\ref{def::WH},
\ben
\label{eq::ZZoppre}
\inv{p}\twonorm{\Z \Z^T - \E \Z \Z^T}
&\le& C'  (C_0 \max_{i} \twonorm{H_i})^2 \left(\sqrt{\frac{n}{p}} \vee \frac{n}{p}\right), 
\een
for absolute constants $c, C'$.
We discuss the concentration of measure bounds on $\norm{YY^T - \E YY^T}$,
using \eqref{eq::ZZoppre} in Sections~\ref{sec::reduction}
and~\ref{sec::finalYYaniso}; cf. Lemmas~\ref{lemma::tiltproject}
and~\ref{lemma::YYdec}.

\subsection{Related work}
\label{sec::related}
In the present work, we use semidefinite relaxation of 
the graph cut problem~\eqref{eq::graphcut}, which was originally formulated 
in~\cite{CHRZ07,Zhou06} in the context of population clustering.
The biological context for this problem is we are
given DNA information from $n$ individuals from $k$ populations 
of origin and we wish to classify each individual into the correct category.  
DNA contains a series of markers called SNPs, each of which has 
two variants (alleles). 
Given the population of origin of an individual, the
genotypes can be reasonably assumed to be generated by drawing alleles
independently from the appropriate distribution.
In the theoretical computer science literature,
earlier work focused on learning from mixture of well-separated
Gaussians (component distributions), where one aims to 
classify each sample according to which component distribution it 
comes from; See for example~\cite{DS00,AK01,VW02,AM05,KMV10,KK10}.
In earlier works~\cite{DS00,AK01}, the separation requirement 
depends on the number of dimensions of each distribution; this has recently 
been reduced to be independent of $p$, the dimensionality of 
the distribution for certain classes of distributions~\cite{AM05,KSV05}.
While our aim is different from those results, where $n > p$ is almost
universal and we focus on cases $p > n$, we do have one common
axis for comparison, the $\ell_2$-distance between any two centers of the
distributions as stated in~\eqref{eq::datasize}, which is essentially optimal.

Suppose~\eqref{eq::kilo} holds so that the $\ell_2$-separation and total data size satisfy
\ben
\label{eq::datasize}
\Delta^2 := p \gamma = \tilde\Omega(1/{(\xi^2)})  \; \text{ and } \; p n =
\tilde\Omega({1}/{(\xi^2 \gamma^2)}), \; \;
\; \text{ where } \; 1/n <\xi \le c w_{\min}^2
\een
and the $\tilde\Omega(\cdot)$ symbol only hides $\psi_2$-constants for
the high dimensional sub-gaussian random vectors $\Z_i \in \R^{p}$
in~\eqref{eq::model}.
Our results show that even when $n$ is small, by increasing $p$ so
that the total sample size satisfies~\eqref{eq::datasize},
we ensure partial recovery of cluster structures using the
SDP~\eqref{eq::sdpmain} or the spectral algorithm as described in
Theorem~\ref{thm::SVD}.
Previously, such results were only known to exist for balanced max-cut
algorithms~\citep{Zhou06,CHRZ07}, where $\tilde\Omega(\cdot)$ symbol
in~\eqref{eq::datasize} may also hide logarithmic factors.
Results in~\cite{Zhou06,CHRZ07} were among the first such
results towards understanding rigorously and intuitively why their
proposed algorithms and previous methods
\cite{PattersonEtAl,PriceEtAl} work with low sample settings when $p \gg n$ and $n p$ satisfies~\eqref{eq::datasize}.
These earlier results still need the SNR to be at the order of $s^2 =
O(\log n)$; Moreover these results were structural as no
polynomial time algorithms were given for finding the max-cut.

The main contribution of the present work is: we use the
proposed SDP~\eqref{eq::sdpmain} and the related spectral algorithms
to find the partition, and prove quantitively tighter bounds than those
in~\cite{Zhou06,CHRZ07} by removing these logarithmic factors.
Recently, this barrier has also been broken down by the sequence of
work~\cite{Royer17,FC18,GV19}, which we elaborate in
Section~\ref{sec::kmeans}, cf. Variation 3.
For example,~\cite{FC18,FC21} have also established 
exponentially decaying error bounds with respect to
an appropriately defined SNR,
which focuses on balanced clusters and requires an extra
$\sqrt{\log n}$ factor in \eqref{eq::FZ18} in the second component:
\ben
\label{eq::FZ18}
\text{In~\cite{FC18}, cf. eq.(8): } \quad \Delta^2 = p \gamma &= & 
  \Omega\left(1 +  \sqrt{\frac{p \log n}{n}}  \right) \quad  \text{or } \\
  \label{eq::FZ21}
    \text{In~\cite{FC21},  cf. eq.(13): } \quad \Delta^2 = p \gamma &= & \Omega\left((1 \vee \frac{p}{n})  + \sqrt{\frac{p \log n}{n}}  \right) 
    \een
As a result, in~\eqref{eq::FZ21}, a lower bound on the sample size is
imposed: $n \ge 1/\gamma$ in case $p > n$,
and moreover, the size of the matrix $n p  \ge \log n/\gamma^2$,
similar to the bounds in~\cite{BCFZ09}; cf. Theorem 1.2 therein.
We refer to~\cite{CHRZ07,BCFZ09} for references to 
earlier results on spectral clustering and graph partitioning.
We also refer to~\cite{KMV10,RCY11,GV15,Abbe16,BWY17,CY18,BMVV+18,GV19,LLLS+20,FC21,LZZ21,AFW22,nda22}
and references therein for related work on the Stochastic Block Models
(SBM), mixture of (sub)Gaussians and clustering in more general metric spaces.
Our proof technique may be of independent interests, since centering the data
matrix so that each column has empirical mean 0 is an idea broadly
deployed in statistical data analysis.

\section{The (oracle) estimators and the global analysis}
\label{sec::estimators}
Exposition in this subsection follows that of~\cite{GV15}, which we
include for self-containment. First we state
Grothendieck's inequality following~\cite{GV15}.
The concept of cut-norm plays a major role in the work of Frieze and 
Kannan~\cite{FK99} on efficient approximation algorithms for dense graph and 
matrix problems. The cut norm is also crucial for 
the arguments in~\cite{GV15} to go through. 
\begin{definition}{\textnormal{\bf (Matrix cut norm)}}
\label{def::cutnorm}
For a matrix  $A =(a_{ij})$, we denote by $\norm{(a_{i j})}_{\infty\to1}$ its
$\ell_{\infty} \to \ell_1$ norm, which is
\bens
\norm{(a_{i j})}_{\infty\to1} & = &
\max_{\norm{s}_{\infty} \le 1} \norm{A s}_1 = 
\max_{s, t \in \left\{-1, 1\right\}^n}
\ip{A, st^T} 
%=  \max_{s, t \in \left\{-1, 1\right\}^n} \sum_{i, j=1}^n a_{i, j} s_i t_j
\eens
This norm is equivalent to the matrix cut norm
defined as: for $A \in \R^{m \times n}$,
\bens
\norm{A}_{\square} & = &
\max_{I \subset [m], J \subset [n]} \abs{ \sum_{i
    \in I} \sum_{j \in J} a_{i, j}}
\eens
and hence
\bens
  \infonenorm{A}
  & = & \max_{x,y \in \{-1, 1\}^n}
  \sum_{i=1}^n \sum_{j=1}^n a_{ij}  x_i y_i \le 
  \twonorm{x}\twonorm{y} \twonorm{A} \le n \twonorm{A} 
  \eens
\end{definition}

\begin{theorem}\textnormal{(Grothendieck's inequality)}
  Consider an $n \times n$ matrix of real numbers $B=(b_{ij})$. Assume
  that, for any numbers $s_i, t_j   \in \{-1, 1\}$, we have
 \ben
 \label{eq::integer}
\abs{ \sum_{i,j} b_{ij} s_i t_j} = \abs{\ip{B, s t^T}}\le 1
\een
 Then for all vectors $S_i, V_i \in B_2^n$, we have
$\abs{  \sum_{i,j} b_{ij} \ip{S_i,  V_j} }= \abs{\ip{B, S V^T}} \le K_G$,
where $K_G$ is an absolute constant referred to as the Grothendieck's constant:
\ben
\label{eq::KG}
K_G \le \frac{\pi}{2 \ln(1+ \sqrt{2})} \le 1.783.
\een
\end{theorem}
Here $B_2^n = \{x \in \R^n: \twonorm{x} \le 1\}$ denotes the unit ball for 
Euclidean norm. Consider the following two sets of matrices:
\bens
\M_1 := \left\{st^T \;  : \; s, t \in \{-1, 1\}^n\right\}, \quad \M_G :=
\left\{S V^T  : \; \text{ all rows}  \; S_i, V_j \in B_2^n \right\}.
\eens
Clearly, $\M_1 \subset \M_G$.
As a consequence, Grothendieck's inequality can be stated as follows: 
\ben
\label{eq::GD}
\forall B \in \R^{n \times n} \; \; 
\max_{Z \in \M_G} \abs{\ip{B, Z}} \le K_G \max_{Z \in \M_1} \abs{\ip{B, Z}}.
\een
Clearly, the RHS \eqref{eq::GD} can be related to the cut norm in
Definition~\ref{def::cutnorm}:
\ben
\label{eq::GD1}
\max_{Z \in \M_1} \abs{\ip{B, Z}} =
  \max_{s, t \in \{-1, 1\}^n} \ip{B, st^T} = \norm{B}_{\infty \to 1}
  \een
To keep the discussion sufficiently general, following~\cite{GV15}, 
we first let $\M_{\opt}$ be any subset of the Grothendieck's set 
$\M_G^{+}$ defined in~\eqref{eq::GDSet}:
\ben 
\label{eq::GDSet}
\M_G^{+} := \left\{Z :  Z \succeq 0, \diag(Z) \preceq I_n\right\}
\subset \M_G \subset [-1, 1]^{n \times n}. 
\een
Lemma~\ref{lemma::GVGD} elaborates on the relationship between 
$\hat{Z}$ for any given $B$ (random or deterministic), and $Z^{*}$ with respect to the  
objective function using $R$, as defined in \eqref{eq::hatZ}. Let
  \ben 
\label{eq::hatZ}
\hat{Z} := \argmax_{Z \in \M_{\opt}} \ip{B, Z} \quad \text{and} \quad Z^{*} := \argmax_{Z \in 
    \M_{\opt}} \ip{R, Z}
  \een
\begin{lemma}{\textnormal{(Lemma 3.3~\cite{GV15})}}
  \label{lemma::GVGD}
  Let $\M_{\opt}$  be any subset of $M_G^{+} \subset [-1, 1]^{n \times 
    n}$ as defined in~\eqref{eq::GDSet}.  Then for $\hat{Z}$ and $Z^{*}$ as defined in \eqref{eq::hatZ},
  \ben
  \label{eq::hatZsand}
  \ip{R, Z^{*}} - 2 K_G \norm{B -R}_{\infty \to 1} \le   \ip{R, \hat{Z}}
  \le   \ip{R, Z^{*}}
  \een
  where the Grothendieck's constant $K_G$ is the same as defined in \eqref{eq::KG}.
  Then
  \ben
\label{eq::upperZZR}
0 \le  \ip{R, Z^{*} - \hat{Z}} & \le &  2 K_G \norm{B -R}_{\infty \to 1} \;\;
 \text{moreover, we have} \\
\label{eq::supZZR}
\sup_{Z \in  \M_{\opt}} \abs{\ip{B- R, Z - Z^{*}} } & \le &  2 K_G \norm{B -R}_{\infty \to 1} 
  \een
\end{lemma}
Lemma~\ref{lemma::GVGD}  shows that $\hat{Z}$ as defined in 
\eqref{eq::hatZ}   for  the original problem for a given $B$
%\eqref{eq::origin}
provides an almost optimal solution to the reference problem if the 
original matrix $B$ and the reference matrix $R$ are close. 
Lemma~\ref{lemma::GVGD} motivates the consideration of the oracle $B$
as defined in \eqref{eq::sdpB} in Section~\ref{sec::estimators} and $\hat{Z}$ as in \eqref{eq::hatZ}.  
Lemma~\ref{lemma::GVGD} appears as Lemma 3.3 in~\cite{GV15}. 
We include the proof in the supplementary
Section~\ref{sec::proofofSDPglobal} for self-containment.

\subsection{The oracle estimators}
The overall goal of convex relaxation is to:
(a) estimate the solution of the 
discrete optimization problem~\eqref{eq::quadratic} with an 
appropriately chosen reference matrix $R$ such that 
solving the integer quadratic problem~\eqref{eq::quadratic} (with $R$ replacing $\bar{A}$) will 
recover the cluster {\it exactly}; (b)
Moreover, the convex set $\M^{+}_{G}$ (resp. $\M_{\opt}$) is chosen
such that the semidefinite relaxation of the static problem~\eqref{eq::quadratic} is tight.
This means that when we replace $A$ (resp. $A'$) with  $R = \E(Y)  \E(Y)^T$ in
SDP~\eqref{eq::sdpmain} (resp. SDP2~\eqref{eq::hatZintro}),
we obtain a solution $Z^{*} = \bar{x} \bar{x}^T$, which can then be used to
recover the clusters exactly; cf.  Lemma~\ref{lemma::ZRnormintro}.

Note that unlike the settings of~\cite{GV15}, $\E A \not= R$,
resulting in a bias; However,
a remedy is to transform~\eqref{eq::sdpmain} into an equivalent {\bf
  Oracle SDP} formulation to bridge the gap between $YY^T$ and the
reference matrix $R$ which we now define: recall
$\M_{\opt} = \left\{Z :  Z \succeq 0, 
  \diag(Z) = I_n\right\} \subset \M^{+}_{G}$,
\ben 
\label{eq::sdpB}
{\bf Oracle SDP: }   &&
\text{maximize}\; \; \ip{B, Z} \quad \text{ subject to} \quad Z \in
\M_{\opt} \; \text{where} \\
%\succeq 0, \diag(Z) = I_n \; \; 
\label{eq::defineBintro}
B  & := & A - \E \tau I_n \; \;\; \text{ where } \; \; \tau =\inv{n}
\sum_{i=1}^n \ip{Y_i, Y_i}
\een 
and $A$ is as in~\eqref{eq::sdpmain}.

Moreover, on $\M_{\opt}$, the adjustment term $\E \tau I_n$ plays no
role in optimization, since the extra trace term $\propto \ip{I_n, Z}
=  \tr(Z)$ is a constant function of $Z$ across the feasible set
$\M_{\opt}$.  However, the diagonal term $\E \tau I_n$ is added in
\eqref{eq::defineBintro} so that the bias $\norm{\E B -R}$ is 
small.
To conclude, the optimization goal~\eqref{eq::sdpmain} is equivalent 
to \eqref{eq::sdpB} in view of Proposition~\ref{prop::optsol}; 
cf~\eqref{eq::optsolAB}.
In words, optimizing the original SDP~\eqref{eq::sdpmain}
over the larger constraint set $\M^{+}_{G}$ is equivalent to 
maximizing $\ip{B, Z}$ over $Z \in \M_{\opt}$ as shown 
in~\eqref{eq::optsolAB}, where we replace the symmetric matrix $A$
with $B$.

\begin{proposition}
  \label{prop::optsol}
  The optimal solutions $\hat{Z}$ as in \eqref{eq::sdpmain}
  must have their  diagonals set to $I_n$.
Thus, the set of optimal solutions $\hat{Z}$ in~\eqref{eq::sdpmain}
coincide with those on the convex subset $\M_{\opt}$ as in \eqref{eq::moptintro},
\ben
\label{eq::Aquiv}
\argmax_{Z \in  \M^{+}_{G} }  \ip{A , Z}
& = &  \argmax_{Z \in \M_{\opt}}   \ip{A , Z} \\
\label{eq::optsolAB}
  & = &
  \argmax_{Z \in \M_{\opt}}  (  \ip{A , Z} - \E \tau \ip{I_n, Z} ).
   \een 
\end{proposition}
We prove Proposition~\ref{prop::optsol} in the supplementary
Section~\ref{sec::optsolAB}.
We emphasize that our algorithm solves the SDP~\eqref{eq::sdpmain}
rather than the oracle SDP~\eqref{eq::sdpB}. 
However, formulating the oracle SDP~\eqref{eq::sdpB}
helps us with the global analysis, in controlling $\norm{\E B - R}$, 
as we now show in Theorem~\ref{thm::reading}. 
\begin{theorem}{\bf{($R$ is the leading term)}}
\label{thm::reading}
Suppose the conditions in Theorem~\ref{thm::SDPmain} hold.
Then with probability at least $1-2\exp(-cn)$,
we have
\bens 
\twonorm{B - R} & \le &  \xi n p \gamma  \; \text{ and } \;
\infonenorm{B - R}  \le   \xi n^2 p \gamma 
\eens
\end{theorem}

\noindent{\bf Discussions.}
Notice that $B$ is not attainable, since we do not know $\E \tau$;
however, this is irrelevant, since in the proposed algorithm
\eqref{eq::sdpmain}, we are able to readily compute $A$ using the
centered data (or their gram matrix).
Theorem~\ref{thm::reading} is useful in proving Theorem~\ref{thm::SDPmain} 
in view of  Lemma~\ref{lemma::GVGD}; A proof sketch
 for Theorem~\ref{thm::reading} appears in  Section
 \ref{sec::proofreadingmain} and the complete proof appears
 in the supplementary Section~\ref{sec::proofofreading}.
The effectiveness of the SDP procedure~\eqref{eq::sdpmain} crucially depends on
controlling  the bias term $\norm{\E B - R}_{\infty \to 1}$ as well as
the concentration of measure bounds on $\norm{B - \E B}_{\infty \to
  1}$, which in turn depend on Lemma~\ref{lemma::EBRtilt}, 
Theorems~\ref{thm::YYnorm} and~\ref{thm::YYaniso}  respectively. 
As we will show in the proof of Theorem~\ref{thm::reading}, the bias term
\bens
\E B -R & = & \E Y Y^T - \E (Y) \E(Y)^T- \E \lambda (E_n - I_n) - \E \tau
I_n 
\eens
is substantially smaller than $\E A-R$ in the operator and cut norm,
under assumption (A2). Moreover, the concentration of measure bounds on $\norm{YY^T - \E  YY^T}$ imply that, up to a constant factor,  the same bounds also
hold for $\norm{B - \E B}$. Controlling both leads to the conclusion in 
Theorem~\ref{thm::reading}.

In Theorem~\ref{thm::SVD}, we prove convergence results on bounding
the angle and $\ell_2$ distance between the leading eigenvectors of
$R$ and $B$ (resp. $YY^T$) respectively.
Indeed, computing the operator and cut norm for $B-R$ is one of the 
key technical steps in the current work,
unifying Theorems~\ref{thm::SVD} and~\ref{thm::SDPmain}.

\subsection{Proof of  Theorem~\ref{thm::SDPmain}}
\label{sec::proofSDPintro}
Lemma~\ref{lemma::ZRnormintro} shows
that the outer product of group membership vector, namely,
$Z^*$ will  maximize $\ip{R, Z}$ among all $Z \in [-1, 1]^{n 
  \times n}$, and naturally among all $Z \in \M_{\opt}$.
The final result we need is to verify a non-trivial global curvature
of the excess risk $\ip{R, Z^{*} - \hat{Z}}$ for the feasible set $\M_{\opt}$ at the maximizer $Z^{*}$, which is given in Lemma~\ref{lemma::onenorm}.
We then combine  Lemmas~\ref{lemma::GVGD} and~\ref{lemma::onenorm},
and Theorem~\ref{thm::reading} to obtain the final error bound for
$\shnorm{\hat{Z}  -Z^{*}}$ in the $\ell_1$ or Frobenius norm. Recall $\onenorm{A} = \sum_{i, j} \abs{a_{ij}}$.
\begin{lemma}{(\textnormal{\bf Optimizer of the reference objective
      function})}
  \label{lemma::ZRnormintro}
  Let $R$ be as defined in Definition~\ref{def::reference}.
  Let  $\M_{\opt} \subseteq \M_G^{+} \subset [-1, 1]^{n \times n}$
  be as defined in~\eqref{eq::moptintro}. Then
\ben
\label{eq::refoptsol}
Z^{*} = \argmax_{Z \in \M_{\opt}}\ip{R, Z} =  \left[\begin{array}{cc}
  E_{n_1} &-E_{n_1 \times n_2} \\
 - E_{n_2 \times n_1} & E_{n_2} 
\end{array}
\right] = \bar{x} \bar{x}^T
\een
\end{lemma}

The proof of Lemma~\ref{lemma::onenorm} follows from ideas in 
Lemma 6.2~\cite{GV15}
%in the context of general stochastic block  model
and is deferred to the supplementary Section~\ref{sec::proofofonenorm}.
As a result, we can apply the Grothendieck's inequality for the random error 
$B - R$ (cf.~Lemma~\ref{lemma::GVGD}) to obtain an upper bound on 
$\ip{R, Z^{*} - \hat{Z}}$ uniformly for all $\hat{Z} \in \M_{\opt}$, 
where $Z^{*}$ is as defined in \eqref{eq::refoptsol}. 
Putting things together, we can prove Theorem~\ref{thm::SDPmain}.
\begin{lemma}
  \label{lemma::onenorm}
  Let $R$ be as defined in Definition~\ref{def::reference} and $Z^{*}$
  be as in \eqref{eq::refoptsol}.
  For every $Z \in \M_{\opt}$, 
 \ben
 \label{eq::Rlower}
 \ip{R, Z^{*} - Z} \ge p \gamma w_{\min}^2 \onenorm{Z- Z^{*}}.
 \een
\end{lemma}

\begin{proofof}{Theorem~\ref{thm::SDPmain}}
We will first conclude from Theorem~\ref{thm::reading} and Lemma~\ref{lemma::GVGD}
that the maximizer of the actual objective function 
$\hat{Z} = \argmax_{Z \in \M_{\opt}} \ip{B, Z},$
must be close to $Z^{*}$ as in \eqref{eq::refoptsol} in terms of the $\ell_1$ distance.
Under the conditions of Lemmas~\ref{lemma::GVGD}
and~\ref{lemma::onenorm}, 
\bens 
\onenorm{\hat{Z} - Z^{*}}/n^2 
& \le&
\frac{\ip{R, Z^{*} - \hat{Z}}}{n^2p \gamma w_{\min}^2 }  \le 
\frac{2 K_G \norm{B -R}_{\infty \to 1}}{n^2  p \gamma w_{\min}^2 }   \le  \frac{2 K_G \xi}{w_{\min}^2 } =: \delta 
\eens 
where by Theorem~\ref{thm::reading}, $\infonenorm{B - R} \le \xi n^2 p \gamma$.
Thus
\bens
\fnorm{Z^{*}- \hat{Z}}^2
& \le &  \norm{Z^{*}- \hat{Z}}_{\max} \onenorm{Z^{*}- \hat{Z}}\le  2 \delta
%\frac{8 K_G \xi n^2}{\min_{j =1, 2} w_j^2 }
%& \le & \frac{4 K_G \norm{B -R}_{\infty \to 1}}{K \gamma \min_{j =1, 2} w_j^2 } \le
\eens
where all entries of $\hat{Z}, Z^{*}$ belong to $[-1, 1]$ and hence
$\norm{Z^{*}- \hat{Z}}_{\max} \le 2$.
%Notice that this bound is off by a factor of $2$ compared to the bound for balanced case. 
\end{proofof}

\section{Semidefinite programming relaxation for clustering}
\label{sec::kmeans}
Denote by $X \in \R^{n \times p}$ the data matrix with row vectors 
$X_i$ as in~\eqref{eq::SSE}.
The $k$-means criterion of a partition $\C = \{C_1, \ldots, \C_k\}$ of sample points $\{1,
\ldots, n\}$ is based on the total sum-of-squared Euclidean distances from each point
$X_i \in \R^{p}$ to its assigned cluster centroid $\vc_j$, namely,
\ben
\label{eq::SSE}
f(X, \C, k) := \sum_{j=1}^k \sum_{i \in \C_j} \twonorm{X_i - \vc_j}^2 \; \text{
  where} \; \; \vc_j := \inv{\abs{\C_j}} \sum_{\ell \in \C_{j}}
X_{\ell} \in \R^{p}
\een
Getting a global solution to~\eqref{eq::SSE} through an integer
programming formulation as in~\cite{PX05,PW07}, is NP-hard and it is NP-hard for $k=2$~\cite{DFK+04,ADHP09}.
Various semidefinite relaxations of the objective function have been
considered in different contexts.  We refer
to~\cite{ZDGH+02,PW07,ABCK+15,IMPV17,LLLS+20,MVW17,Royer17,FC18,GV19,FC21}
and references therein for a more complete picture.
Let $\Psi_n$ denote the linear space of real $n$ by $n$ symmetric
matrices.

\noindent{\bf Representation of the partition.}
The work by~\cite{ZDGH+02,PX05,PW07} show that minimizing the $k$-means objective $f(X, 
\C, k)$ is equivalent to solving the following maximization problem:
\ben 
\label{eq::relax7}
&& \text{maximize} \; \; \ip{\hat{S}_n, Z}  \quad \text{ s.t. }  Z \in 
\cpk 
\een
where  $\hat{S}_n = X X^T$ and the constraint set $\cpk$ is defined as
in~\eqref{eq::cpk}:
\ben
\label{eq::cpk}
\mathcal{P}_k =\{B \in \Psi_{n}: B \ge 0, B^2 = B, B \vecone_n = \vecone_n, \tr(B) = k\}
\een
where $B \ge 0$ means that all elements of $B$ are nonnegative.
Hence matrices in $\cpk$ are block diagonal, symmetric, nonnegative
projection matrices with $\vecone_n$ as an eigenvector.
The following matrix set $\Phi_{n,k}$ is a compact convex subset of 
$\Psi_n$, for any $k \in [n]$: 
\ben 
\label{eq::tracek}
\Phi_{n,k}
= \left\{Z \in \Psi_n:  I_n \succeq Z \succeq 0, \tr(Z) = k \right\}
\een
\noindent{\bf Variation 1.}
Peng and Wei~\cite{PW07} first replace the requirement that $Z^2 =
Z$, namely, $Z$ is a projection matrix, with the relaxed condition that
all eigenvalues of $Z$ must stay in $[0, 1]$: $I_n \succeq Z \succeq
0$. Now consider the following semidefinite relaxation of ~\eqref{eq::relax7},
\ben
\label{eq::relax16}
&& \text{maximize}
\quad \ip{\hat{S}_n, Z} \text{ s.t.} \;Z \in \M_k \; \text{
  where } \; \M_k = \{Z \in \Phi_{n,k}: Z  
 \ge 0, Z \vecone_n = \vecone_n \}
  \een
  The key differences between this and the SDP~\eqref{eq::sdpmain} are:
  (a) In the convex set~$\M_{\opt}$~\eqref{eq::moptintro}, we do not enforce that all
entries are nonnegative, namely, $Z_{ij} \ge 0, \forall i, j$; This
allows faster computation; (b) In order to derive concentration of
measure bounds that are sufficiently tight,
we make a natural, yet important data processing step in the current 
work, where we center the data according to their column means
following Definition~\ref{def::estimators} before computing $A$ as in~\eqref{eq::defineAintro};
(c) Given this centering step,  we do not need to enforce $Z \vecone_n
= \vecone_n$. See Variation 2 for details.

\noindent{\bf Variation 2.}
To speed up computation, one can drop the nonnegative 
constraint on elements of $Z$ in \eqref{eq::relax16}~\citep{ZDGH+02, PW07}.
The following semidefinite relaxation is also considered
in~\cite{PW07}: 
\ben
\label{eq::relax17}
 \text{maximize} \quad \ip{\hat{S}_n, Z}
 && \text{ s.t. } \; Z \vecone_n = \vecone_n, Z \in \Phi_{n,k}
\; \;\text{ for } \;  \Phi_{n,k} \text{ as in}~\eqref{eq::tracek}.
\een
Moreover, Peng and Wei~\citep{PW07} show that the set of feasible
solutions to~\eqref{eq::relax17} have immediate connections to the SVD
of $YY^T$, via the following reduction step, closely related to our proposal.
When $Z$ is a feasible solution to~\eqref{eq::relax17},
$\vecone_n/\sqrt{n}$ is the unit-norm leading eigenvector of $Z$ and 
one can define
\ben 
\label{eq::Z1}
Z_1 & := & Z - \inv{n} \vecone_n \vecone_n^T \; \; \text{ and hence 
}\; \; Z_1 :=  (I-P_1) Z = (I-P_1)Z (I-P_1).
\een 
Then $\tr(Z_1) =\tr(Z) -1 = k-1$ and $Z_1 \in \Phi_{n, k-1}$.
Hence~\eqref{eq::relax17} is reduced to
 \ben 
 \label{eq::relax20}
 \text{maximize} \quad \ip{YY^T, Z_1} \quad \text{ s.t. }   I_n \succeq
 Z_1 \succeq 0, \tr(Z_1) = k-1
 \een
 since $YY^T = (I-P_1) \hat{S}_n (I-P_1)$.
Let $\lambda_1 \ge  \ldots \ge \lambda_{n-1}$ be the
largest $(n-1)$ eigenvalues of $YY^T$ in descending order.
The optimal solution to~\eqref{eq::relax20} can be achieved if and
only if $\ip{YY^T,  Z_1} = \sum_{i=1}^{k-1} \lambda_i$; see for example~\cite{OW93}.
Then the algorithm for solving~\eqref{eq::relax20} and
correspondingly~\eqref{eq::relax17} is given as follows~\citep{PW07}: \\
(a) Use singular value decomposition method to compute the first $k-1$
largest eigenvalues of $YY^T$, and their corresponding eigenvectors
$v_1, \ldots, v_{k-1}$; (b) Set
\bens
Z_1 =\sum_{j=1}^{k-1} v_j v_j^T; \; \; \text{and return} \; \; Z =
\inv{n} \vecone_n \vecone_n^T + Z_1 \; \text{ as a solution to \eqref{eq::relax17}}.
\eens
Now for $k=2$, we have $Z_1 = v_1 v_1^T$.
In Theorem~\ref{thm::SVD}, we show convergence for the angle as well as 
the $\ell_2$ distance between the two vectors $v_1$ and $\bar{v}_1$, 
where $v_1$ and $\bar{v}_1$ are the leading eigenvectors of $YY^T$ and 
the reference matrix $R$ respectively.
Theorem~\ref{thm::SVD} demonstrates another excellent application of our estimation procedure and concentration of measure bounds, 
namely, Theorem~\ref{thm::reading}.
\begin{theorem}{\textnormal{\bf (SVD: imbalanced case)}}
\label{thm::SVD}
Denote by $v_1$  the leading unit-norm eigenvector of $YY^T$,
which also coincides with that of $A$ \eqref{eq::defineAintro} and
$B$~\eqref{eq::defineBintro}. Let $\bar{v}_1$ be the leading unit-norm eigenvector of $R$ as  in~\eqref{eq::Rtilt}:
\ben
\label{eq::Rleadingv1}
\bar{v}_1 = [w_2 \vecone_{n_1}, -w_1 \vecone_{n_2}]/\sqrt{w_2 w_1 n}
= [\sqrt{w_2/w_1} \vecone_{n_1}, -\sqrt{w_1/w_2} \vecone_{n_2}]/{\sqrt{n} },
\een
where $\ip{\bar{v}_1, \vecone_n}=0$.
Then under the conditions in Theorem~\ref{thm::reading},
we have with probability at least $1-2 \exp(-c n)$,  for some absolute constants $c, c_0, c_1, c_2$,
\ben
\label{eq::angSVD}
\sin(\theta_1) & := & \sin(\angle({v}_1, \bar{v}_1)) \le 
\frac{2 \twonorm{B - R}}{w_1 w_2 n p \gamma} \le \frac{2 \xi}{w_1 w_2} \\
\label{eq::normSVD}
  \min_{\alpha=\pm 1}  \twonorm{\alpha v_1 -  \bar{v}_1}^2
  & \leq & \delta', \text{ where} \; \; \delta' = {8  \xi^2}/{(w_1^2
    w_2^2)} \le c_2 \xi^2/w_{\min}^2; 
\end{eqnarray}
where $\theta_1 = \angle({v}_1, \bar{v}_1)$ denotes
the angle between the two vectors $v_1$ and $\bar{v}_1$.
\end{theorem}

\begin{corollary}\textnormal{(Clustering with $o(n/s^2)$ misclassified
    vertices)}
  \label{coro::SVD}
  Suppose that $w_{1}, w_2 \in (0, 1)$ are bounded away from $0, 1$. 
  Under the conditions in Theorem~\ref{thm::SVD}, we have with  
probability at least $1-2 \exp(-c n)$,  for some absolute constants
$c$, the signs of the coefficients of $v_1$ correctly estimate the
partition of the vertices into two clusters, up to at most $O(\xi^2
n)$ misclassified vertices.
\end{corollary}

\noindent{\bf Discussions.}
We prove Theorem~\ref{thm::SVD} and its corollary in the supplementary
Section~\ref{sec::proofofSVD}. The signs of the coefficients of $v_1$ correctly estimate the partition of the vertices, up to at most $\delta' n \asymp \xi^2 n$ misclassified 
vertices, where recall $\xi^2 \asymp 1/s^2$~\eqref{eq::trend}.
Hence the misclassification error is bounded to be inversely proportional to the SNR
parameter $s^2$; cf. \eqref{eq::trend}.
This should be compared with~\eqref{eq::definedelta}, where we show in
Theorem~\ref{thm::SDPmain} that we
have up to at most $\delta n \asymp \xi n$  misclassified
vertices, which is improved to $O( n \exp(-c_0 s^2 w_{\min}^4))$ in Theorem~\ref{thm::exprate}.
Moreover, one can sort the values of $v_1$ and
find the {\it nearly optimal } partition according to the $k$-means 
criterion; See Section~\ref{sec::experiments} for Algorithm 2 and numerical examples.

\noindent{\bf Variation 3.}
The main issue with the $k$-means relaxation is that the solutions
tend to put sample points into groups of the same sizes, and moreover,
the diagonal matrix $\Gamma$ can cause a bias, where
\bens
\Gamma = (\E [\ip{\Z_i, \Z_j}])_{i,j} = \diag([\tr(\cov(\Z_1)),
\ldots, \tr(\cov(\Z_n)) ]),
\eens
especially when $V_1, V_2$ differ from each other; See
the supplementary Section~\ref{sec::proofofbias} for bias analysis.
In~\cite{Royer17,GV19,BGLR+20}, they propose a preliminary estimator
of $\Gamma$, denoted by $\hat{\Gamma}$, and consider
\ben
\label{eq::relaxadjust}
&& \hat{Z} \in \arg\max_{Z \in \M_k} \ip{X X^T - \hat{\Gamma}, Z} \; \;
\text{ where} \; \M_k \; \text{ is as in~\eqref{eq::relax16} }
\een
instead of the original Peng-Wei SDP relaxation~\eqref{eq::relax16}.
Although our general results in Theorem~\ref{thm::exprate} coincide
with that of~\cite{GV19} for $k=2$, we emphasize that we prove these
bounds for the SDP~\eqref{eq::sdpmain}, which is motivated by the
graph partition problem~\eqref{eq::graphcut}, while they establish
such bounds for the semidefinite relaxation based on the $k$-means
criterion~\eqref{eq::SSE} directly, following~\cite{PW07}.
There, cf.~\eqref{eq::relax16}, and \eqref{eq::relaxadjust},
the matrix $Z$ is not only constrained to be positive semidefinite but
also with non-negative entries. As mentioned, the advantage of
dropping the nonnegative constraints on elements of $Z$ in
\eqref{eq::sdpmain} is to speed up the computation.

Hence another main advantage of our SDP and spectral formulation is
that  we do not need to have a  separate estimator for
$\tr(\Sigma_j)$, where $\Sigma_j, j=1, 2$ denote the covariance
matrices of sub-gaussian random vectors $\Z_j, j \in [n]$, so long as
(A2) holds. When it does not, one may consider adopting 
similar ideas. We emphasize that part of our probabilistic bounds, 
namely, Theorems~\ref{thm::YYcrossterms} and~\ref{thm::YYcovcorr}, 
already work for the general $k$-means clustering problem.

\section{Outline of the arguments for proving 
  Theorem~\ref{thm::reading}}
\label{sec::proofreadingmain}
We emphasize that results in this section apply to both settings under 
consideration: design matrix with independent entries or 
with independent anisotropic sub-gaussian rows. This allows us to 
prove Theorem~\ref{thm::reading} for both cases.
Let $Y$ be as in Definition~\ref{def::estimators}.
By definition of~\eqref{eq::defineAintro} and \eqref{eq::defineBintro},
\ben
\label{eq::Bdev}
A - \E A =
B -\E B & := & Y Y^T - \E Y Y^T  - (\lambda- \E \lambda)(E_n - I_n) \\
\text{ hence} \;
\nonumber
\infonenorm{B - R}
& = & \infonenorm{B- \E B + \E B -R} \\
\label{eq::Bdev2}
& \le & \infonenorm{B- \E B} + 
\infonenorm{\E B -R} 
\een
We have by the triangle inequality,~\eqref{eq::Bdev},~\eqref{eq::Bdev2} and
the supplementary Lemma~\ref{lemma::TLbounds},
for 
\ben
\label{eq::YYop}
\twonorm{B - R}  &\le & 2 \twonorm{\Psi}+ \twonorm{\E B - R} \;
\text{ where } \; \Psi := YY^T -  \E (Y Y^T), \\
\nonumber
\text{ and } \; \;
\infonenorm{B - R}  &\le &
\infonenorm{\Psi} + n \twonorm{\Psi} +
\infonenorm{\E B - R}
\een
Lemma~\ref{lemma::EBRtilt}  states that the bias $\E B -R$ is
substantially reduced for $B$ as in \eqref{eq::defineBintro}, thanks
to the adjustment term $\E \tau  I_n$, and even more so when clusters
have similar variance profiles in the sense that
\eqref{eq::Varprofile} is bounded.
Theorem~\ref{thm::reading} follows immediately from
Lemma~\ref{lemma::EBRtilt} and
Theorem~\ref{thm::YYnorm} (resp.  \ref{thm::YYaniso}), where we bound
$\norm{YY^T -  \E (Y Y^T)}$ for design matrix with independent
entries (resp. with independent anisotropic sub-gaussian rows).
All results except for Theorems~\ref{thm::YYnorm} and \ref{thm::YYaniso} are stated as
deterministic bounds.
Let $c_7, c_8, C_2, C_3, \ldots$ be some absolute constants.
All constants such as $1/6, 2/3, \ldots$ are arbitrarily chosen.
\begin{lemma}
  \label{lemma::EBRtilt}
  Suppose (A2) holds.
  Suppose that $\xi \ge \inv{2n} (4 \vee \inv{w_{\min}})$ and $n \ge 4$.
  Then we have 
\bens 
\label{eq::EBRtiltnorm}
\twonorm{\E B - R} & \le & \frac{2}{3} \xi  n p \gamma  \; \text{ and }
\; 
\infonenorm{\E B - R} \le \frac{2}{3} \xi  n^2 p \gamma
\eens
Finally, when $V_1 = V_2$, we have $\infonenorm{\E B - R} \le  n
\twonorm{\E B - R} \le  p n \gamma/3$.
\end{lemma}

\begin{theorem}{\textnormal{\bf (Design with independent entries)}}
  \label{thm::YYnorm}
In the initial settings as specified in Theorem~\ref{thm::SDPmain},
suppose that $(A1)$, $(A2)$ and~\eqref{eq::kilo} hold with $\max_{j, k} \norm{z_{jk}}_{\psi_2} := \norm{X_{jk} - \E X_{jk}}_{\psi_2} \le C_0$.
Then, with probability at least $1 - 2\exp(-c_7 n)$,
  \bens
\twonorm{YY^T - \E (Y Y^T)}
& \le &
C_2 C_0^2 (\sqrt{p n} \vee n) + C_3 C_0 n \sqrt{p \gamma}
\le \inv{6}\xi n p \gamma  
\eens
\end{theorem}

\begin{theorem}{\textnormal{\bf (Anisotropic design matrix.)}}
  \label{thm::YYaniso}
Let $Y$ be as in Definition~\ref{def::estimators}.
Suppose all conditions in Theorem~\ref{thm::SDPmain} and
Lemma~\ref{lemma::twogroup} hold. Suppose~\eqref{eq::miles} holds.
Then with probability at least $1 - 2\exp(-c_8 n)$,
  \bens
  \twonorm{YY^T - \E (Y Y^T)}
& \le &
\inv{12} \xi n p \gamma  + C_4  (C_0 \max_{i} \twonorm{H_i})^2
(\sqrt{p n} \vee n) \le \inv{6}  \xi n p \gamma.
\eens
where $C_0$ is the same as in~\eqref{eq::miles} and~\eqref{eq::Zpsi2}.
\end{theorem}
We prove Lemma~\ref{lemma::EBRtilt} in the supplementary
Section~\ref{sec::proofofEBR}, where balanced cases are shown to be
slightly more tightly bounded; cf Lemma~\ref{lemma::unbalancedbias} therein.
We prove Theorems~\ref{thm::YYnorm} and~\ref{thm::YYaniso} in
 the supplementary Section~\ref{sec::proofofYYnorm} and Section~\ref{sec::finalYYaniso} respectively. It is understood that for both theorems, we also obtain 
$\infonenorm{YY^T - \E (YY^T)}$ to be within a factor of $O(n) 
\shnorm{YY^T - \E (YY^T)}_2$.
We prove Theorem~\ref{thm::reading}
in   the supplementary Section~\ref{sec::proofofreading}.

\subsection{Reduction}
\label{sec::reduction}
In this section, we present a unified framework for bounding
$\twonorm{Y Y^T - \E Y Y^T}$. First, 
\ben
\nonumber
\lefteqn{YY^T - \E (Y Y^T) 
  = YY^T - \E (Y) \E(Y)^T  + \E (Y) \E(Y)^T -\E (Y Y^T) } \\
\label{eq::projection}
& = &\E(Y)(Y-\E(Y))^T + (Y-\E(Y))(\E(Y))^T  +
\hat{\Sigma}_Y - \Sigma_Y
\een
where $\hat\Sigma_Y =  (Y-\E(Y))(Y-\E(Y))^T$ and
\bens
\hat{\Sigma}_Y -
\Sigma_Y
& = &  (Y-\E(Y)) (Y-\E(Y))^T + \E (Y) \E(Y)^T -\E (Y Y^T),
\eens
from which we obtain from the well known relationship on covariance matrix 
\bens 
\Sigma_Y := \E \left((Y-\E(Y))(Y-\E(Y))^T \right) = \E (YY^T) - \E(Y) \E(Y)^T.
\eens
We now state in Lemma~\ref{lemma::tiltproject} a reduction 
principle for bounding the first component 
in~\eqref{eq::projection}:
To control $$\norm{M_Y} =\norm{\E(Y)(Y-\E(Y))^T + (Y-\E(Y))(\E(Y))^T},$$
we need to bound the projection of each mean-zero random vector $\Z_j,
\forall j \in [n]$, along the direction of $v := \mu^{(1)}-\mu^{(2)}$.
In other words,  a particular direction for which we compute the 
one-dimensional marginals, is the direction between $\mu^{(1)}$ and
$\mu^{(2)}$.

\begin{lemma}{{\bf (Reduction: a deterministic comparison lemma)}}  
  \label{lemma::tiltproject}
  Let $\Z_j, j \in [n]$ be row vectors of $X - \E X$ and $\hat{\mu}_n$
  be as defined in~\eqref{eq::muhat}.
For $x_i \in \{-1,1\}$,
\ben
\sum_{i=1}^n x_i \ip{Y_i -\E Y_i, \mu^{(1)} -\mu^{(2)}}
& \le &
\label{eq::pairwise2}
\frac{2(n-1)}{n} \sum_{i=1}^n \abs{\ip{\Z_i, \mu^{(1)}  -\mu^{(2)}} } 
\een
Then we have for $M_Y :=\E(Y)(Y-\E(Y))^T + (Y-\E(Y))(\E(Y))^T$,
 \bens
\infonenorm{M_Y}
& \le & 8 w_1 w_2 (n-1) \sum_{i=1}^n \abs{\ip{\Z_i, \mu^{(1) }-\mu^{(2)
    } }}
\;\text{ and} \; \\
\nonumber
 \twonorm{M_Y}
& \le & 4 \sqrt{n}  \sqrt{w_1 w_2} \sup_{q \in S^{n-1}} \abs{\sum_{i} q_i \ip{\Z_i, \mu^{(1)} -\mu^{(2)}}}
\eens
\end{lemma}

Upon obtaining \eqref{eq::EYpre}, Lemma~\ref{lemma::tiltproject} is 
deterministic and does not depend on covariance structure of $\Z$. 
On the other hand, controlling the second component in
\eqref{eq::projection} amounts to the problem of covariance
estimation given the mean matrix $\E(Y)$;
Lemma~\ref{lemma::YYdec} is again deterministic, where we show that
controlling the operator (and cut) norm of $\hat\Sigma_Y -\Sigma_Y$ is
reduced to controlling that for $\Z \Z^T - \E(\Z \Z^T)$.   We prove  Lemmas~\ref{lemma::tiltproject} and~\ref{lemma::YYdec} in 
 the supplementary Sections~\ref{sec::tiltproj} and~\ref{sec::SigmaYYproj} respectively.
\begin{lemma}
  \label{lemma::YYdec}
Suppose that $Y$ and $Z$ are matrices as defined in
Definition~\ref{def::estimators}.
The following holds: 
\bens 
\hat{\Sigma}_Y - \Sigma_Y = (I-P_1) (\Z \Z^T - \E (\Z\Z^T)) 
(I-P_1) 
\eens 
cf. Proposition~\ref{prop::decompose}  in the supplementary material. 
 Then $\twonorm{\hat\Sigma_Y -\Sigma_Y}
 \le \twonorm{\Z \Z^T - \E(\Z \Z^T)}$.
\end{lemma}

\section{Proof outline of Theorem~\ref{thm::YYnorm}}
\label{sec::YYcutnorm}
We provide a proof outline for Theorem~\ref{thm::YYnorm} in this
section.
We will bound these two components~\eqref{eq::projection}
in Lemma~\ref{lemma::projerr} and Theorem~\ref{thm::YYcrossterms}
respectively. Lemma~\ref{lemma::projerr} follows from
Lemma~\ref{lemma::tiltproject} and the sub-gaussian concentration of
measure bounds in Lemma~\ref{lemma::isotropic}.
We will only state the operator norm bound in
Theorem~\ref{thm::YYcrossterms}, with the understanding  that cut norm
of a matrix is within $O(n)$ factor of the operator norm  on the same matrix. 
We defer the proof of Theorems~\ref{thm::YYnorm}
and~\ref{thm::YYcrossterms} to the supplementary
Sections~\ref{sec::proofofYYnorm} and~\ref{sec::mainwell}
respectively. The proof for Lemmas~\ref{lemma::projerr} and~\ref{lemma::isotropic}
appear in the supplementary Section~\ref{sec::isonorm}.
Let $c, c', c_1, c_5, C_3, C_4, \ldots $ be absolute constants.
\begin{lemma}{{\bf (Projection: probabilistic view)}}
\label{lemma::projerr}
Suppose conditions in Theorem~\ref{thm::YYnorm} hold.
Then we have with probability at least $1 -2 \exp(-c n)$, 
\bens
\label{eq::Yprop}
\twonorm{\E(Y)(Y-\E(Y))^T + (Y-\E(Y))(\E(Y))^T }
& \le & 
2 C_3 C_0 n \sqrt{p \gamma}   \text{ and} \;\\
\label{eq::Yprocut}
\infonenorm{\E(Y)(Y-\E(Y))^T + (Y-\E(Y))(\E(Y))^T }
& \le & C_4 C_0 n(n-1) \sqrt{p \gamma}
\eens 
\end{lemma}

Lemma~\ref{lemma::isotropic} follows from the sub-gaussian tail bound,
since the one-dimensional marginals of $\Z_j,  \forall j\in [n]$
are sub-gaussian with bounded $\psi_2$ norms.
Denote by
  \ben
\label{eq::definemu}
\mu:= \frac{\mu^{(1)}-\mu^{(2)}}{\twonorm{\mu^{(1)}-\mu^{(2)}}}
=\frac{\mu^{(1)}-\mu^{(2)}}{\sqrt{p \gamma}} \in \Sp^{p-1}
\een
\begin{lemma}
  \textnormal{\bf (Projection for sub-gaussian random vectors)}
  \label{lemma::isotropic}
  In the settings of Theorem~\ref{thm::YYnorm},
  suppose (A1) holds and $C_0 = \max_{i,j} \norm{z_{ij}}_{\psi_2}$.
Then for any $t > 0$, and any $u =(u_1, \ldots, u_n) \in \{-1, 1\}^n$ 
\ben
\prob{\sum_{i =1}^{n} u_i \ip{\Z_i, \mu}  \ge t}
\label{eq::sumZOEYsubg}
& \le &
2 \exp\left(- {c t^2}/{(C_0^2 n)}\right); \\
\text{and for any } \; q \in \Sp^{n-1},
\prob{\sum_{i=1}^n \ip{q_i \Z_i, \mu} \ge t }
\label{eq::qZOEYsubg}
& \le & 2\exp\left(- {c' t^2}/{C_0^2}\right)
\een
\end{lemma}

\begin{theorem}
  \label{thm::YYcrossterms}
In the settings of Theorem~\ref{thm::YYnorm},
we have with probability at least $1 - 2\exp(-c_6 n)$,
  \bens
  \twonorm{\hat\Sigma_Y - \E \hat\Sigma_Y}
  &=&  \twonorm{\E (\Z \Z^T) - \E (\Z \Z^T)}
   \le  C_2 C_0^2( \sqrt{p n} + n) \le \inv{12} \xi n p \gamma.
  \eens
\end{theorem}

\section{Proof outline for Theorem~\ref{thm::YYaniso}}
\label{sec::finalYYaniso}
We provide an outline for Theorem~\ref{thm::YYaniso} in this section.
First, we state Lemma~\ref{lemma::projWH}, 
where we extend Lemma~\ref{lemma::projerr} to the 
anisotropic cases.
The anisotropic version of Lemma~\ref{lemma::isotropic} is presented 
in the supplementary Lemma~\ref{lemma::anisoproj}.
The model under consideration in Theorem~\ref{thm::YYcovcorr}
is understood to be a special case of Theorem~\ref{thm::ZHW}.
Theorem~\ref{thm::YYaniso} follows from Theorem~\ref{thm::YYcovcorr}
and Lemma~\ref{lemma::projWH} immediately, and the probability
statements hold upon adjusting the constants. We defer all proofs to the supplementary
Section~\ref{sec::corrproofs}.
Let $c, c', C_2, C_3, C_4, \ldots$ be absolute constants.

  \begin{lemma}{{\bf (Projection: probabilistic view)}}
    \label{lemma::projWH}
  Let $\mu$ be as in~\eqref{eq::definemu} and $M_Y$ be as in 
Lemma~\ref{lemma::tiltproject}. 
    Suppose all conditions in Theorem~\ref{thm::YYaniso} hold.
Then with probability at least $1 -2 \exp(-c' n)$,   we have
\bens 
\infonenorm{M_Y} & \le &
C_4 (C_0 \max_{i} \twonorm{R_i \mu}) n(n-1) \sqrt{p \gamma}
\le \inv{12} \xi n(n-1) p \gamma \\
\twonorm{M_Y}
& \le &
2 C_3 (C_0 \max_{i} \twonorm{R_i \mu}) n \sqrt{p \gamma} \le 
 \inv{12} \xi n p \gamma 
 \eens
\end{lemma}

\begin{theorem}
  \label{thm::YYcovcorr}
  In the settings of Theorem~\ref{thm::YYaniso}, we have
  with probability at least $1 - 2\exp(-c_6 n)$, 
\ben
\label{eq::ZZHop}
\twonorm{\hat\Sigma_Y - \E \hat\Sigma_Y}
& \le &
\twonorm{\Z \Z^T - \E  \Z \Z^T} \le  C_2   (C_0 \max_{i}
\twonorm{H_i})^2 (\sqrt{n p} \vee n).
\een
\end{theorem}

\begin{theorem}{\textnormal\noindent{\bf (Hanson-Wright inequality for anisotropic  sub-gaussian vectors.)}}
  \label{thm::ZHW}
 Let $H_1, \ldots, H_n$ be deterministic $p \times m$ matrices,
 where we assume that $m \ge p$.
Let $\Z^T_1, \ldots, \Z^T_n \in \R^{p}$ be row vectors of $\Z$.
 We generate $\Z$ according to
 Definition~\ref{def::WH}.

Then we have for $t > 0$, for any $A = (a_{ij}) \in \R^{n \times n}$,
\ben
\nonumber
\lefteqn{\prob{
  \abs{\sum_{i=1}^n  \sum_{j \not=i}^n \ip{\Z_{i}, \Z_{j}} a_{ij}} > t}}\\
\label{eq::genDH}
& \le &
2 \exp \left(- c\min\left(\frac{t^2}{(C_0 \max_{i} \twonorm{H_i})^4  p
      \fnorm{A}^2}, \frac{t}{(C_0 \max_{i} \twonorm{H_i})^2 \twonorm{A}} \right)\right) 
\een
where $\max_{i} \norm{\Z_i}_{\psi_2} \le C C_0 \max_{i} \twonorm{H_i}$
in the sense of \eqref{eq::Zpsi2}.
\end{theorem}

\noindent{\bf Remarks on covariance estimation.}
Essentially, \eqref{eq::ZZHop} matches the optimal bounds on
covariance estimation, where the mean-zero random matrix consists of
independent columns $\Z^j, j=1,\ldots, p$ that are isotropic,
sub-gaussian random vectors in $\R^n$, or columns which can be
transformed to be isotropic through a common covariance matrix.
See, for example, Theorems 4.6.1 and 4.7.1~\cite{Vers18}.
The difference between \eqref{eq::ZZHop} and such known results are:
(a) we do not assume that columns are independent; (b)  we do not
require anisotropic row vectors to share identical covariance
matrices. More generally, we allow
the (sample by sample) covariance matrix
$\Sigma_X := \E \Z \Z^T= \diag( [\tr(H_1H_1^T), \ldots, \tr(H_nH_n^T)])$,
through Definition~\ref{def::WH}; and hence we are estimating a
diagonal matrix with $p$ dependent features,
where we assume that $\E(X)$ is given.
We state the operator norm bound in Theorem~\ref{thm::YYcovcorr}, 
where it is understood that \eqref{eq::ZZHop} holds under the general 
covariance model as considered in Definition~\ref{def::WH} and
Theorem~\ref{thm::ZHW}.
We prove Theorem~\ref{thm::ZHW} in the supplementary Section~\ref{sec::opcorrelated}.
The proof might be of independent interests.
Such generalization is useful since we may consider the more general
$k$-component mixture problems, as elaborated in Section~\ref{sec::kmeans}. 
See also Exercise 6.2.7~\cite{Vers18} for a related result.

\section{Experiments}
\label{sec::experiments}
In this section, we use simulation to illustrate the effectiveness and
convergence properties of the two estimators.
We use a similar setup as the one used in ~\cite{BCFZ09}.
We generate data that is a mixture of two populations.
Data matrix $X \in \{0, 1\}^{n \times p}$ consists of
independent Bernoulli random variables, where the mean parameters
$ \E (X_{ij}) := q_{\psi(i)}^j$ for all $i \in [n]$ and $j \in [p]$,
where $\psi(i) \in \{1, 2\}$ assigns nodes $i$ to  a group $\C_1$ or
$\C_2$ for each $i \in [n]$. Let $\size{\C_1} = w_1 n$ and $\size{\C_2} = w_2 n$.
We conduct experiments for both balanced ($w_1 = w_2$) and imbalanced cases.
The entrywise expected values are chosen as follows: for half of the
$p$ features, the mean parameters $q_1^j > q_2^j$,
and for the other half, $q_1^j< q_2^j$ such that $\forall j$,
$q_1^j, q_2^j \in \{\frac{1 + \alpha}{2} + \frac{\e}{2}, \frac{1 - \alpha}{2} +
\frac{\e}{2}\}$. We set $\e = 0.1 \alpha$ and $\alpha=0.04$.  Hence $\gamma =
\alpha^2 = 0.0016$, $\frac{1}{\gamma^2}=390,625$, and $\frac{1}{\gamma}=625$.
We implement {\bf Algorithm 1:} the SDP as described in~\eqref{eq::sdpmain},
and classify according to signs of $\hat{x}$ as  prescribed by
Corollary~\ref{coro::misclass}; and  {\bf Algorithm 2:} the Peng-Wei spectral
method following~\cite{PW07}.\\
\begin{tabular}{p{5.4in}}
\hline
  {\bf Algorithm 2: Spectral method for $k$-means clustering
  (Peng-Wei)~\cite{PW07}:} \\ \hline
  
  {\bf Input:} Centered data matrix $Y \in \R^{n \times p}$, $k=2$\\
  % with $n$ rows and $p$ columns \\
  {\bf Output:} A group assignment vector $P$ \\
  
%    \item Initialize vector $P=0$, which has length of $n$ 
 %   \item Compute $Y$ by removing column means in $X$,
  \noindent{\bf Step 1.}
  Use SVD to obtain the leading eigenvector $v_1$ of  $Y Y^T$ and
      let $v := v_1$;\\
 \noindent{\bf Step 2.}
Let $S$ be the vector of sorted values of $v$ in descending
order. For each index $j$ in $[n]$, compute
the two means $\vc_{1}$, $\vc_2$, one for each of the two groups,
namely, $\C_1 = S_L := \{S_1, \ldots, S_j\}$ and $\C_2 = S_R :=
\{S_{j+1}, \ldots, S_n\}$ to the left (inclusive) and the right of this index;\\
 \noindent{\bf Step 3.}
Compute the total sum-of-squared Euclidean distances from each point within
a particular group to the respective mean, according to~\eqref{eq::SSE};
Let $t$ be the index that gives the minimum total distance, and its corresponding value
  be $S_t$; \\
  \noindent{\bf Step 4.}
  Set $P_i = 1$ if $v_i \ge S_t$, and $P_i = -1$ if $v_i < S_t$.
 %   \item Return $P$
  \\ \hline
\end{tabular}

 \vskip 15pt

{\bf Success rate and misclassification rate.}
For each experiment, we run 100 trials; and for each trial, we first generate a
data matrix $X_{n \times p}$ according to the mixture of two Bernoulli distributions with
parameters described above, and then feed $Y$~\eqref{eq::defineY} to
the two estimators for classification.  We measure success rate and
misclassification rate based on $P$, the output assignment vector.  Success rate is computed as the number of
correctly classified individuals divided by the sample size $n$.
Hence misclassification rate is $1 - $ success rate.
Each data point corresponds to the average of 100 trials.
Fig.~\ref{fig::succ-rate} shows the average success rates (over $100$ trials) as $n$
increases for different values of $p$ for the balanced case.

We observe that SDP has higher average success rate for
each setting of $(n, p)$ when $np\gamma^2 > 1.5$, despite the exhaustive search in
Algorithm 2; For $np\gamma^2 < 1.5$, the rates are closer.
We also see from the plot that when
$p < 1/\gamma = 625$, for example, when $p = 500$,
the success rate remains flat across $n$.
Note that a success rate of $50\%$ is equivalent to a total failure.
In contrast, when $n$ is smaller than $1/\gamma$, as we increase $p$,
we can always classify with a high success rate. In general,
$np\gamma^2 > 1$ is indeed necessary to obtain a success rate larger
than $60\%$, when $p \ge 1/{\gamma}$.
When $n < 625$, $np \gamma^2$ plays the role of the SNR, since $n p
\gamma^2 < p \gamma$; This remains the case throughout our experiments.

{\bf Angle and $\ell_2$ convergence.}
Here we take a closer look at the trends of $\hat{x}$ and $\hat{Z}$,
the solution to SDP~\eqref{eq::sdpmain} as $n$ increases, and of
$v_1$, the leading eigenvector of $YY^T$.
In the second experiment, we set $p \in \{20000, 50000,
80000\}$, and increase $n$.  In the left column of
Fig.~\ref{fig::angle-norm-imb}, which is for the imbalanced case of $w_1=0.7$, we plot $\theta_{\SDP} := \angle(\hat{x}, \bar{x})$ between
$\hat{x}$ and its reference vector $\bar{x}$ as defined in
Theorem~\ref{thm::SDPmain} and  Corollary~\ref{coro::misclass}.

For Algorithm 2, $\theta_1 := \angle(v_1,\bar{v}_1)$ between $v_1$ and
 its reference $\bar{v}_1$, where  $\bar{v}_1$ is as defined in
 Theorem~\ref{thm::SVD}.
In this case, the angle $\angle(\bar{v}_1, \bar{x})$ between the two
reference vectors is about 22 degrees (blue horizontal dashed line).
We observe that as $n$ increases, for both algorithms,
the angles $\theta_{\SDP}$ and $\theta_1$ decrease, but
$\theta_{\SDP}$  drops much faster and decreases to $0$ when $n > 200$
for $p = 80,000$.
We also show the angle $\phi = \angle(\hat{x}, v_1)$ between the two
leading eigenvectors $\hat{x}$ and $v_1$, which largely remains flat
across all $n$.

In the right column of Fig.~\ref{fig::angle-norm-imb},
we plot $\sin(\theta_1)$ for Algorithm 2, and for SDP,
we plot $\sin(\theta_{\SDP})$, $\shnorm{Z^{*} -
  \hat{Z}}_2/{n}$, and $\shnorm{Z^{*} - \hat{Z}}_F/{n}$,
where $Z^{*} = \bar{x} \bar{x}^T$. We see that for Algorithm 1, all three metrics 
decrease as $n$ increases, following an exponential decay in $n$ as predicted by 
our theory in Theorem~\ref{thm::exprate} and
Corollary~\ref{coro::misexp}, where in each plot, $p, \gamma$ are
being fixed. The gaps between the three curves   for SDP shrink when $p, n$ increase.
For Algorithm 2,
$\sin(\theta_1)$ also decreases as $n$ increases, but at a
slower rate of $1/n$, again as predicted by Theorem~\ref{thm::SVD} and
Corollary~\ref{coro::SVD}.

\begin{figure}[!tb]
  \centering
  \vskip-35pt
  \includegraphics[width=5.5in]{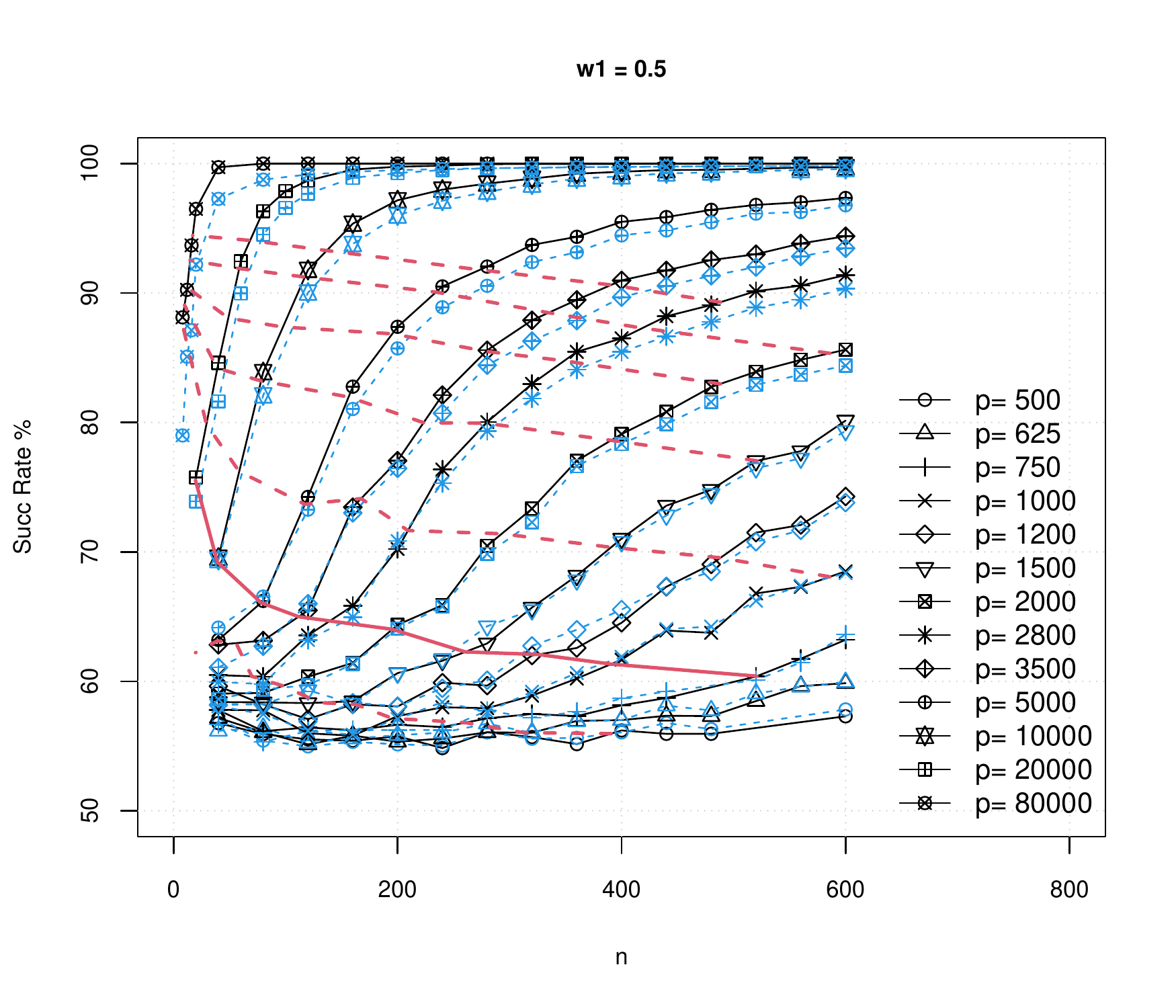}
\caption{Balanced case $w_1=0.5$. We plot the success rate for various
  values of dimension $p$ ranging from $500$ to $80000$, as $n$
  increases. Here  $\gamma = 0.0016$ and $1/\gamma = 625$.
For two lines with the same marker, the black solid line is for SDP
solution $\hat{Z}$ (and partition based on $\hat{x}$), and the blue dashed line is for
Algorithm 2. Red lines (with no markers) highlight the success rates at different
levels of $n p \gamma^2$ ranging from  $0.5$ to $3.5$,
from bottom to top with a step of $0.5$.  The solid red line is for $n p \gamma^2 = 1$.
In general, $np\gamma^2 > 1$ is necessary to see a success rate larger
than $60\%$, when $p \ge 625 = 1/{\gamma}$.}
\label{fig::succ-rate}
\end{figure}

\begin{figure}[!tb]
\centering
\vskip-10pt
\begin{subfigure}{.45\textwidth}
  \includegraphics[width=2.4in]{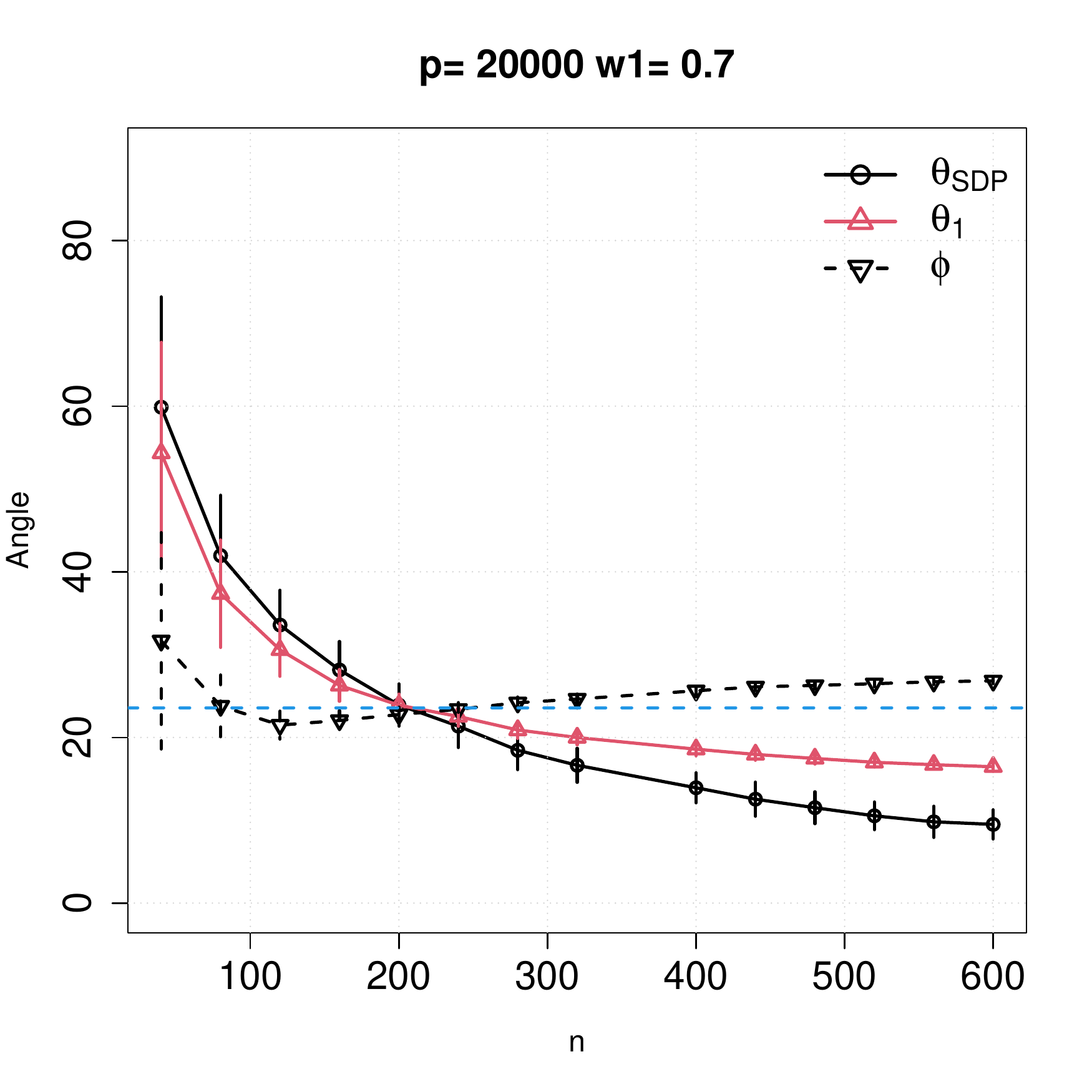}
  \end{subfigure}%
\begin{subfigure}{.45\textwidth}
  \includegraphics[width=2.4in]{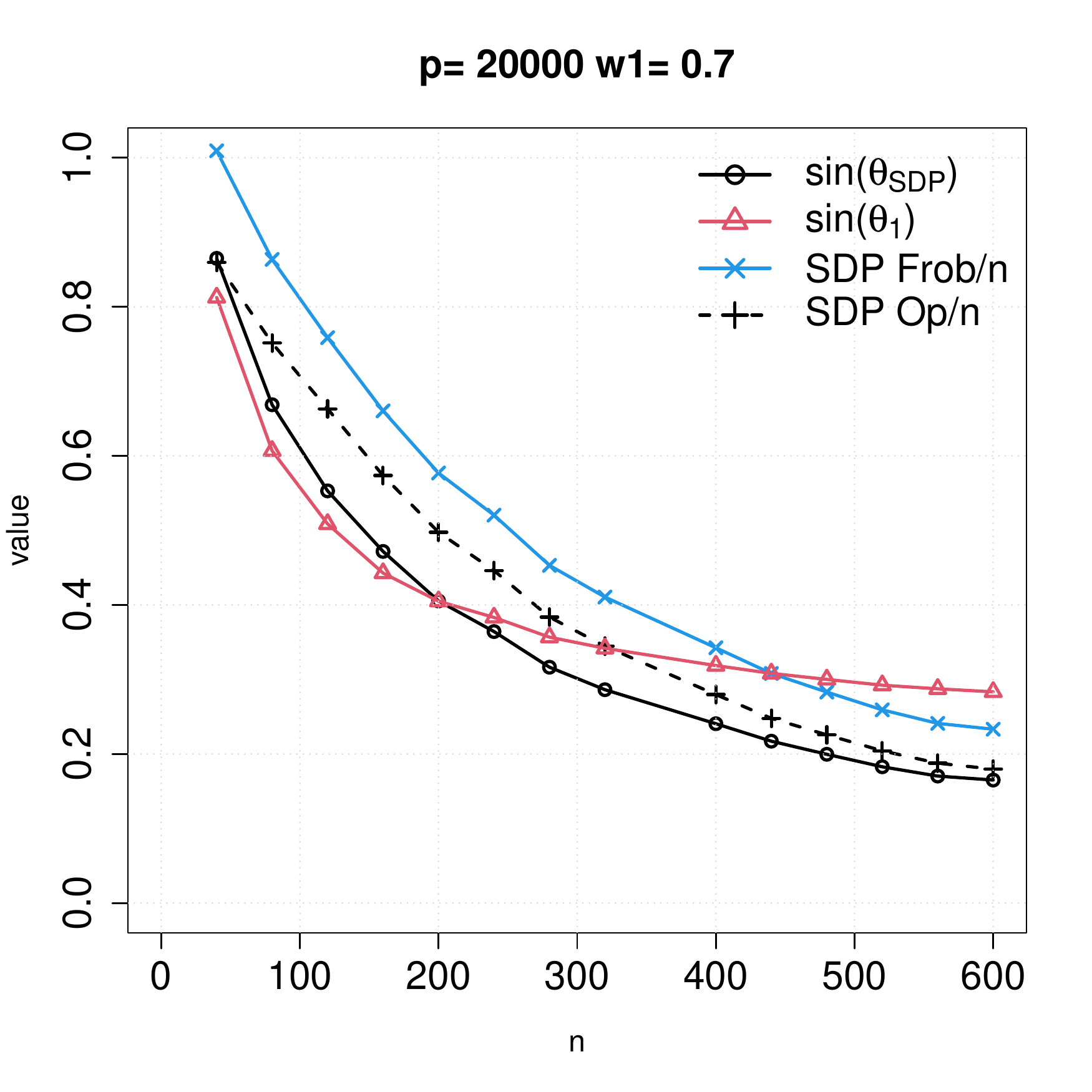}
  \end{subfigure}
  \bigskip
   \vskip-20pt
%end of first row
\begin{subfigure}{.45\textwidth}
  \includegraphics[width=2.4in]{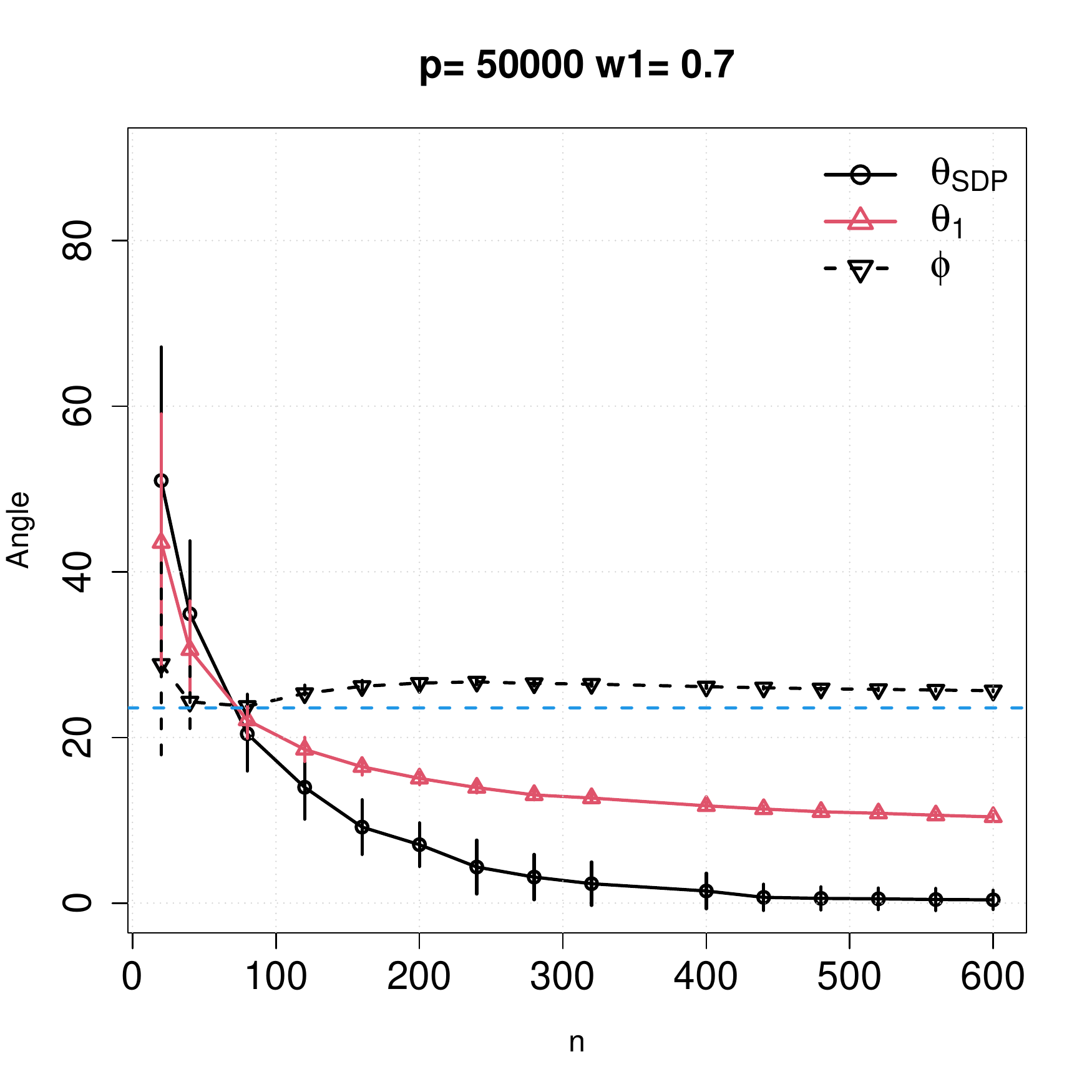}
  \end{subfigure}%
\begin{subfigure}{.45\textwidth}
  \includegraphics[width=2.4in]{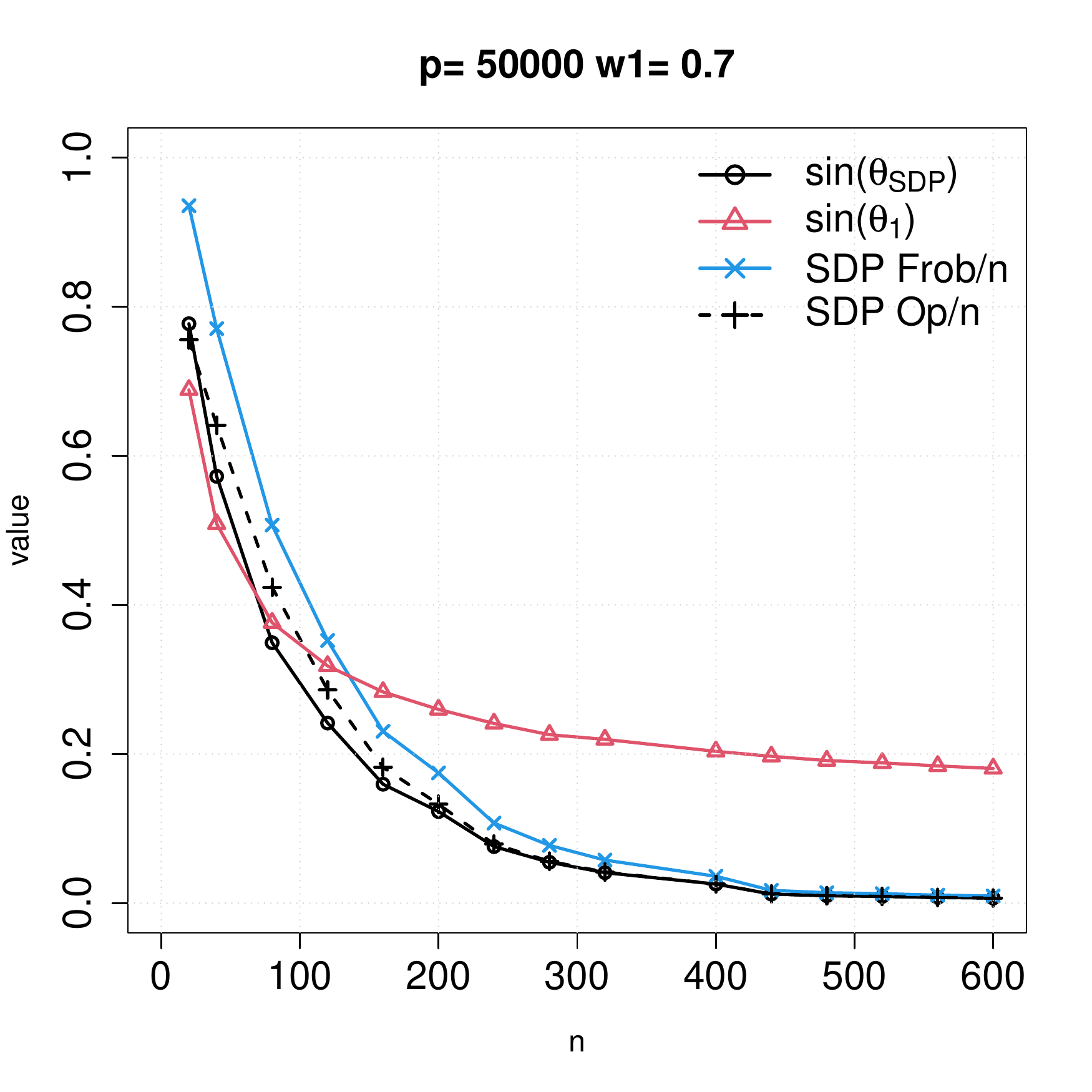}
  \end{subfigure}
\bigskip
   \vskip-20pt
\begin{subfigure}{.45\textwidth}
  \includegraphics[width=2.4in]{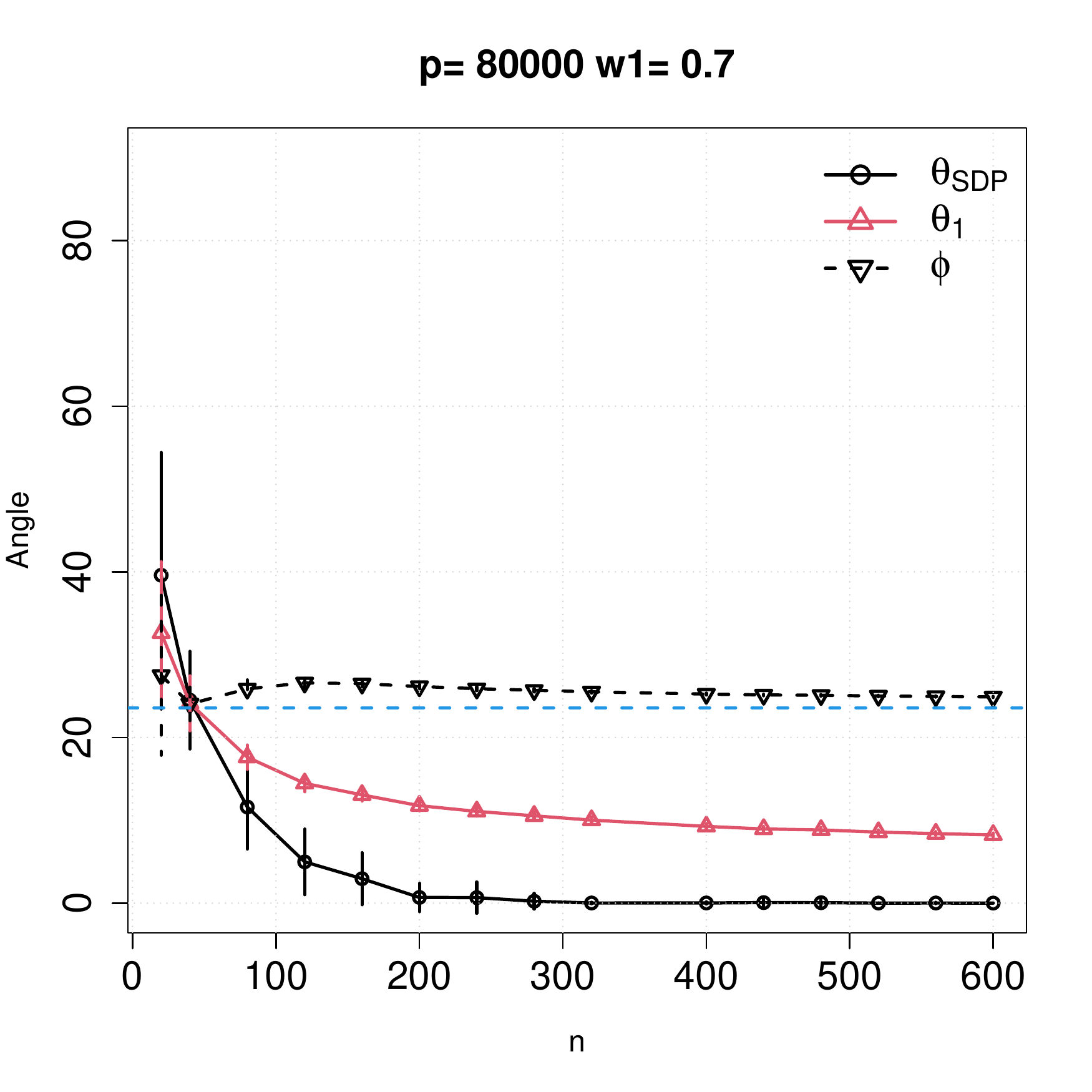}
  \end{subfigure}%
\begin{subfigure}{.45\textwidth}
  \includegraphics[width=2.4in]{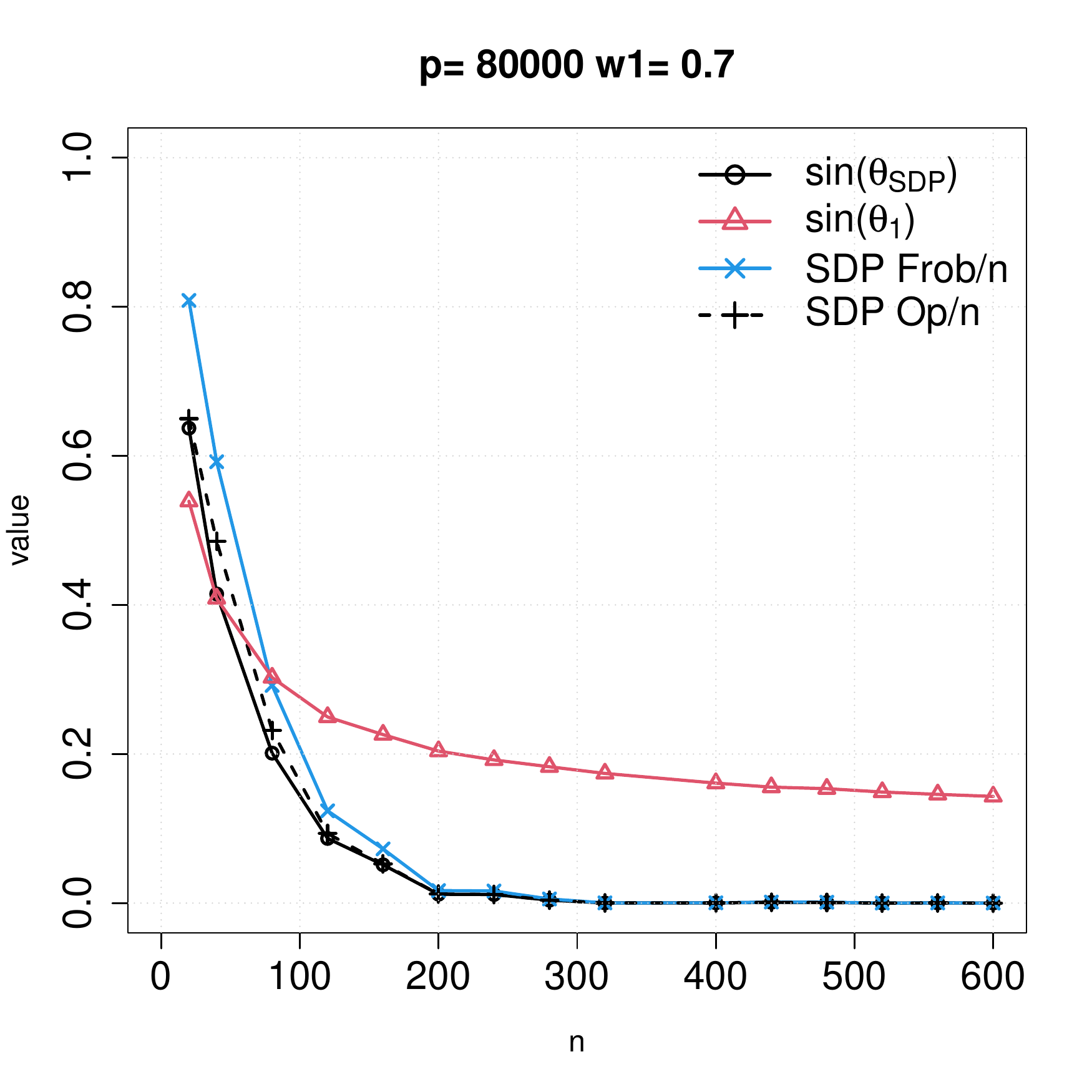}
  \end{subfigure}
\caption{Imbalanced case $w_1=0.7$,  $p \in \{20000, 50000, 80000\}$.
Left column shows the angle $\theta_{\SDP}$ (resp. $\theta_1$) between the leading
eigenvector $\hat{x}$ of SDP solution $\hat{Z}$ (resp. $v_1$ of
$YY^T$) and $\bar{x}$ (resp. $\bar{v}_1$). As $n$ increases, $\theta_{\SDP}$
decreases faster than $\theta_1$, especially for larger values of $p$.
Horizontal dashed (straight) line is the static angle
$\angle(\bar{v}_1, \bar{x})$;
The dashed curve around it is for the random $\phi = \angle(\hat{x},
v_1)$.
Each vertical bar shows one standard deviation over $100$ trials.
Right column plots $\sin(\theta_{\SDP})$, 
  ${\shnorm{Z^{*} - \hat{Z}}}_F/{n}$, $\shnorm{Z^{*} - \hat{Z}}_2/{n}$
  for SDP, and $\sin(\theta_1)$ for Algorithm 2.
}
\label{fig::angle-norm-imb}
\end{figure}

\section{Conclusion}
\label{sec::conclude}
Exploring the tradeoffs of $n$ and $p$ that are sufficient for 
classification, when sample size $n$ is small, 
is both of theoretical interests and practical value.
A recent line of work establishes approximate recovery guarantees 
of the SDPs in the low-SNR regime for sub-gaussian
mixture models; see \cite{FC18,GV19} among others.
The present work aims to further illuminate the geometric and
probabilistic features for this problem,
while allowing cluster sizes and variance profiles to vary across the two
populations. Although we use the population clustering problem as a  motivating example, 
our concentration of measure analyses in
Section~\ref{sec::finalYYaniso}, upon adaptation, will work for the general
settings~\eqref{eq::model} as well.
In particular, we study SDP relaxation as well as a simple spectral  
algorithm, which are efficiently solvable in both theoretical and 
practical senses, and provide a unified analysis of the two most commonly studied procedures  
in the literature.  By doing so, we gained new insight that the
leading eigenvectors not only contain sufficient  information for
clustering but it is also feasible to use algorithmic  techniques to
identify group memberships effectively once the SNR is bounded below by a constant.

\section*{Acknowledgement}

I would like to thank Alan Frieze for reading a crude draft of this
manuscript, and Mark Rudelson for many helpful discussions.
I thank my family for their support, especially during the pandemic.

\appendix

\section{Organization}
We prove Corollaries~\ref{coro::misclass} and~\ref{coro::misexp} in
Section~\ref{sec::proofofmisclass}.
Proofs for lemmas appearing in Section~\ref{sec::estimators}  appear in
Section~\ref{sec::proofofSDPglobal}.
We prove Theorem~\ref{thm::reading}
in  Section~\ref{sec::proofofreading}.
Proof of Theorem~\ref{thm::SVD} appears in 
Section~\ref{sec::proofofSVD}.
Proofs of Theorem~\ref{thm::YYnorm} appears in 
Section~\ref{sec::proofofYYnorm}.
Section~\ref{sec::isonorm} contains the concentration of measure analysis with 
regards to the random matrix $YY^T - \E (YY^T)$, leading to Theorem~\ref{thm::YYnorm}.
In Section~\ref{sec::proofofbias}, we prove the corresponding result for 
Lemma~\ref{lemma::EBRtilt}.
Section~\ref{sec::corrproofs} contains the concentration of measure 
analysis for anisotropic random vectors, leading to the conclusion of Theorem~\ref{thm::YYaniso}.
In Section~\ref{sec::proofofZHW}, we prove Theorem~\ref{thm::ZHW}, the
Hanson-Wright inequality for anisotropic sub-gaussian vectors, which
may be of independent interests.

\section{Proof of Corollaries~\ref{coro::misclass} and~\ref{coro::misexp}}
\label{sec::proofofmisclass}
Theorem~\ref{thm::DK} is a well-known result in perturbation theory.
See~\cite{BCFZ09} for a proof. See also Theorem 4.5.5~\cite{Vers18}
and Corollary 3 in~\cite{YWS15}.

\begin{theorem}\textnormal{\bf{(Davis-Kahan)}}
\label{thm::DK}
For $A$ and $M$ being two symmetric matrices and $E = M - A$. 
Let $\lambda_1(A) \geq \lambda_2(A) \geq \ldots \geq \lambda_n(A)$ 
be eigenvalues of $A$, with orthonormal eigenvectors 
$v_1, v_2, \ldots, v_n$ and 
let $\lambda_1(M) \geq  \lambda_2(M) \geq \ldots \geq \lambda_n(M)$ 
be eigenvalues of $M$ and $w_1, w_2, \ldots, w_n$ be the corresponding 
orthonormal eigenvectors of $M$, with $\theta_i = \angle(v_i, w_i)$. Then
\begin{gather}
\label{eq:eigen-vectors}
\theta_i \sim \sin(\theta_i) \leq 
\frac{2\twonorm{E}}{\gap(i, A)} \; \text{ where} \; \; \gap(i, A) = \min_{j \not= i}\abs{\lambda_i(A) - \lambda_j(A)}.
\end{gather}
\end{theorem}

\begin{proofof}{Corollary~\ref{coro::misclass}}
  See proof of Corollary 1.2~\cite{GV15} for the first result,
  which follows from Davis-Kahan Theorem and is a direct
  consequence of the error bound \eqref{eq::hatZFnorm}, while noting that
  the largest eigenvalue of $\bar{x} \bar{x}^T$ is $n$ while all
  others are 0, and hence the spectral gap in the sense of
  Theorem~\ref{thm::DK} equals $n$;
  In more details, we have by Theorem~\ref{thm::SDPmain} and Corollary 3~\citep{YWS15},
  \bens
\min_{\alpha = \pm 1}
\twonorm{(\alpha \hat{x} - \bar{x})/\sqrt{n}}^2
  &  \le & \frac{2^{3} \twonorm{\hat{Z} - \bar{x} \bar{x}^T}^2}{\gap(1, \bar{x} \bar{x}^T)^2} 
  \le   {2^{3} \fnorm{\hat{Z} - \bar{x} \bar{x}^T}^2}/{n^2} \\
  &  \le &   2^{3}   {4  K_G  \xi}/{w_{\min}^2}
  \eens
\end{proofof}

\begin{proofof}{Corollary~\ref{coro::misexp}}
  %On the other hand,
  The angle between $\hat{x}/\sqrt{n}$ and $\bar{x}/\sqrt{n}$ can be expressed as
  \begin{eqnarray}
  \label{eq::angsdp}
  \cos(\theta_{\SDP}) = \cos(\angle(\hat{x}, \bar{x})) 
  & = & \ip{\hat{x}, \bar{x}}/n
  \end{eqnarray}
The upper bound on $\sin(\theta_{\SDP})$ follows from 
Theorems~\ref{thm::exprate} and~\ref{thm::DK};
Moreover, by Davis-Kahan Theorem, cf. Corollary 3~\citep{YWS15}, we
have with probability at least $1-2 \exp(-c n) - 2/n^2$,
\bens
\min_{\alpha = \pm 1}
\twonorm{(\alpha \hat{x} - \bar{x})/\sqrt{n}}
  &  \le & \frac{2^{3/2} \twonorm{\hat{Z} - \bar{x} \bar{x}^T}}{\gap(1, \bar{x} \bar{x}^T)} 
  \le   {2^{3/2} \fnorm{\hat{Z} - \bar{x} \bar{x}^T}}/{n} \\
    &  \le &   2^{3/2}   (2 \onenorm{\hat{Z} - \bar{x} \bar{x}^T})^{1/2}/{n}  \le 
   4 \exp(-c_0 s^2 w_{\min}^4/2)
   \eens
   where the last inequality holds upon adjusting that constants. The corollary thus holds.
\end{proofof}

\section{Proofs for results in Section~\ref{sec::estimators}}
\label{sec::proofofSDPglobal}
Combining \eqref{eq::GDSet}, \eqref{eq::GD} and \eqref{eq::GD1}, we
have the following Fact~\ref{fact::GD}.

\begin{fact}{(Grothendieck's inequality, PSD)}
\label{fact::GD}
Every matrix $B \in \R^{n \times n}$ satisfies
\bens
\max_{Z \in \M_G^{+}} \abs{\ip{B, Z}} \le K_G \norm{B}_{\infty \to 1}.
\eens
\end{fact}

\subsection{Proof of Lemma~\ref{lemma::GVGD}}
\begin{proofof2}
  %{Lemma~\ref{lemma::GVGD}}
The upper bound in \eqref{eq::hatZsand} is trivial by definition of
$Z^{*}, \hat{Z} \in \M_{\opt}$ and uses the fact that $Z^{*} := \arg\max_{Z  \in   \M_{\opt}} \ip{R, Z}$;
The lower bound depends on Fact \eqref{fact::GD}, which implies that
\ben 
\label{eq::GD2}
\forall Z \in 
\M_{\opt}, \quad \abs{\ip{B-R, Z}} \le K_G \norm{B-R}_{\infty \to 1} =: \ve 
\een
Now to prove the lower bound in\eqref{eq::hatZsand}, we will first
replace $R$ by $B$ using \eqref{eq::GD2},
\bens
\ip{R, \hat{Z}} & \ge & \ip{B, \hat{Z}} - \ve \\
& \ge & \ip{B, Z^{*}} - \ve  \ge  \ip{R, Z^{*}} - 2 \ve
\eens
where the second inequality uses the fact that
$\hat{Z} := \arg\max_{Z  \in \M_{\opt}} \ip{B, Z}$,  and $\hat{Z}, Z^{*} \in 
\M_{\opt}$ by definition~\eqref{eq::hatZ}, while the last inequality
holds by~\eqref{eq::GD2}, since $Z^{*} \in \M_{\opt}$ and hence
\ben
\label{eq::Zstar}
\abs{\ip{B-R, Z^{*}}} & \le & K_G \norm{B-R}_{\infty \to 1} 
\een
Hence~\eqref{eq::upperZZR} holds.
Finally, we prove \eqref{eq::supZZR}; By \eqref{eq::GD2},
\eqref{eq::Zstar}, and the triangle inequality, we have
for all $Z \in \M_{\opt}$, 
\bens
\abs{\ip{B-R, Z-Z^{*}}} & \le & 2 K_G \norm{B-R}_{\infty \to 1} 
\eens
from which \eqref{eq::supZZR} follows.
\end{proofof2}

\subsection{Proof of Proposition~\ref{prop::optsol} }
\label{sec::optsolAB}
\begin{proofof2}
Recall
\bens
\M_{\opt} & := & \left\{Z :  Z \succeq 0, \diag(Z) = I_n\right\}
\subset \M^{+}_{G} \subset [-1, 1]^{n \times n}; 
\eens
Notice that for the second term in \eqref{eq::defineAintro}, we have
$\ip{(E_n -I_n), Z} = \ip{(E_n -I_n), \offd(Z)}$ in the objective
function~\eqref{eq::sdpmain}, which does not depend on $\diag(Z)$;
Hence, to maximize
\bens
\ip{A, Z}
& = & \ip{{A}, \offd(Z)} + \ip{A, \diag(Z)} \\
& = & \ip{\offd(A), \offd(Z)} + \ip{\diag(YY^T), \diag(Z)},
\eens
one must set the diagonal $Z_{jj} \in [0, 1]$ to be 1, since $\diag(YY^T)
\ge 0$. Moreover, increasing $\diag(Z)$ will only make it
easier to satisfy $Z \succeq 0$ and hence to maximize $\ip{\offd(A),
  \offd(Z)}$.
Thus, the set of optimizers $\hat{Z}$ as in~\eqref{eq::hatZintro} must
satisfy $\diag(\hat{Z})  = I_n$. Thus \eqref{eq::Aquiv} holds by
definition of~$\M_{\opt}$ as above.
Moreover,~\eqref{eq::optsolAB} holds due to the
fact that $\ip{I_n, Z} = \tr(Z) = n$ for all $Z$ in the feasible set $\M_{\opt}$.
\end{proofof2}

\subsection{Proof of Lemma~\ref{lemma::ZRnormintro} }
\begin{proofof2}
One can check that the maximizer of $\ip{R, Z}$ on the larger set
$[-1, 1]^{n \times n}$, which contains the feasible set
$\M_{\opt}$, is $Z^{*}$.
Clearly $\diag(Z^{*}) = I_n$. Since $Z^{*} = \bar{x} \bar{x}^T\succeq
0$ belongs to the smaller set
$\M_{\opt} \subset M_{G}^+$, it must be the maximizer of $\ip{R, Z}$ on that set as well. 
\end{proofof2}

\subsection{Proof of Lemma~\ref{lemma::onenorm} }
\label{sec::proofofonenorm}
\begin{proofof2}
    %{ Lemma~\ref{lemma::onenorm} }
  We will prove that  \eqref{eq::Rlower} holds for all $Z \in [-1,
  1]^n \supset \M_{\opt}$.
 Recall we have
\bens
\nonumber
 R= \E(Y) \E(Y)^T
 &=: & \twonorm{\mu^{(1)} -\mu^{(2)}}^2 
\left[
\begin{array}{cc}
  w_2^2 \vecone_{n_1} \otimes \vecone_{n_1} &- w_1 w_2 \vecone_{n_1}
\otimes \vecone_{n_2 } \\
 - w_1 w_2 \vecone_{n_2} \otimes \vecone_{N_1}  & w_1^2 \vecone_{n_2}
 \otimes \vecone_{n_2}   \end{array}
\right]  \\
 &= &
p \gamma 
\left[
\begin{array}{cc}
  w_2^2 E_{n_1} & - w_1 w_2E_{n_1 \times n_2} \\
 - w_1 w_2 E_{n_2 \times n_1} & w_1^2 E_{n_2} 
\end{array}
\right] =:                                            
p \gamma 
\left[
\begin{array}{cc} \mathbb{A} & \mathbb{B}\\
 \mathbb{B}^T & \mathbb{C}
\end{array}
\right]
\eens
Now all entries of $Z^{*}, Z$ belong to $[-1, 1]$.
Clearly, for the upper left and lower right diagonal blocks, denoted by $\D = \{\mathbb{A}, \mathbb{C}\}$, we have
$Z^{*} - Z \ge 0$, since all entries of $Z^{*}$ on these blocks are $1$s.
Similarly, for the off-diagonal blocks $\{\mathbb{B}, \mathbb{B}^T\}$,
we have
$Z-Z^{*} \ge 0$ since all entries of $Z^{*}$  on these blocks are $-1$s.
Thus we have
\bens
\inv{p \gamma}\ip{R, Z^{*} - Z}
& = &  \sum_{(i,j) \in \mathbb{A}} w_2^2 (Z^{*} - Z)_{ij} +
\sum_{(i,j) \in \mathbb{C}} w_1^2 (Z^{*} - Z)_{ij}  -
\sum_{(i,j) \in \mathbb{B}, \mathbb{B}^T} w_1 w_2 (Z^{*} - Z)_{ij}  \\
& \ge &
(w_2^2 \wedge w_1^2 \wedge w_1 w_2)
\big(
  \sum_{(i,j) \in \D} (Z^{*} - Z) + \sum_{(i,j) \in \{\mathbb{B},
    \mathbb{B}^T\}}  (Z- Z^{*})_{ij}  \big) \\
& \ge & \min_{j =1, 2} w_j^2 \onenorm{Z- Z^{*}}
\eens
where we use the fact that 
\bens
\sum_{(i,j) \in \D} (Z^{*} - Z)_{ij}  + \sum_{(i,j) \in \{\mathbb{B},
  \mathbb{B}^T\}}  (Z- Z^{*})_{ij}  
& = & \onenorm{Z- Z^{*}}
\eens
The lemma is thus proved.
\end{proofof2}

\section{Proof of Theorem~\ref{thm::reading}}
\label{sec::proofofreading}

We first state Lemma~\ref{lemma::TLbounds}.
\begin{lemma}\textnormal{(Deterministic bounds)}
  \label{lemma::TLbounds}
Let $\lambda$ and $\tau$ be as defined in~\eqref{eq::defineLT}
using matrix $Y$ as specified in Definition~\ref{def::estimators}.
By definition of $\tau$ and $\lambda$, we have
\ben
 \label{eq::lambdabound}
(n-1) \abs{\lambda - \E \lambda}
& = &  \abs{\tau - \E \tau} \le \twonorm{YY^T - \E (Y Y^T)}
\een
\end{lemma}

\begin{proof}
Now \eqref{eq::lambdabound} holds since  $2 {n \choose 2} \abs{\lambda - \E \lambda}
= \abs{n (\tau - \E \tau) } \le n \twonorm{YY^T - \E (Y Y^T)}$ where
\bens
\nonumber
\lefteqn{\abs{\sum_{i=1}^n 
  (\ip{Y_i, Y_i} -  \E \ip{Y_i, Y_i}) } =: n \abs{(\tau - \E \tau) }}\\
& \le &
\label{eq::tau2}
n \max_{i}  \abs{\ip{Y_i, Y_i} -  \E \ip{Y_i, Y_i} } \le n
\twonorm{YY^T - \E (Y Y^T)} 
\eens
\end{proof}

\begin{proofof}{Theorem~\ref{thm::reading}}
We have by Theorem~\ref{thm::YYnorm} (resp. Theorem~\ref{thm::YYaniso}),
with probability at least $1-2\exp(-cn)$,
\bens
\infonenorm{B- \E B}
&  \le & \infonenorm{ Y Y^T  - \E (Y Y^T)} +  \abs{\lambda - \E 
  \lambda} \infonenorm{E_n-I_n} \\
& \le &
\infonenorm{ Y Y^T  - \E (Y Y^T)} +  n \twonorm{Y Y^T  - \E   (Y Y^T)}
\le  \inv{3} \xi n^2 p \gamma 
\eens
where we use the fact that $\infonenorm{E_n-I_n} = n(n-1)$ and  by Lemma~\ref{lemma::TLbounds},
\bens 
(n-1) \abs{\lambda -\E \lambda} = \abs{\tau -\E \tau}
=\inv{n}\abs{\tr(YY^T - \E (YY^T))} \le \twonorm{Y Y^T  - \E   (Y Y^T)}
\eens
Now $\twonorm{E_n-I_n} \le \norm{E_n-I_n}_{\infty} = (n-1)$;
and thus similarly, 
\bens
\twonorm{B- \E B}
&  \le & \twonorm{ Y Y^T  - \E (Y Y^T)} +
\abs{\lambda - \E   \lambda} \twonorm{E_n-I_n} \\
& \le &
2\twonorm{ Y Y^T  - \E (Y Y^T)} \le  \inv{3} \xi n p \gamma
\eens
Theorem~\ref{thm::reading} then holds by the triangle 
inequality; cf.~\eqref{eq::Bdev2} and~\eqref{eq::YYop}, in view of 
Lemma~\ref{lemma::EBRtilt}.  See also Lemma~\ref{lemma::unbalancedbias}. 
\end{proofof}

\section{Proof of Theorem~\ref{thm::SVD}}
\label{sec::proofofSVD}
\begin{proofof2}
  It is known that for any real symmetric matrix, there exist a set of $n$ 
  orthonormal eigenvectors. First we state Fact~\ref{fact::fullbasis}.
  Fact~\ref{fact::fullbasis}  is also not surprising, since the 
  sum of all off-diagonal entries of $A$ is 0.

\begin{fact}
\label{fact::fullbasis}
Suppose that we observe one instance of the 
gram matrix $\hat{S}_n := X X^T$. 
Then
\ben
\label{eq::YYgram}
YY^T & = &   (I-P_1) XX^T  (I-P_1) 
\een
where
$$\ip{YY^T, E_n} = \vecone^T_n YY^T 
\vecone_n = 0.$$
Moreover, by construction, we have for $A$ as defined in~\eqref{eq::defineAintro},
\bens
\ip{A, E_n}
& = & \vecone_n^T YY^T \vecone_n - \lambda \ip{ (E_n -I_n), E_n}  = -
\lambda n (n-1) =\tr(YY^T)
\eens
where
\bens
\lambda & = & \inv{n(n-1)}\sum_{i\not= j} \ip{Y_i, Y_j} 
=-\inv{n (n-1) }\tr(YY^T) = -\frac{\tau}{n-1}
\eens 
In other words, we have $\ip{A, P_1} = \ip{A, \vecone_n \vecone_n^T/n}
= \tr(YY^T)/n =: \tau$
\end{fact}

Recall that $R$ is rank one with 
  $\lambda_{\max}(R) = \tr(R) = w_1 w_2 n p \gamma$ while $\bar{x} R \bar{x} /n =
  (4 w_1 w_2) w_1 w_2 n p \gamma \le \inv{4} n p \gamma$.
Hence $\bar{x}/\sqrt{n}$ coincides with $\bar{v}_1$ when $w_1 = w_2 = 1/2$.
Hence for $\gap(1, R)$, as defined in Theorem~\ref{thm::DK},
\bens
\gap(1, R) = \lambda_{\max}(R) = w_1 w_2 n p \gamma.
\eens
We check the claim that the leading eigenvector of $B$ coincides 
with that of $YY^T$ in Fact~\ref{fact::topeigen}.
Clearly,
\begin{eqnarray}
  \label{eq::angsvd}
  \cos(\theta_1) = \cos(\angle({v}_1, \bar{v}_1))
  & = & \ip{v_1, \bar{v}_1}
\end{eqnarray}
and hence 
\bens 
\theta_1 = \arccos (\ip{v_1, \bar{v}_1}). 
\eens
  Hence we can use the first eigenvector of $YY^T$ to partition the
  two groups of points in $\R^{p}$.   To obtain an upper bound on
  $\sin(\theta_1)$, we apply the Davis-Kahan perturbation bound as
  follows.  Since $v_1, \bar{v}_1$ are the leading eigenvectors of $B$ and $R$
  respectively, \eqref{eq::angSVD} holds by
  Theorems~\ref{thm::reading} and~\ref{thm::DK}:
  \bens
\sin(\theta_1) & := & \sin(\angle({v}_1, \bar{v}_1)) \le 
\frac{2 \twonorm{B - R}}{\lambda_{\max}(R)} = 
\frac{2 \twonorm{B - R}}{w_1 w_2 n p \gamma} \le \frac{2 \xi}{w_1 w_2}.
\eens
Moreover, we have~\eqref{eq::normSVD} holds by Corollary
3~\citep{YWS15}: since
\bens
  \min_{\alpha=\pm 1}  \twonorm{\alpha v_1 -  \bar{v}_1}^2
  & \leq & \left(\frac{2^{3/2} \twonorm{B - R}}{w_1 w_2 n p \gamma}
  \right)^2 \\
    & \leq & \frac{2^{3} \xi^2}{(w_1 w_2)^2} =: \delta', \text{ where} \; \; \delta' = {8  \xi^2}/{(w_1^2
    w_2^2)} \le c_2 \xi^2/w_{\min}^2; 
\eens
The theorem thus holds.
\end{proofof2}

It remains to state Fact~\ref{fact::topeigen}.
\begin{fact}
  \label{fact::topeigen}
Let $YY^T = \sum_{j=1}^{n-1} \lambda_j v_j v_j^T$.
  Denote by $\tilde{A} = YY^T - \lambda E_n$, then
 \ben
  \label{eq::eigenA}
\tilde{A} & := & YY^T - \lambda E_n = 
\sum_{j=1}^{n-1} \lambda_j v_j v_j^T + \frac{n \tau}{n-1} \vecone
\vecone^T/n \succeq 0;
\een
The leading eigenvector of $\tilde{A}$ (resp. $A$ and $B$) will
coincide with that of $YY^T$ with
 \ben
 \label{eq::eigenYY}
 \lambda_{\max}(\tilde{A}) & = & \lambda_{\max}(YY^T)   \ge 
n \tau/(n-1)
\een
where strict inequality holds if and only if not all eigenvalues of
$YY^T$ are identical.
Thus the symmetric matrices $A = \tilde{A}  + \lambda I_n$ and $B = A
+ \E \tau I_n$ also share the same leading eigenvector $v_1$ with
$YY^T$, so long as not all eigenvalues of $YY^T$ are identical,
with $\lambda_{\max}(A) \ge \tau$.
\end{fact}

\begin{proof}
  Clearly,  the additional terms involving $E_n$ and $I_n$ are either 
orthogonal to eigenvectors $v_1, \ldots, v_{n-1}$ of $YY^T$, or 
act as an identity map on the subspace spanned by   $\{v_1, \ldots, 
v_{n-1}\}$. 
Now   \eqref{eq::eigenA} holds since $\ip{v_j, \vecone_n}=0$ for all
$j$ and hence $\{v_1, \ldots, v_{n-1} , \vecone_n/\sqrt{n}\}$ forms
the set of orthonormal eigenvectors for $\tilde{A}$ (resp. $A$ and
$B$); and moreover, in view of Fact~\ref{fact::fullbasis},
\bens
- \lambda E_n & = & \frac{n \tau}{n-1} P_1 =\frac{n 
  \tau}{n-1} \vecone_n  \vecone_n^T/n,  \; \; \text{ where } \;\; \lambda =-\frac{\tau}{n-1}, 
\eens
Since we have at most $n-1$ non-zero eigenvalues  
and they sum up to be $\tr(YY^T)$, we have
\bens
\lambda_{\max}(YY^T) & \ge &  \tr(YY^T)/(n-1) = n \tau/(n-1)
\eens
where strict inequality holds when these eigenvalues are not all  
identical.

Finally,~\eqref{eq::eigenYY} holds since
$\lambda_1(\tilde{A}) := \lambda_{\max}(YY^T)$ in view of the
eigen-decomposition~\eqref{eq::eigenA} and the displayed equation immediately above.
Now for $A =  YY^T - \lambda (E_n -I_n)$, we have $\tr(A) =
\tr(YY^T)$, and hence $\lambda_{\max}(A) \ge \tau$.
Moreover, the extra terms $\propto I_n$ in $A$ (resp. $B$)
will not change the order of the sequence of eigenvalues for $B$ (resp. $A$) with
respect to that established for $\tilde{A}$; Hence all symmetric matrices
$B$, $\tilde{A}$, and $A$ share the same leading eigenvector $v_1$
with $YY^T$.
\end{proof}

\section{Proofs for results in Section~\ref{sec::proofreadingmain}}

Proposition~\ref{prop::decompose} holds regardless 
of the weights or the number of mixture components.
\begin{proposition}{\textnormal{\bf (Covariance projection: general
      mixture models)}}
  \label{prop::decompose}
  Let $Y = X- P_1 X$ be as defined in
  Definition~\ref{def::estimators}.  Let $\Z = X - \E  X$. We first rewrite $\hat\Sigma_Y =  (Y-\E(Y))(Y-\E(Y))^T$ as follows:
\ben 
\label{eq::SigmaY}
\hat\Sigma_Y 
& := &  (I-P_1) \Z \Z^T(I-P_1)  = M_1 -(M_2 - M_3), \\
\label{eq::defineM1}
\; \; \text{ where } \; 
M_1 & := & \hat\Sigma_X = (X-\E(X)) (X-\E(X))^T = \Z \Z^T, \; \; \text{
  and} \; \; P_1 =\inv{n} \vecone_n \vecone_n^T \\
\label{eq::defineM2}
M_2 & = &  \Z \Z^T P_1 + P_1 \Z \Z^T, \; \text {and} \;
M_3  = P_1 \Z \Z^T P_1,  \\
\label{eq::YYrelations2}
\text{ and} \; \;
\Sigma_Y & := &\E \hat\Sigma_Y :=  (I-P_1) \E (\Z \Z^T)(I-P_1) 
\een
Then we have \eqref{eq::projection}, since
\bens
\label{eq::YYrelations}
YY^T - \E (Y) \E(Y)^T 
& = &
\hat\Sigma_Y + \E(Y)(Y-\E(Y))^T + (Y-\E(Y))(\E(Y))^T.
\eens
\end{proposition}

\begin{proof}
Recall $\vecone(X) =\inv{n}\vecone_n \vecone_n^T X =: P_1 X$; Then
\ben 
\nonumber 
\hat{\Sigma}_Y := (Y - \E Y)(Y - \E Y)^T 
& = & (X-\E(X) -(\vecone{(X)} - \E \vecone{(X)} )) (X-\E(X) 
-(\vecone{(X)}- \E \vecone{(X)} ))^T \\
& = &
\nonumber
(X - P_1 X - (\E(X)- P_1 \E X))(X - P_1 X - (\E(X)- P_1 \E X))^T\\ 
& = & \big(  (I-P_1) (X -\E(X)\big) \big( (I-P_1) (X -\E(X)) \big)^T  \\
\nonumber
& = & (I-P_1) (X -\E(X)) (X -\E(X))^T (I-P_1)  \\
\nonumber 
& = & (I-P_1) \Z \Z^T (I-P_1) 
\een
The rest are obvious.
\end{proof}

\subsection{Proof of Theorem~\ref{thm::YYnorm}}
\label{sec::proofofYYnorm}
\begin{proofof2}
We use a shorthand notation for $M_Y:= \E(Y)(Y-\E(Y))^T + (Y-\E(Y))(\E(Y))^T$.
Thus we have by the triangle inequality,
~\eqref{eq::projection}, \eqref{eq::projection}
(Proposition~\ref{prop::decompose}),
Lemma~\ref{lemma::projerr}, and Theorem~\ref{thm::YYcrossterms},
with probability at least $1-2\exp(- c_6 n) - 2 \exp(-c n)$, 
\bens
\nonumber
\twonorm{YY^T - \E (Y Y^T)}
\nonumber
& \le & \twonorm{\hat\Sigma_Y -\Sigma_Y} +
\twonorm{M_Y} \le \twonorm{\Z \Z^T - \E (\Z \Z^T) } +  \twonorm{M_Y} \\
\label{eq::YYsum2}
& \le & C_2 C_0^2( \sqrt{p n} \vee n) + 2 C_3 C_0 n \sqrt{p \gamma} 
\le \inv{6} \xi n p \gamma  \\
\infonenorm{YY^T - \E (Y Y^T)}
& \le & n  \twonorm{YY^T - \E (Y Y^T)} \le   \inv{6} \xi n^2 p \gamma 
\eens
where the last inequality holds by~\eqref{eq::kilo}, while adjusting 
the constants.
\end{proofof2}

\subsection{Proof of Lemma~\ref{lemma::tiltproject} }
\label{sec::tiltproj}
First, we verify \eqref{eq::EYpre}:
\ben
\nonumber
\forall i \in \C_1\; \; \E Y_i
&= & \E X_i - (w_1 \mu^{(1)} + w_2 \mu^{(2)}) =
\mu^{(1)}(1 - w_1) - w_2 \mu^{(2)}\\
\label{eq::EY1}
&= & 
w_2 (\mu^{(1)} -  \mu^{(2)}) \\
\nonumber
\forall i \in \C_2 \; \; 
\E Y_i
& = & \E X_i - (w_1 \mu^{(1)} + w_2 \mu^{(2)}) =
\mu^{(2)} (1- w_2) - w_1 \mu^{(1)}\\
\label{eq::EY2}
& = &
w_1 (\mu^{(2)} -\mu^{(1)})
\een

\begin{lemma}
  \label{lemma::pairwise4}
Suppose all conditions in Lemma~\ref{lemma::tiltproject} hold. Then
\ben
\sup_{q \in \Sp^{n-1}} 
\sum_{i=1}^n q_i \ip{Y_i - \E (Y_i),\mu^{(1)} -\mu^{(2)}}
& \le &
\label{eq::pairwise4}
2 \sup_{q \in \Sp^{n-1}}  \sum_{i=1}^n q_i \ip{\Z_i, \mu^{(1)} -\mu^{(2)}}.
\een
\end{lemma}

\begin{proof}
\bens
\lefteqn{\sum_{i=1}^n q_i \ip{Y_i - \E (Y_i),\mu^{(1)} -\mu^{(2)}}
= 
\inv{n} \sum_{i=1}^n q_i \ip{\sum_{j\not=i} (\Z_i - \Z_j),\mu^{(1)}
  -\mu^{(2)}} }\\
& = &
\frac{n-1}{n} \left(\sum_{i=1}^n q_i \ip{\Z_i, \mu^{(1)} -\mu^{(2)}}
\right) + 
\inv{n} \sum_{i=1}^n q_i \ip{\Z_i -\sum_{j=1}^n \Z_j,\mu^{(1)}
  -\mu^{(2)}} \\
& = &
\sum_{i=1}^n q_i \ip{\Z_i, \mu^{(1)} -\mu^{(2)}}  -
\inv{n} \sum_{i=1}^n q_i \sum_{j=1}^n \ip{\Z_j,\mu^{(1)}  -\mu^{(2)} }
\eens
where 
\bens
\lefteqn{
\abs{\inv{n} \sum_{i=1}^n q_i \sum_{j=1}^n \ip{\Z_j,\mu^{(1)}  -\mu^{(2)} }}
\le 
\sup_{q \in \Sp^{n-1}}  \inv{n} \onenorm{q} \abs{\sum_{j=1}^n
  \ip{\Z_j,\mu^{(1)} -\mu^{(2)}}} }\\
& \le &
\abs{ \sum_{j=1}^n \inv{\sqrt{n}} \ip{\Z_j,\mu^{(1)}  -\mu^{(2)}}} \le
\label{eq::pairwise5}
\sup_{q \in \Sp^{n-1}}  \sum_{i=1}^n q_i \ip{\Z_i, \mu^{(1)} -\mu^{(2)}} 
\eens
Thus \eqref{eq::pairwise4} holds and the lemma is proved.
\end{proof}

\begin{proofof}{Lemma~\ref{lemma::tiltproject} }
Due to the symmetry, we need to compute only
\bens
\norm{(Y-\E(Y))\E(Y)^T} & = &
\norm{(\Z -(\vecone{(X)} - \E \vecone{(X)} )) (\E(X) - \E \vecone{(X)}
  )^T} 
\eens
First, we show that \eqref{eq::pairwise2} holds. 
Now
\bens
Y_i - \E Y_i  & =&
(X_i - \E X_i) - ((\hat{\mu}_n - \E \hat{\mu}_n) =
\Z_i- \left(\inv{n} \sum_{i=1}^n (X_i - \E X_i) \right)\\
& = & \frac{n-1}{n} \Z_i - \inv{n}\sum_{j\not=i}^n \Z_j = 
\inv{n} \sum_{j\not=i}^n (\Z_i - \Z_j)
\eens
and for $x_i \in \{-1,1\}$,
\bens
\lefteqn{
\sum_{i=1}^n x_i \ip{Y_i -\E Y_i, \mu^{(1)} -\mu^{(2)}} =
\inv{n}  \sum_{i=1}^n x_i \sum_{j\not=i}^n \ip{(\Z_i - \Z_j), \mu^{(1)}
  -\mu^{(2)}} } \\
\nonumber 
& \le &
\inv{n}  \sum_{i=1}^n \sum_{j\not=i}^n \abs{\ip{(\Z_i - \Z_j), \mu^{(1)}
    -\mu^{(2)}} } \le \frac{2(n-1)}{n} \sum_{i=1}^n \abs{\ip{\Z_i, \mu^{(1)}  -\mu^{(2)}} } 
\eens
Then we have by definition of the cut norm,
\eqref{eq::EY1}, \eqref{eq::EY2}, and \eqref{eq::pairwise2},
\bens
\lefteqn{
\infonenorm{(Y-\E (Y)) \E (Y)^T} = \sup_{x, y \in \{-1,1\}^{n}}
\sum_{i=1}^n x_i \cdot}\\
& &
\left(\sum_{j \in \C_1} y_j \ip{Y_i - \E (Y_i),  w_2 (\mu^{(1)} -\mu^{(2)})} +
  \sum_{j \in \C_2} y_j  \ip{Y_i - \E (Y_i), w_1 (\mu^{(2)} -\mu^{(1)})}\right)\\
& \le &
\sup_{x, y \in \{-1,1\}^{n}}
  \sum_{i=1}^n x_i \left(\ip{ Y_i - \E (Y_i), \mu^{(1)}
    -\mu^{(2)}} (\sum_{j \in \C_1} w_2 y_j - \sum_{j \in \C_2} 
w_1 y_j  ) \right)\\
& \le &
(w_2 \abs{\C_1} + w_1 \abs{\C_2}) \sum_{i=1}^n \abs{\ip{Y_i - \E 
    (Y_i), \mu^{(1)}  -\mu^{(2)}}} \\
& = &
2 w_2 w_1 n \max_{x \in \{-1, 1\}^n}\sum_{i=1}^n x_i \ip{Y_i - \E 
  (Y_i), \mu^{(1)}  -\mu^{(2)}}  \\
& \le &
4 w_1 w_2(n-1) \sum_{i=1}^n \abs{\ip{\Z_i, \mu^{(1)}  -\mu^{(2)}} }
\eens
Similarly, we have by \eqref{eq::EY1} and \eqref{eq::EY2},
\ben
\nonumber
\lefteqn{
  \twonorm{(Y-\E (Y)) \E (Y)^T} =
\sup_{q, h \in \Sp^{n-1}}
\sum_{i=1}^n q_i \cdot}\\
& &
\nonumber
\left(\sum_{j \in \C_1} h_j \ip{Y_i - \E (Y_i),  w_2 (\mu^{(1)} -\mu^{(2)})}
+  \sum_{j \in \C_2} h_j  \ip{Y_i - \E (Y_i), w_1 (\mu^{(2)} -\mu^{(1)})}\right)\\
& \le &
\label{eq::defQ}
\sup_{q, h \in \Sp^{n-1}} 
\left(\sum_{i=1}^n q_i \ip{ Y_i - \E (Y_i),\mu^{(1)} -\mu^{(2)}}
  (\sum_{j \in \C_1} w_2 h_j - \sum_{j \in \C_2} w_1 h_j  ) \right) =: Q
\een
where by \eqref{eq::defQ} and \eqref{eq::pairwise4}, and $w_1 w_2 \le 1/4$,
\bens
Q & \le &
\sup_{q \in \Sp^{n-1}} 
\abs{\sum_{i=1}^n q_i \ip{ Y_i - \E (Y_i),\mu^{(1)}
      -\mu^{(2)}} }\cdot  \sup_{h \in \Sp^{n-1}} 
\abs{\sum_{j \in \C_1} w_2 h_j - \sum_{j \in \C_2} w_1 h_j  } \\
& \le &
2 \sup_{q\in \Sp^{n-1}} \abs{\sum_{i=1}^n q_i \ip{\Z_i, \mu^{(1)}
    -\mu^{(2)}} } \cdot \sqrt{n w_1 w_2}  \le  \sqrt{n} \sup_{q \in \Sp^{n-1}} \abs{\ip{\sum_{i} q_i \Z_i,
    \mu^{(1)} -\mu^{(2)}}}
\eens
where for $\twonorm{h_{\C_i}} =  \sqrt{\sum_{j \in \C_i} h^2_j}, i
  =1, 2$ and $h \in \Sp^{n-1}$, 
  \bens
\abs{\sum_{j \in \C_1} w_2 h_j - \sum_{j \in 
    \C_2} w_1 h_j }
&  \le & 
w_2 \sum_{j \in \C_1} \abs{h_j} + w_1 \sum_{j \in 
  \C_2} \abs{h_j}   =:  w_2  \onenorm{h_{\C_1}} + w_1 \onenorm{h_{\C_2}} \\
&  \le & 
w_2 \sqrt{\abs{\C_1}} \twonorm{h_{\C_1}} + w_1 \sqrt{\abs{\C_2}}
\twonorm{h_{\C_2}} \\
&  \le & 
\sqrt{w_1 w_2 n}
\left(\sqrt{w_2} \twonorm{h_{\C_1}} + \sqrt{w_1}  \twonorm{h_{\C_2}} \right)
\le \sqrt{w_1 w_2 n}
\eens
where $1= w_1 + w_2 \ge 2 \sqrt{w_1 w_2}$ and by
Cauchy-Schwarz, we have for $\bar{w_0} = (\sqrt{w_2}, \sqrt{w_1})$ such
that $\twonorm{\bar{w_0}} =\sqrt{w_1 + w_2} =1$ and $z =( \twonorm{h_{\C_1}},
\twonorm{h_{\C_2}})$ such that $\twonorm{z} = 1$,
$$\ip{\bar{w_0}, z} = \left(\sqrt{w_2} \twonorm{h_{\C_1}} + \sqrt{w_1}
  \twonorm{h_{\C_2}} \right) \le \twonorm{w_0} \twonorm{z} =1.$$
\end{proofof}

\subsection{Proof of Lemma~\ref{lemma::YYdec}}
\label{sec::SigmaYYproj}
\begin{proofof2}
Reduction in the operator norm holds since by Proposition~\ref{prop::decompose},
\bens
\nonumber
\hat\Sigma_Y -\Sigma_Y
& :=& (Y - \E Y)(Y - \E Y)^T - \E ((Y - \E Y)(Y - \E Y)^T) \\
\nonumber
& = & (I-P_1) (\Z \Z^T - \E(\Z \Z^T)) (I-P_1)
\eens
and clearly,
 \ben
   \twonorm{\hat\Sigma_Y -\Sigma_Y}
   \label{eq::ZZop}
& \le & \twonorm{I-P_1} \twonorm{\Z \Z^T - \E(\Z \Z^T)} \twonorm{I-P_1} \\
   \nonumber
& \le & \twonorm{\Z \Z^T - \E(\Z \Z^T)}. 
\een
\end{proofof2}

\section{Proofs for Section~\ref{sec::YYcutnorm} on isotropic design}
\label{sec::isonorm}
  Under (A2),  the row vectors $\Z_1, \Z_2, \ldots, \Z_n \in \R^{p}$ of 
  matrix $\Z = X -\E X$, are independent,  sub-gaussian vectors with 
  sub-gaussian  norm, cf. Lemma 3.4.2~\cite{Vers18}.
  To bridge the deterministic bounds in Lemma~\ref{lemma::tiltproject}
and the probabilistic statements in Lemma~\ref{lemma::projerr}, we 
use the tail bounds in Lemma~\ref{lemma::isotropic}.
Combining Lemmas~\ref{lemma::tiltproject} and~\ref{lemma::isotropic}
proves Lemma~\ref{lemma::projerr}.

\subsection{Proof of Lemma~\ref{lemma::projerr}}
\label{sec::projectproof}
\begin{proofof2}
Let $\ve = 1/3$. Let $\Pi_n$ be an $\ve$-net of $\Sp^{n-1}$ such 
that  $\size{\Pi_n} \le (1+ 2/\ve)^{n}$;
We have by \eqref{eq::qZOEYsubg} and the union bound,
\ben
\label{eq::defineE3}
\prob{\E_3} & := &
  \prob{\abs{\sup_{ q \in \Pi_n}
      \sum_{i=1}^n q_i \ip{\Z_i, \mu }}\ge \half C_3 C_0  \sqrt{n}} \\
  \nonumber
& \le & 9^n \cdot 2 \exp\left(- c_6 {C_3^2 n}/{4}\right) \le 2\exp(-c_1 n)
\een
for some absolute constants $C_3, c_1$ and $\mu$ as in~\eqref{eq::definemu}.
Thus we have on event $\E_3^c$, by a standard approximation argument,
\bens
\label{eq::concertq}
\sup_{q \in \Sp^{n-1}} \sum_{i=1}^n q_i \ip{\Z_i, \mu}
& \le &
\inv{1-\ve} \sup_{q \in \Pi_n} \sum_{i=1}^n q_i \ip{\Z_i, \mu} \le  C_3 C_0 \sqrt{n}
\eens
Similarly, we have by  the union bound and \eqref{eq::sumZOEYsubg},
\ben
\label{eq::defineE4}
\prob{\E_4}
& := &
\prob{  \max_{ u \in \{-1, 1\}^n}
  \abs{\sum_{i=1}^n u_i \ip{\Z_i, \mu}}
  \ge \half C_4 C_0 n} \\
\nonumber
& \le &
2^n \exp\left(- c_5 \frac{(C_4 C_0)^2 n^2}{4 C_0^2 n}\right)  \le 2\exp(-c'n);
\een
Hence on $\E_4^c$, the second inequality follows
from~\eqref{eq::pairwise2}.
\nonumber
\bens
\sup_{u \in \{-1, 1\}^n} \sum_{i=1}^n u_i \ip{Y_i -\E Y_i,   \mu}
&\le &  \frac{2(n-1)}{n} \sup_{u \in \{-1, 1\}^n}
\sum_{i=1}^n u_i \ip{\Z_i, \mu} \le C_4 C_0 (n-1) 
\eens
We have by Lemma~\ref{lemma::tiltproject}, on event $\E_3^c \cap \E_4^c$,
\bens
\twonorm{M_Y}
& \le &
4 \sqrt{n}  \sqrt{w_1 w_2} \sup_{q \in S^{n-1}} \abs{\sum_{i} q_i 
  \ip{\Z_i, \mu^{(1)} -\mu^{(2)}}} \le 
\label{eq::Y2}
2 C_3 C_0 n \sqrt{p \gamma}
\eens
and
\bens
\nonumber
\infonenorm{M_Y}
& \le & 8 w_1 w_2 (n-1) \sup_{u \in \{-1, 1\}^n}
\sum_{i=1}^n u_i \ip{\Z_i, \mu^{(1) }-\mu^{(2) } } \le 
\label{eq::Y1}
C_4 C_0 n(n-1) \sqrt{p  \gamma}
\eens
Thus the lemma holds upon adjusting the constants.
\end{proofof2}

\subsection{Proof of Lemma~\ref{lemma::isotropic}}
\begin{proofof2}
  %{Lemma~\ref{lemma::isotropic}}
  Let $\mu$ be as in~\eqref{eq::definemu} and recall
\bens 
\max_i \norm{\Z_i}_{\psi_2} \le C C_0
\eens
Moreover,
by independence of $\Z_1, \ldots, \Z_n$, we have for $u =(u_1, \ldots, u_n) \in \{-1, 1\}^n$,
\bens
\norm{\sum_{i=1}^n u_i \ip{\Z_i, \mu}}^2_{\psi_2}
& \le &  C \sum_{i=1}^n \norm{\ip{\Z_i, \mu}}^2_{\psi_2}  \le  C \sum_{i=1}^n  \norm{\Z_i}^2_{\psi_2}
\eens
where $\norm{\ip{\Z_j, \mu}}_{\psi_2} \le  \norm{\Z_j}_{\psi_2}$ by
definition of~\eqref{eq::Wpsi},
and for any $q \in \Sp^{n-1}$ and $t > 0$, we have
\bens 
\norm{\sum_{i=1}^n q_i \ip{\Z_i, \mu}}^2_{\psi_2}
\le C  \sum_{i=1}^n q_i^2 \norm{\ip{\Z_i, \mu}}^2_{\psi_2} 
 \le   C \max_{i} \norm{\Z_i}^2_{\psi_2} \le  C' C_0^2
\eens
Thus we have the following sub-gaussian tail bounds,
for any $u =(u_1, \ldots, u_n) \in \{-1, 1\}^n$ and $t > 0$,
\bens
\nonumber
\prob{\sum_{i =1}^{n} u_i \ip{\Z_i, \mu}  \ge t}
& \le &
2 \exp\left(-\frac{c t^2}{ \sum_{i=1}^n 
    \norm{\Z_i}^2_{\psi_2}}\right)  \le 
2 \exp\left(-\frac{c t^2}{C_0^2 n}\right)
\eens
and for any $q \in \Sp^{n-1}$ and $t > 0$, 
\bens
\prob{\sum_{i=1}^n q_i \ip{\Z_i, \mu} \ge t }
& \le &
2 \exp\left(-\frac{c t^2}{\max_{i=1}^n \norm{\Z_i}^2_{\psi_2}}\right) 
\le 2\exp\left(- \frac{c' t^2}{C_0^2}\right)
\eens
See Proposition  2.5.2 (i)~\cite{Vers18}. Thus the lemma holds.
\end{proofof2}

\subsection{Proof sketch of Theorem~\ref{thm::YYcrossterms}}
\label{sec::mainwell}
First, notice that $M_1 = \Z \Z^T =: \hat\Sigma_X$ is the empirical covariance matrix based on 
the original data $X$. 
In order to prove the concentration of measure bounds for 
Theorem~\ref{thm::YYcrossterms}, we will first state
the operator norm bound on $M_1 -\E M_1$ in
Lemma~\ref{lemma::ZZoporig}.

Let $C_{\diag}, C_{\offd}, C_1, C_2, c, c', \ldots$ be some absolute
constants, which may change line by line. 
Denote by $\E_0$ the following event:
%the following event:
\ben 
\label{eq::defineE0}
\E_0: &&
\exists j \in [n] \quad \abs{\twonorm{\Z_j}^2 - \E \twonorm{\Z_j}^2}
%=\abs{\sum_{k=1}^K  S_{jk}}
> C_{\diag}  C_0^2   (\sqrt{n p} \vee n) 
  \een
\begin{lemma}{\bf (M1 term: operator norm)}
  \label{lemma::ZZoporig}
Choose $\ve = 1/4$ and construct an $\ve$-net $\Net$ whose size is 
upper bounded by $\abs{\Net} \le (\frac{2}{\ve}+1)^n \le 9^n$. 
Recall that $\Z = X - \E X$. Fix $\ve = 1/4$.
 Under the conditions in Theorem~\ref{thm::YYcrossterms},
 denote by $\E_8$ the following event:
 \bens
\text{ event } \E_8: \quad  \left\{\max_{q, h \in \Net}
  \sum_{i=1}^n  \sum_{j \not=i}^n \ip{\Z_{i}, \Z_{j}}  q_i h_j > C_1
  C_0^2  (\sqrt{n p} \vee n)\right\}
  \eens
 As a consequence, on event $\E_0^c \cap  \E_8^c$,  we have
\bens 
\twonorm{\Z \Z^T - \E  \Z \Z^T}
&\le&
C_2 C_0^2 (\sqrt{n p} \vee n) 
\eens
where $\prob{\E_0^c \cap  \E_8^c} \ge 1- 2\exp(-c_4 n)$,
upon adjusting the constants.
\end{lemma}

When we take $\Z$ as a sub-gaussian ensemble with independent entries, 
our bounds on $\norm{M_1 -\E M_1}$  (cut norm and operator norm) 
depend on the Bernstein's type of inequalities and higher dimensional 
Hanson-Wright inequalities.
We will state Lemma~\ref{lemma::eventE0} in
Section~\ref{sec::subexp}, where we bound the probability of event $\E_0$.
We prove Lemma~\ref{lemma::ZZoporig} in 
Section~\ref{sec::proofofZZopnorm} using the standard net argument.
Neither weights, nor the number of mixture components, will affect
such bounds.

\subsection{Bounds on independent sub-exponential random variables}
\label{sec::subexp}
We now derive the corresponding bounds using properties of
sub-exponential random variables.
The sub-exponential (or $\psi_1$) 
norm of random variable $S$, denoted by $\norm{S}_{\psi_1}$, is defined as 
\ben 
\label{eq::subexp}
&&
\norm{S}_{\psi_1} = \inf\{t > 0\; : \; \E \exp(\abs{S}/t) \le 2 \}. 
\een 
A random variable $Z$ is sub-gaussian if and only if $S :=
Z^2$ is sub-exponential with $\norm{S}_{\psi_1} =
\norm{Z}^2_{\psi_2}$. 

The proof does not depend on the specific sizes $\abs{\C_j} \forall j$ of clusters.
Lemma~\ref{lemma::bernstein} concerns the sum of independent
sub-exponential random variables.
We also state the Hanson-Wright  inequality~\cite{RV13}.
\begin{lemma}
  \label{lemma::bernstein}{\textnormal{(Bernstein's inequality, cf. Theorem 2.8.1~\cite{Vers18})}}
Let $X_1, \ldots, X_n$ be independent, mean-zero, sub-exponential
random variables. Then for every $t >0$, 
\ben
\prob{\abs{\sum_{j=1}^n X_{j} } \ge t} & \le &
\nonumber 
2 n \exp\left(-c\min\left(\frac{t^2}{\sum_{j=1}^n \norm{X_{j}}^2_{\psi_1}}, 
    \frac{t}{\max_{j} \norm{X_{j}}_{\psi_1}}\right) \right) 
\een
\end{lemma}

\begin{theorem}{\textnormal{\cite{RV13}}}
\label{thm::HW}
Let $X = (X_1, \ldots, X_m) \in \R^m$ be a random vector with independent components $X_i$ which satisfy $\E X_i = 0$ and $\norm{X_i}_{\psi_2} \leq C_0$. Let $A$ be an $m
\times m$ matrix.
Then, for every $t > 0$, 
\bens 
\prob{\abs{X^T A X - \E (X^T A X) } > t} 
\leq 
2 \exp \left(- c\min\left(\frac{t^2}{C_0^4 \fnorm{A}^2}, \frac{t}{C_0^2 \twonorm{A}} \right)\right). 
\eens 
\end{theorem}

\begin{lemma}
  \label{lemma::eventE0}
Let $\Z = (z_{jk}) \in \R^{n \times p}$ be a random matrix whose
entries are independent, mean-zero, sub-gaussian random variables with
$\max_{j, k} \norm{z_{jk}}_{\psi_2} \le C_0$. 
Then we have for $t_{\diag} =C_{\diag} C_0^2 (\sqrt{n p} \vee n)$,
\bens
 \prob{\E_0} & :=  &
\prob{\exists j  \in [n], \quad\abs{ \twonorm{\Z_j}^2 - \E \twonorm{\Z_j}^2} 
  > t_{\diag}}  \le  2 \exp(- c_0 n )
\eens
where $\{\Z_j, j \in [n]\}$ are row vectors of matrix $\Z$, and
$C_{\diag}$ and  $c_0$ are absolute constants.
\end{lemma}

\begin{proof}
Denote by 
\ben 
\label{eq::defineSJK}
S_{jk} = z_{jk}^2-  \E z_{jk}^2, \quad \text{where} \; z_{jk} = X_{jk}
-\E X_{jk}, \forall i \in [n], k \in [p]
\een
It follows from \eqref{eq::subexp} that $S_{jk}$ is a mean-zero, sub-exponential random variable since
  \bens
  \max_{j,k}
\norm{S_{jk}}_{\psi_1} & \le  &  C C_0^2 \; \; \; \text{ since} \; \;
  \forall j, k, \; \norm{z^2_{jk}}_{\psi_1} =   \norm{z_{jk}}^2_{\psi_2} \le C_0^2, \\
\text{ and } \; \norm{S_{jk}}_{\psi_1} & =  &
\norm{z^2_{jk} - \E z^2_{jk} }_{\psi_1} \le 
C \norm{z_{jk}^2}_{\psi_1} = C \norm{z_{jk}}^2_{\psi_2} \le C C_0^2
\eens
See Exercise 2.7.10~\cite{Vers18}.
Set $t_3 = C_{\diag} C_0^2 \sqrt{n p} \vee n$.
We have by Bernstein's inequality Lemma~\ref{lemma::bernstein}
and the union bound, the following large deviation bound:
\ben
\nonumber
\lefteqn{
  \prob{\exists j  \in [n], \quad\abs{ \twonorm{\Z_j}^2 - \E \twonorm{\Z_j}^2} 
    \ge t_3}  \le \sum_{j =1}^n \prob{\abs{\sum_{k=1}^{p} S_{jk} } \ge t_3}  \le }\\
&&
\nonumber
2 n \exp\left(-c\min\left(\frac{(C_{\diag} C_0^2 \sqrt{n p} \vee n)^2}{\max_{j \in [n]}\sum_{k=1}^p
      \norm{S_{jk}}^2_{\psi_1}}, \frac{C_{\diag} C_0^2 \sqrt{n p} \vee n}{\max_{j,k}
      \norm{S_{jk}}_{\psi_1}}\right) \right) \\
\label{eq::diag2exp}
& \le & 2n \exp\left(-c\min\left(\frac{C^2_{\diag}n p}{p},  C_{\diag} n
  \right) \right)  \le 
2 \exp(-c_0 n) 
\een
where 
for all ${j \in [n]}$, $\sum_{k=1}^p   \norm{S_{jk}}^2_{\psi_1} \le p   C^2 C_0^4$.
The lemma thus holds. 
\end{proof}

\subsection{Proof of Lemma~\ref{lemma::ZZoporig}}
\label{sec::proofofZZopnorm}
\begin{proofof2}
  We use Theorem~\ref{thm::HW} to bound the off-diagonal part.
  Recall $\max_{j, k} \norm{z_{jk}}_{\psi_2} \le C_0$.
  Let $\mvec{\Z} = \mvec{X - \E X}$ be formed by concatenating
  columns of  matrix $\Z$ into a long vector of size $np$.
For a particular realization of $q, h \in \Sp^{n-1}$,
  we construct a block-diagonal matrix $\tilde{A}(q,h)$, where $\diag(\tilde{A}) = 0$,
with $p$ identical block-diagonal coefficient matrices
$A_{qh}^{(k)} = \offd( q \otimes h), \forall k$ of size $n \times n$
along the diagonal.
Then
\bens 
\abs{\sum_{i=1}^n q_i \sum_{j\not=i}^n h_j \ip{\Z_i, \Z_j} }
 &:=& \abs{\mvec{\Z}^T \tilde{A}_{q, h} \mvec{\Z}}
\eens 
and
\bens
\forall q, h \in \Sp^{n-1}, \quad
\fnorm{\tilde A(q, h)}^2  & = & \sum_{k=1}^p \fnorm{A_{qh}^{(k)}}^2
= p \fnorm{\offd(q \otimes h)}^2 \le p, \\
\twonorm{\tilde{A}(q, h)} & = &
 \twonorm{\offd(q \otimes h)} \le 1, \\
\text{ since } \;\;
\twonorm{\offd(q \otimes h)} & \le &  \fnorm{\offd(q \otimes h)} \le 
\fnorm{q \otimes h}^2  = \tr(qh^T h q^T) =1.
\eens
Taking a union bound over all $\abs{\Net}^2$ pairs $q, h \in \Net$, the
$\ve$-net of $\Sp^{n-1}$,  we have by Theorem~\ref{thm::HW},
for some sufficiently large constants $C_1$ and $c > 4\ln 9$,
\bens
\lefteqn{
  \prob{\E_8} := \prob{\max_{q, h \in \Net}
 \abs{\sum_{i=1}^n \sum_{j \not= i}^n q_i  h_j \ip{\Z_{i}, \Z_{j}}}> C_1
 C_0^2 (\sqrt{p n} \vee n)} } \\
&\le &
\abs{\Net}^2 \prob{\abs{\mvec{\Z}^T \tilde{A}(q,h) \mvec{\Z}} > C_1
  C_0^2 (\sqrt{p n} \vee n)} \\
&\le &  2 \times 9^{2n} \exp \left(- c\min\left({C_1^2 p n}/{p},
  {C_1 n}  \right)\right) \\
& \le &  2\exp\left(-c n + 2n \ln 9\right) = 2 \exp\left(-c_3 n \right) 
\eens
A standard approximation argument shows that under $\E_8^c$, we have
  \bens
  \twonorm{\offd(\Z \Z^T)}
  & =  &
  \sup_{q\in \Sp^{n-1}}   \sum_{i=1}^n \sum_{j \not=i}^n 
 q_i q_j \ip{\Z_{i}, \Z_{j}}  \\
 & \le &
\inv{(1-2\ve)} \sup_{q, h\in \Net}   \sum_{i=1}^n \sum_{j \not=i}^n 
 q_i h_j \ip{\Z_{i}, \Z_{j}}  \le 2 C_1 C_0^2  (\sqrt{ np} \vee n)
 \eens
See for example Exercise 4.4.3~\cite{Vers18}.

The large deviation bound on the operator norm follows from the
triangle inequality: on event $\E_8^c \cap \E_0^c$, 
   \bens
     \twonorm{\Z \Z^T - \E(\Z \Z^T) }
  & \le &
  \twonorm{\diag(\Z \Z^T - \E(\Z \Z^T)) }   +  \twonorm{\offd(\Z \Z^T)} \le C_2 C_0^2  (\sqrt{ np} \vee n)
  \eens
  for some absolute constant $C_2$.
\end{proofof2}

\section{Bias terms}
\label{sec::proofofbias}
This section proves results needed for Theorem~\ref{thm::reading}.
We prove Lemmas~\ref{lemma::EBRtilt} in Section \ref{sec::proofofEBR},
where we also state Lemma~\ref{lemma::unbalancedbias}.
Combining~\eqref{eq::Etau},~\eqref{eq::Elamb}, and
Proposition~\ref{prop::biasfinal}, we obtain an expression on $\E
B-R$. Recall $R =\E (Y) \E(Y)^T$ is as defined in~\eqref{eq::Rtilt}.
We have the following facts about $R$.
\begin{fact}
  \label{fact::Rtrace}
  When we sum over all entries in $R$, clearly, we have for $R$ as
  defined in~\eqref{eq::Rtilt},  $\vecone^T_n R \vecone_n = 0$, 
  \ben
  \label{eq::negTR}
\vecone^T_n\offd( R) \vecone_n & = &\vecone^T_n R 
\vecone_n- \tr(R) = - \tr(R) = -n p \gamma w_2 w_1, \;\text{ where} \\
 \nonumber
\inv{p \gamma}\tr(R) & = & n w_2^2 w_1 +  n w_1^2 w_2=w_1 w_2 n \quad
\text{ and hence }  \quad \twonorm{R} = \tr(R) =w_1 w_2 n p \gamma
\een
\end{fact}

 \begin{lemma}
   \label{lemma::traceR}
We have by Fact~\ref{fact::Rtrace}, ${\tr(R)}/{n}  =  p \gamma w_1 w_2$
and hence for $n \ge 4$,
\ben
\label{eq::Rop}
\twonorm{\frac{\tr(R)}{n-1} \left( I_n-
    {E_n}/{n}\right)} & \le & \frac{n}{n-1}
o \gamma w_1 w_2 \le p \gamma /3 \\
\label{eq::Rcut}
\text{ and hence } \quad
\infonenorm{\frac{\tr(R)}{n-1} \left( I_n-{E_n}/{n}\right)} & \le  &
\frac{n}{n-1} n p \gamma w_1 w_2 \le n p \gamma /3
\een
where $\infonenorm{I_n-{E_n}/{n}} \le n
\twonorm{I_n-{E_n}/{n}} = 1$ since $I_n- {E_n}/{n}$ is a projection matrix.
\end{lemma}

\subsection{Some useful propositions}

Next, we compute the mean values in Proposition~\ref{eq::M123hub}, and 
we obtain an expression on $\E B-R$ in
Proposition~\ref{prop::biasfinal}.
Proposition~\ref{eq::M123hub} is proved in Section~\ref{sec::biasproofs}. 
\begin{proposition}{\textnormal{\bf (Covariance projection: two groups)}}
  \label{eq::M123hub}
    Let $N_j = w_j n$ for $j \in \{1, 2\}$.
    W.l.o.g.,  suppose that the first $N_1$ rows in $X$ are in $\C_1$
    and the following $N_2$ rows are in $\C_2$.
    Let $M_1, M_2, M_3$ be defined as in
    Proposition~\ref{prop::decompose}.
  Let $V_1$ and $V_2$ be the same as in~\eqref{eq::Varprofile}:
    \ben 
  \label{eq::defineV1}
  V_1 := \E \ip{\Z_j, \Z_j} \; \;  \forall j \in \C_1  \;\;   \text{ and } \; \;  
 V_2 := \E \ip{\Z_j, \Z_j} \; \;  \forall j \in \C_2 
  \een
  Let  $V_m :=  w_1 V_1 + w_2 V_2 =W_n/n$.
Let  $\hat\Sigma_Y$ be as in \eqref{eq::SigmaY}.
  Let  $W_n$ be defined as in~\eqref{eq::defineWn}:
  \ben
 \label{eq::defineWn}
  W_n & := &
     \E \sum_{i=1}^n \ip{\Z_i, \Z_i} = \sum_{i=1}^n \sum_{k=1}^p \E  z_{ik}^2.
\een
 Then
  \ben
    \label{eq::EM1unb}
    \E M_1 & =&  
\left[\begin{array}{cc} V_1 I_{N_1} & 0 \\
        0 & V_2 I_{N_2}
      \end{array}
    \right]    =: (w_1 V_1 + w_2 V_2) I_n + W_0, \\
 \label{eq::EM1unb}
 \E M_2 & =&  \frac{V_1 + V_2}{n} E_n + W_2, \\ 
\label{eq::EM3unb}
\E M_3 &:= &
\E \ip{\hat{\mu}_n - \E \hat{\mu}_n, \hat{\mu}_n - \E 
  \hat{\mu}_n}
\vecone_n \otimes \vecone_n  = \frac{W_n}{n^2} E_n,\\
    \text{ and } \; \; 
W_n/n^2 & := & (\abs{\C_1} V_1 +  \abs{\C_2} V_2)/n^2 = (w_1 V_1 +
w_2 V_2)/n =: \frac{V_w}{n},
\een
where $\tr (\E M_1) /n = V_m$ and
\ben
\label{eq::defineW2}              
 W_0   & =&
(V_1 -V_2) \left[\begin{array}{cc} w_2 I_{N_1} & 0 \\
0 & -w_1 I_{N_2} \end{array}    \right]\; \; \text {and} \;  \;
W_2  :=
\frac{V_1-V_2}{n}
\left[
\begin{array}{cc}
E_{N_1} &  0 \\
0  & - E_{N_2} 
\end{array}
\right].
\een
Now putting things together, we obtain the expression for covariance
of $Y$: 
\bens
\lefteqn{\Sigma_Y := \E (YY^T) - \E(Y) \E(Y)^T
 = 
\E M_1 + \E M_3- \E M_2 } \\
& = &
\left[\begin{array}{cc}
        V_1 I_{N_1} & 0 \\
0 &  V_2 I_{N_2} 
\end{array}
\right]  - \frac{w_2 V_1 + w_1 V_2 }{n} E_n - W_2 \; 
\eens
    which simplifies to
    \bens
    \Sigma_Y =    V (I_n - E_n/n) \; \text{ in case } \; \; V_1  = V_2 = V
    \eens
  \end{proposition}
We prove Proposition~\ref{prop::biasfinal} in
Section~\ref{sec::biasproofEBR}.
Intuitively, $W_2$ and $\mathbb{W}$ arise due to the imbalance in variance profiles. 
\begin{proposition}{\textnormal{\bf (Bias decomposition)}}
  \label{prop::biasfinal}
Let $M_1, M_2, M_3$, $W_0$, $W_2$, $V_1, V_2...$ be the same as in
Propositions~\ref{prop::decompose}
and~\ref{eq::M123hub}.
Then
\ben
\label{eq::wow}
\E B - R 
& =&  W_0 - \mathbb{W} -\frac{\tr(R)}{(n-1)} (I_n - \frac{E_n}{n})  \;
\text{ where} \; \; \\
\label{eq::wow2}
\mathbb{W}
& := & W_2 
+\frac{ (V_1 - V_2)(w_2 - w_1)}{n} E_n 
\een
where for $W_2$ defined in~\eqref{eq::defineW2}.
Moreover, when $V_1 = V_2$, $W_0 = \mathbb{W}=0$.
\end{proposition}

Next, we state the following fact about $W_2$.
\begin{fact}
  \label{fact::W2}
  Denote by
\bens 
\mathbb{W} :=
W_2 - \inv{n} \tr(W_2) I_n- \frac{\vecone_n^T \offd(W_2) \vecone_n}{n(n-1)} (E_n -I_n) 
\eens
Then $\mathbb{W}$ coincides with \eqref{eq::wow2}.
Moreover, we have
\bens
\mathbb{W} 
& =&  W_2 -\frac{ (V_2 - V_1)(w_2 - w_1)}{n} E_n \\
& := &
\frac{V_1-V_2}{n}
\left[
\begin{array}{cc}
2 w_2   E_{N_1}  &  (w_2 - w_1) E_{N_1 \times N_2} \\
(w_2 - w_1)  E_{N_1 \times N_2}   & - 2 w_1 E_{N_2} 
\end{array}
\right]
\eens
Moreover, we have
\ben
\label{eq::W3cutnorm}
\twonorm{\mathbb{W}}
& \le& \abs{V_1 - V_2} (w_1 \vee w_2)  \; \text{ and } \;\;
\infonenorm{\mathbb{W}} \le  n \abs{V_1 - V_2} (w_1 \vee w_2) 
\een
\end{fact}

\begin{proof}
  By definition of $W_2$ as in~\eqref{eq::defineW2}, we have
\bens
\tr(W_2) & = &
\frac{V_2-V_1}{n}(w_2 n-w_1 n)= (V_2 - V_1)(w_2 - w_1) \;  \text{ and }  \\
\vecone_n^T W_2 \vecone_n 
& = &
\frac{V_2-V_1}{n} (w^2_2 n^2 - w_1^2 n^2)
=  n(V_2 - V_1)(w^2_2 - w^2_1) \\
& = & n(V_2 - V_1)(w_2 - w_1) 
\eens
Thus
\bens 
\frac{\vecone_n^T \offd(W_2) \vecone_n}
{n(n-1)} & = & \frac{\vecone_n^T W_2 \vecone_n -\tr(W_2)}{n(n-1)} =
\frac{(V_2 - V_1)(w_2 - w_1)}{n};
\eens
Then
\bens
\mathbb{W} 
& = &
W_2 - \inv{n} \tr(W_2) I_n- \frac{\vecone_n^T \offd(W_2) 
  \vecone_n}{n(n-1)} (E_n -I_n) \\
& = &
W_2 -\frac{(V_2 - V_1)(w_2 - w_1)}{n} E_n.
\eens
Hence~\eqref{eq::wow2} holds.
Moreover, by symmetry, $\twonorm{\mathbb{W}} \le \norm{\mathbb{W}}_{\infty}   \le 
\abs{V_1 - V_2} (w_1 \vee w_2).$
\end{proof}

\subsection{Proof of Lemma~\ref{lemma::EBRtilt}}
\label{sec::proofofEBR}
\begin{proofof2}
We have by  Proposition~\ref{prop::biasfinal},
\ben
\nonumber                                   
\norm{\E B - R}
  & =&  \norm{W_0 - \mathbb{W} -\frac{\tr(R)}{(n-1)} (I_n - \frac{E_n}{n})}
  \\
 \label{eq::normEBR}
& \le &
  \norm{W_0}                  
 +\norm{ \mathbb{W}} +\norm{\frac{\tr(R)}{(n-1)} (I_n - \frac{E_n}{n}) }
 \een
 where the $\norm{\cdot}$ is understood to be either the operator or
 the cut norm.
Recall that 
\bens 
W_0 & := & \E M_1 - V_w I_n = (V_1 - V_2)\left[\begin{array}{cc} w_2 I_{w_1 n} & 0 \\
0 & - w_1  I_{w_2 n}   \end{array}    \right] \\
 \text{ and hence} \; \; 
  \vecone^T_n W_0 \vecone_n & = & \tr(W_0) = 0 \; \; \text{and} \; \; 
  V_m := \tr (\E M_1)/n = w_1 V_1 + w_2 V_2 
\eens 
and  hence $W_0$ disappears if the two clusters have identical sum of
variances: $V_1 = V_2$.
Clearly, 
\ben 
\label{eq::W0opnorm}
\twonorm{W_0}  & \le & \abs{V_1 -  V_2}  (w_2 \vee w_1) 
\een 
Combining \eqref{eq::normEBR},~\eqref{eq::W0opnorm},~\eqref{eq::W3cutnorm},
and~\eqref{eq::Rcut}, we have for
$w_{\min} := w_1 \wedge w_2$ and $ n  \xi \ge \inv{2 w_{\min}}$
\bens
\twonorm{\E B - R} & \le  &
\twonorm{W_0} + \twonorm{\mathbb{W}} + 
\twonorm{\frac{\tr(R)}{(n-1)} (I_n - \frac{E_n}{n})} \\
& \le & 2 \abs{V_1 -  V_2} (w_1 \vee w_2) + p \gamma/3 \\
& = &
\frac{2}{3} \xi n p \gamma (1- w_{\min}) + p \gamma/3 
\le \frac{2}{3} \xi n p \gamma 
\eens
where we use the fact that
\bens
\frac{2}{3} \xi p \gamma n w_{\min} \ge  p \gamma/3  \; \; \text{ since
} \; \; 2 \xi n w_{\min} \ge  1
\eens
Now the bound on the cut norm follows since 
\bens 
\infonenorm{\E B - R} & \le  & n \twonorm{\E B - R}   \le \frac{2}{3} \xi n^2 p \gamma 
\eens 
The lemma thus holds for the general setting; when $V_1 = V_2$, we
show the improved bounds in Lemma~\ref{lemma::unbalancedbias}.
\end{proofof2}

Corollary~\ref{coro::W2bound} follows from the proof of
Lemma~\ref{lemma::EBRtilt}, which we state to prove a bound for the
balanced cases. The proof is given in Section~\ref{sec::W2coro}.
\begin{corollary}
 \label{coro::W2bound}
For general cases, we have by definition,
\bens
  W_2 & := &
\frac{V_1-V_2}{n}
\left[
\begin{array}{cc}
E_{N_1} &  0 E_{N_1 \times N_2} \\
0  E_{N_1 \times N_2}   & - E_{N_2} 
\end{array}
\right]
\eens
Moreover we have the following term which depends on the weights,
\bens
    \infonenorm{\frac{ (V_2 - V_1)(w_2 - w_1)}{n} E_n} & \le  &
    n \abs{(V_2 - V_1)(w_2 - w_1) } \le 
    \xi n^2 p \gamma \abs{w_2 - w_1}  \\
    \infonenorm{W_2} & := & n \abs{V_1 - V_2} (w_1^2 + w_2^2) \\
  \text{ and hence } \; \;
  \infonenorm{W_0 -W_2}&  \le & n \twonorm{W_0-W_2}
  <  n \abs{V_1 -  V_2} \le  \xi p n^2 \gamma
\eens
\end{corollary}

\begin{lemma}{\textnormal{\bf (Reductions)}}
\label{lemma::unbalancedbias}
    Let $W_0, W_2, V_1, V_2$ be the same as in Proposition~\ref{eq::M123hub}.
 Recall that $R =\E(Y) \E(Y)^T$.
\noindent{\bf When $V_1 = V_2$}, we have
\bens
\E B - R  = - \frac{\tr(R)}{(n-1)} (I_n -    \frac{E_n}{n}) =
- p \gamma w_2 w_1 \frac{n}{n-1}(I_n -    \frac{E_n}{n})
\eens
and hence for $n \ge 4$,
\bens
\infonenorm{ \E B - R} =
\infonenorm{p \gamma w_2 w_1 \frac{n}{n-1}(I_n - \frac{E_n}{n})} \le
n p \gamma/3
\eens
\noindent{\bf For balanced clusters, that is, when $w_1 = w_2$},
we have
\bens
\E B - R
& =&  W_0 - W_2 -\frac{\tr(R)}{(n-1)} (I_n - \frac{E_n}{n})
\eens
and hence for $n \ge 4$,
\bens
\twonorm{ \E B - R}
& \le &
\frac{p \gamma}{3} + \abs{V_1 - V_2} \le \inv{3} (1 + o(1)) \xi p n \gamma  \\
\infonenorm{ \E B - R}
& \le &
\frac{n p \gamma}{3} + \abs{V_1 - V_2} n \le \inv{3} (1 + o(1)) \xi p n^2 \gamma.
\eens
\end{lemma}

\begin{proof}
Recall
\bens
W_0 -W_2 
& =&
(V_1 - V_2) \left[\begin{array}{cc} w_2 I_{N_1} -   E_{N_1} /n & 0 \\
 0 & - (w_1 I_{N_2}  - E_{N_2} /n) \end{array}    \right]
\eens
The case where $V_1 = V_2$ follows from
Lemma~\ref{lemma::traceR} and~\eqref{eq::wow}.
We now show the balanced case where $w_1 = w_2$.
Under the conditions of Lemma~\ref{lemma::EBRtilt}, we have  for $\xi
= \Omega(1/n)$, by \eqref{eq::Rop} and~\eqref{eq::wow},
\bens 
\twonorm{\E B - R} & \le  &
\twonorm{W_0-W_2} + \twonorm{\frac{(V_1 - V_2)(w_1 - w_2)}{n} E_n} +
\twonorm{\frac{\tr(R)}{n-1} \left( I_n-
    {E_n}/{n}\right)} \\
& \le  &
\abs{V_1 - V_2} + p \gamma/3 \le \inv{3} \xi n p \gamma + p \gamma/3 
\eens
Similarly, we obtain
\bens
\infonenorm{ \E B - R}
& \le  &
\abs{V_1 - V_2} n + n p \gamma/3
\le   \inv{3} \xi n^2 p \gamma + n p \gamma/3 
\eens
The proof follows from Corollary~\ref{coro::W2bound} immediately.
\end{proof}

\subsection{Proof of Proposition~\ref{eq::M123hub}}
\label{sec::biasproofs}
\begin{proofof2}
  %{Proposition~\ref{eq::M123hub}}
  Denote by  $\tr (\E M_1) /n = (w_1 V_1 + w_2 V_2) = V_m$.
\bens
\nonumber
\E M_1 & =& 
\E  \left((X-\E(X)) (X-\E(X))^T \right)  = \E \Z \Z^T \\
& =& 
\left[\begin{array}{cc} V_1 I_{w_1 n} & 0 \\
        0 & V_2 I_{w_2 n}
      \end{array}    \right] 
%& =& V_m I_n +  \left[\begin{array}{cc} V_1 - V_m  I_{w_1 n} & 0 \\
%        0 & V_2 -V_m +  I_{w_2 n}
%    \end{array}
%    \right]
\eens
Moreover, upon subtracting the component of $V_w I_n = \inv{n}\tr (\E M_1) 
I_n$ from $\E M_1$, we have $W_0$:
\bens
\E M_1   - V_w I_n  & := & \E M_1   - \inv{n}\tr (\E M_1)  I_n
=
    (V_1 -V_2) \left[\begin{array}{cc} w_2 I_{w_1 n} & 0 \\
        0 & - w_1 I_{w_2 n} 
            \end{array}    \right]        
          =:   W_0;
\eens
Next we evaluate $\E M_2$: for $\Z = X -\E X$
\bens
\nonumber
\E M_2 & = & \E \left(\left(
    \Z (\vecone{(X)} - \E \vecone{(X)} )^T +
    (\vecone{(X)}- \E \vecone{(X)})\Z^T\right)\right) \\
& = &
\label{eq::EM2unb}
\frac{1}{n}
\left[
\begin{array}{cc}
2 V_1 E_{N_1} &
  (V_1 + V_2) E_{N_1 \times N_2} \\
  (V_1 + V_2) E_{N_1 \times N_2}  &
2 V_2 E_{N_2}
\end{array}
\right]  \\
\text{and hence} \quad
\E M_2 - \frac{ (V_1 + V_2)}{n} E_n
& = &
\nonumber
\frac{V_1 - V_2}{n}
\left[
\begin{array}{cc}
E_{N_1} &
 0 E_{N_1 \times N_2} \\
 0 E_{N_1 \times N_2}  &
- E_{N_2}
\end{array}
\right]   =: W_2
\eens
For $W_n$ as defined in  \eqref{eq::defineWn}, we have
\bens
\E \ip{\hat{\mu}_n - \E \hat{\mu}_n, \hat{\mu}_n - \E
  \hat{\mu}_n} &:= & \inv{n^2}\sum_{i=1}^n \sum_k \E z_{ik}^2 =
\frac{w_1V_1 +  w_2 V_2}{n} = \frac{V_m}{n} \\
\E M_3 &:= & \E \ip{\hat{\mu}_n - \E \hat{\mu}_n, \hat{\mu}_n - \E
  \hat{\mu}_n} E_n = \frac{\abs{\C_1} V_1 +  \abs{\C_2} V_2}{n^2} E_n
= \frac{W_n}{n^2} E_n
\eens
Now putting things together,
\bens
\lefteqn{\E (YY^T) - \E(Y) \E(Y)^T
 = 
\E M_1 + \E M_3- \E M_2 } \\
& = &
\left[\begin{array}{cc}
        V_1 I_{N_1} & 0 \\
0 &  V_2 I_{N_2} 
\end{array}
\right] + \frac{V_m}{n} E_n - \frac{ (V_1 + V_2)}{n} E_n - W_2 \; \; \; \text{ where} \\
\frac{V_m}{n}  - \frac{ (V_1 + V_2)}{n} & = &
-\frac{V_1(1 -w_1) + V_2(1-w_2)}{n} =
-\frac{V_1 w_2+ V_2 w_1}{n} 
\eens
The proposition thus holds.
\end{proofof2}

\subsection{Proof of Proposition~\ref{prop::biasfinal}}
\label{sec::biasproofEBR}
\begin{proofof2}
  %{Proposition~\ref{prop::biasfinal}}
First, we have by Proposition~\ref{prop::decompose}, and $R= \E(Y) \E(Y)^T$, 
\ben 
\label{eq::YYrelations3}
 \E (YY^T)  = \Sigma_Y  + R =  \E M_1  - \E M_2  + \E M_3 + R 
\een 
We have by Fact~\ref{fact::Rtrace} and~\eqref{eq::YYrelations3}, 
\ben 
\nonumber 
\E \tau & = &
\inv{n}\sum_{i=1}^n \E \ip{Y_i, Y_i}
= \E \tr(YY^T)/n =\tr({\Sigma}_Y)/n + \tr(R)/n\\
\label{eq::Etau}
& = &
\inv{n} \left(\tr(\E M_1 + \E M_3)- \tr(\E M_2) \right)+ \tr(R)/n  \\
\nonumber 
\E \lambda & = & \inv{n(n-1)}\sum_{i\not=j}^n \E \ip{Y_i, Y_j} =
 \inv{n(n-1)} \vecone_n^T \E(\offd(YY^T)) \vecone_n \\
& = &
\label{eq::Elamb}
\inv{n(n-1)} \vecone_n^T \offd(\E M_3  -\E M_2) \vecone_n -
\frac{\tr(R)}{n(n-1)} 
\een
where in \eqref{eq::Elamb} we use the fact 
that $\offd(\E M_1) =0$ by~\eqref{eq::EM1unb} and \eqref{eq::negTR}.
Hence by definition of $B$ and $R$, we have
\ben
\E B - R & =  &
\nonumber
\E M_1 - \E M_2 + \E M_3 - \E \tau I_n  -\E \lambda 
(E_n - I_n) \\
\nonumber
& =  &
\E M_1 - \E M_2 + \E M_3 - \left(\inv{n}\tr(\E M_1 + \E M_3) -
  \inv{n} \tr(\E M_2) +  \frac{\tr(R)}{n}\right)  I_n \\
\label{eq::EBRfinal}
&  & -\left(\inv{n(n-1)} \vecone_n^T \offd(\E M_1 + \E M_3  -\E M_2)
  \vecone_n - \frac{\tr(R)}{n(n-1)}\right)(E_n -I_n)
\een
  \bit
  \item
Notice that $\E M_3 = \frac{V_w}{n} E_n$ and hence its contribution
to $\E \tau$ and $\E \lambda$ is the same;
Thus we have
\ben
\nonumber
\tr(\E M_3) & =  & \inv{(n-1)} \vecone_n^T \offd(\E M_3) 
  \vecone_n =\frac{W_n}{n} = V_m \; \text{   and
    by~\eqref{eq::EM3unb}}, \\
  \label{eq::EM30}
&& \E M_3 - \inv{n} \tr(\E M_3) I_n 
- \frac{\vecone_n^T \offd(\E 
  M_3) \vecone_n}{n(n-1)}(E_n - I_n) = 0; 
\een
\item
$\E M_1$ is a diagonal matrix and hence $\offd(\E M_1) =0$.
Now we have by~\eqref{eq::EM1unb},
\ben
\label{eq::M1fix}
&& \E M_1 -  \inv{n} \tr(\E M_1) I_n = W_0 \text{ and } \offd(\E M_1) 
= 0 
\een
\item
For $\E M_2$, we decompose it into one component proportional to 
$E_n$: $\bar{M}_2 :=\frac{V_1 + V_2}{n} E_n$ and another component 
$W_2 = \E M_2 -\bar{M}_2$. 
By Proposition~\ref{eq::M123hub}, we have 
\ben
\nonumber
W_2 &= & \E M_2 -\bar{M}_2 =\E M_2 - \frac{V_1 + V_2}{n} E_n \\
 & := &
\frac{V_1-V_2}{n}
\left[
\begin{array}{cc}
E_{N_1} &  0 E_{N_1 \times N_2} \\
0  E_{N_1 \times N_2}   & - E_{N_2} 
\end{array}
\right] \\
\label{eq::BarM2}
\text{ where by definition } &&
\bar{M}_2  - \frac{\tr(\bar{M}_2)}{n} I_n - \frac{1_n^T 
  \offd(\bar{M}_2) 1_n}{n(n-1)} (E_n- I_n) =0
\een
\eit
Thus we have by Fact~\ref{fact::W2} and \eqref{eq::BarM2}
\ben
\nonumber
\lefteqn{
  \E M_2 - \inv{n} \tr(\E M_2) I_n- \frac{\vecone_n^T \offd(\E 
  M_2) \vecone_n}{n(n-1)} (E_n -I_n) }\\
& = &
\label{eq::M2red}
W_2 - \inv{n} \tr(W_2) I_n-
\frac{\vecone_n^T \offd(W_2) \vecone_n}{n(n-1)} (E_n -I_n)  =:
\mathbb{W}
\een
Now by \eqref{eq::Etau},~\eqref{eq::Elamb},
\eqref{eq::EBRfinal},~\eqref{eq::EM30},~\eqref{eq::M1fix},~\eqref{eq::M2red},
and Proposition~\ref{eq::M123hub}, 
\ben
\nonumber
\E B - R & =  &
\E M_1 - \E M_2 + \E M_3 - \E \tau I_n  -\E \lambda 
(E_n - I_n) \\
\label{eq::prelude}
& =:  &
W_0 - \mathbb{W} -  \frac{\tr(R)}{n-1} (I_n -E_n/n)
\een
where in step 2, we simplify all terms involving $\E M_1$ and $\E M_2$, and 
eliminate all terms involving $\E M_3$.
\end{proofof2}

\subsection{Proof of Corollary~\ref{coro::W2bound} }
\label{sec::W2coro}
\begin{proofof2}
Now
\bens
W_0 - W_2   & =&
(V_1 -V_2) \left[\begin{array}{cc} w_2 I_{N_1} -E_{N_1}/n& 0 \\
        0 & -(w_1 I_{N_2} -E_{N_2} /n) \end{array} \right].
    \eens
Moreover, due to symmetry, for $w_j > 1/n$,
\bens
\twonorm{W_0-W_2}
& \le &
\norm{W_0-W_2}_{\infty} :=\max_{i} \sum_{j=1}^n \abs{W_{0,ij}
  -W_{2,ij}} \\
& \le &  \abs{V_1 - V_2}
((w_2 -1/n) + (w_1 n-1)/n) \vee( (w_1 -1/n)+ (w_2 n-1) /n)
< \abs{V_1 - V_2}
\eens
where 
\bens 
\lefteqn{((w_2 -1/n) + (w_1 n-1)/n) \vee ( (w_1 -1/n)+ (w_2 n-1) 
  /n)}\\
&  =&
((w_2 -1/n) + (w_1-1/n) \vee( (w_1 -1/n)+ (w_2 -1/n) = 1 -2/n 
\eens
Thus we have by the triangle inequality,
\bens
\infonenorm{W_0 -W_2}
&  \le&
\infonenorm{W_2} + \infonenorm{W_0} \\
& \le & \abs{V_1 - V_2}n  \big(\abs{w_2 w_1 + w_1 w_2} + (w_1^2 +
w_2^2)\big)\\
&  = & \abs{V_1 - V_2}n  \le \half \xi n^2 p \gamma 
\eens
\end{proofof2}

\section{Proofs for Section~\ref{sec::finalYYaniso}}
\label{sec::corrproofs}

\begin{proofof}{Theorem~\ref{thm::YYaniso}}
  The proof of Theorem~\ref{thm::YYaniso}
  follows that of  Theorem~\ref{thm::YYnorm} in
  Section~\ref{sec::proofofYYnorm},
  in view of  Theorem~\ref{thm::YYcovcorr} and Lemma~\ref{lemma::projWH}.
Finally, the probability statements hold by adjusting the constants.
\end{proofof}

\subsection{Preliminary results}
Lemma~\ref{lemma::anisoproj} follows from the sub-gaussian tail bound.
We prove Lemma~\ref{lemma::projWH} in Section~\ref{sec::proofprojWH}.
\begin{lemma}
  \textnormal{\bf (Projection for anisotropic sub-gaussian random
    vectors).}
  \label{lemma::anisoproj}
 Suppose all conditions in Lemma~\ref{lemma::projWH} hold. Let $\mu$
 be as defined in~\eqref{eq::definemu}.
 Then
\ben
\label{eq::Zcovdef}      
\norm{\ip{\Z_j, \mu}}_{\psi_2}
& \le & C_0 \norm{\ip{\Z_j, \mu}}_{L_2} :=  C_0 \sqrt{\mu^T \cov(\Z_j)
  \mu} \\
\nonumber
\text{where } \; \sqrt{\mu^T \cov(\Z_j)  \mu}
& = &  \sqrt{\mu^T H_i H_i^T \mu} = \twonorm{R_i \mu} \; \; \text{ for
  each} \; j \in \C_i, i =1, 2
\een 
Thus for any $t > 0$, for some absolute constants $c, c'$,
we have for each $q \in \Sp^{n-1}$ and $u =(u_1,\ldots, u_n) \in \{-1,
1\}^n$ the following tail bounds:
\ben
\label{eq::proanis}
\prob{\abs{\sum_{i=1}^n q_i \ip{\Z_i, \mu}} \ge t} & \le &
2 \exp\left(- \frac{c t^2}{(C_0 \max_{i} \twonorm{R_i \mu})^2}\right),
\; 
\text{ and} \; \\
\prob{\abs{\sum_{i =1}^{n} u_i \ip{\Z_i, \mu }} \ge t}
\label{eq::proanisum}
& \le &
2 \exp\left(-\frac{c' t^2}{n (C_0 \max_{i} \twonorm{R_i \mu})^2}\right)
\een
\end{lemma}

We prove Lemma~\ref{lemma::E0correlated} in Section~\ref{sec::proofE0corre}.
As a special case, we recover results in Lemma~\ref{lemma::eventE0}.
\begin{lemma}
  \label{lemma::E0correlated}
  Let $W_j$ be a mean-zero, unit variance, sub-gaussian random vector with
  independent entries, with $\max_{j, k} \norm{w_{jk}}_{\psi_2} \le  C_0$.
Let $\{\Z_j, j \in [n]\}$ be row vectors of $\Z$,  where $\Z_j = H_i
W_j$ for $j \in \C_i, i=1, 2$. Then for each $j \in \C_i, i=1, 2$, we have 
\bens
\prob{\abs{\twonorm{\Z_j}^2 - \E \twonorm{\Z_j}^2} > t}
& = &
\prob{\abs{\twonorm{H_i W_j}^2 -\fnorm{H_i}^2} > t} \\
& \le &
2 \exp \left(- c\min\left(\frac{t^2}{(C_0^2\twonorm{H_i} \fnorm{H_i})^2}, \frac{t}{(C_0\twonorm{H_i})^2} \right)\right) 
\eens
Hence for rank $p$ matrix $H_i$, we recover the result in \eqref{eq::diag2exp} in case  
$H=I$,
\ben
\nonumber
\prob{\abs{\twonorm{\Z_j}^2 - \E \twonorm{\Z_j}^2} > t}
& \le &
\label{eq::rK}
2 \exp \left(- {c}\min\left(\frac{t^2}{p (C_0 \twonorm{H_i})^4 }, 
    \frac{t}{(C_0 \twonorm{H_i})^2} \right)\right) 
\een
\end{lemma}

Next we show that conclusion identical to those in
Lemma~\ref{lemma::ZZoporig} holds, upon updating events $\E_0$ and $\E_8$ for the operator
norm for anisotropic random vectors $\Z_j$.
Denote by $\E_0$ the event: for some absolute constant $C_{\diag}$,
\bens
\E_0 := \left\{\exists j \in [n]\quad \abs{\twonorm{\Z_j}^2 - \E
    \twonorm{\Z_j}^2} > C_{\diag}  (C_0 \max_{i} \twonorm{H_i})^2  (\sqrt{n p} \vee n)
\right\}
\eens
Denote by $\E_8$ the following event:  for some absolute constant $C_1$,
 \bens 
 \E_8: \quad 
\left\{\max_{q, h \in \Net}
  \sum_{i=1}^n  \sum_{j \not=i}^n \ip{\Z_{i}, \Z_{j}}  q_i h_j >  C_1 
  ( C_0 \max_{i} \twonorm{H_i})^2 (\sqrt{n p} \vee n)\right\}
  \eens 
  where $\Net$ is  the $\ve$-net of $\Sp^{n-1}$ for $\ve < 1/4$ as
  constructed in Lemma~\ref{lemma::ZZoporig}.

\subsection{Proof of Lemma~\ref{lemma::projWH}}
\label{sec::proofprojWH}
\begin{proofof2}
  %{Lemma~\ref{lemma::projWH}}
  Let $c, c', C_3, C_4$ be some absolute constants.
  Let $M_Y :=\E(Y)(Y-\E(Y))^T + (Y-\E(Y))\E(Y)^T$.
  Clearly, vectors $\Z_1, \Z_2, \ldots, \Z_n \in \R^{p}$ are independent. 
  Let $\mu$ be as in~\eqref{eq::definemu}.
  Then, we have by Lemma~\ref{lemma::anisoproj},
\ben
\nonumber
\lefteqn{\prob{\E_4} :=\prob{\max_{ u \in \{-1, 1\}^n}
    \sum_{i =1}^{n} u_i \ip{\Z_i, \mu}} \ge  \half C_4 n  (C_0  \max_i \twonorm{R_i \mu}) }\\
& \le&
2^n 2 \exp\left(- \frac{c' n^2 (C_0 \max_i \twonorm{R_i \mu})^2  }
  {C  (C_0 \max_i \twonorm{R_i \mu})^2 n}\right)  \le  2 \exp(- c' n).
\een
Thus we have on event $\E_4^c$, 
 \bens
 \sum_{i} \abs{\ip{\Z_i, \mu^{(1)} -\mu^{(2)}}}
& < & \half C_4  (C_0 \max_{i} \twonorm{R_i \mu}) n \sqrt{p \gamma}.
 \eens
Construct an $\ve$-net $\Pi_n$ of $\Sp^{n-1}$, where $\ve  = 1/3$ and
$\size{\Pi_n} \le (1+ 2/\ve)^{n}$. For a suitably chosen constant $C_3$,  we have
by Lemma~\ref{lemma::anisoproj},
\bens
\prob{\E_3} & := &
  \prob{\exists q \in \Pi_n,
    \abs{\sum_{i=1}^n q_i \ip{\Z_i, \mu}} \ge  \half C_3  (C_0 \max_i \twonorm{R_i \mu}) \sqrt{n }}\\
\label{eq::E8enso}
& \le&
9^n 2 \exp\left(- \frac{c' n (C_0 \max_i \twonorm{R_i
      \mu})^2}{C (C_0 \max_i \twonorm{R_i \mu})^2}\right)   \le  2 \exp(- c' n).
\eens
Moreover, by a standard approximation argument, we have on event $\E_3^c$,
\bens
 \sup_{q \in \Sp^{n-1}}  \sum_{i=1}^n q_i \ip{\Z_i, \mu}
  & \le &\inv{1-\ve} \sup_{q \in \Pi_n}  \sum_{i=1}^n q_i \ip{\Z_i,
    \mu}  \le  C_3 C_0 (\max_i \twonorm{R_i \mu})\sqrt{n}.
   \eens
 We have by Lemma~\ref{lemma::tiltproject},
 on event $\E_4^c \cap \E_3^c$,
 \bens
\norm{M_Y}_{\infty \to 1} & \le & 8 w_1 w_2 (n-1) \sum_{i=1}^n \abs{\ip{\Z_i, \mu^{(1)}
      -\mu^{(2)} }}  \le C_4 n (n-1)  (C_0 \max_{i} \twonorm{R_i
    \mu})\sqrt{p \gamma}  \\
\twonorm{M_Y} & \le &
4\sqrt{ w_1 w_2 n} \sup_{q \in \Sp^{n-1}} \abs{\sum_{i=1}^n  q_i \ip{\Z_i, \mu^{(1)} -\mu^{(2)}}} \le 
2 C_3 (C_0 \max_{i} \twonorm{R_i \mu}) n \sqrt{p \gamma}.
\eens
\end{proofof2}

\subsection{Proof of Lemma~\ref{lemma::anisoproj}}
\label{sec::projlemmas}
\begin{proof}
  %{Lemma~\ref{lemma::anisoproj} }
  Denote by $\mu =(\mu^{(1)}  -\mu^{(2)})/ \sqrt{p \gamma} \in 
  \Sp^{p-1}$.  First, we have by definition, $\Z_j = H_i W_j$,
  for each $j \in  \C_i, i =1, 2$; cf.~\eqref{eq::Wpsi}.
Hence  $\Z_j$ is a sub-gaussian random vector with its marginal $\psi_2$ norm bounded in the
sense of \eqref{eq::covZ1} and \eqref{eq::covZ2} with
\ben
\label{eq::covW1}
\norm{\ip{\Z_j,  \mu}}_{\psi_2} & := &
\norm{\ip{H_i W_j, \mu}}_{\psi_2} 
\le \norm{W_j}_{\psi_2} \twonorm{R_i \mu} \le C_0  \twonorm{R_i \mu}.
\een
where $W_j \in \R^{m}$ is a mean-zero, isotropic, sub-gaussian random 
vector satisfying $\norm{W_j}_{\psi_2} \le C_0$;
Hence \eqref{eq::Zcovdef} holds and for all $j \in \C_i$,
\ben
\norm{\ip{\Z_j, \mu^{(1)}  -\mu^{(2)}}}_{\psi_2}
& \le &  
C_0 \sqrt{p \gamma} \twonorm{R_i \mu} \; \; \text{ and } \;  
R_i = H_i^T.
\een
First, we have by independence of $\Z_j, \forall j$,
\bens 
\forall q \in \Sp^{n-1}, \quad 
\norm{\sum_{i=1}^n q_i \ip{\Z_i, \mu^{(1)}-\mu^{(2)}}}^2_{\psi_2}
& \le &  C \sum_{i=1}^n q_i^2 \norm{\ip{\Z_i, \mu^{(1)}-\mu^{(2)}}}^2_{\psi_2}  \\
& \le &  C p \gamma (C_0 \max_{i} \twonorm{R_i \mu})^2  \\
\text{ and for } \;  u_i \in \{-1, 1\}, \; \; 
\norm{\sum_{i=1}^n u_i \ip{\Z_i, \mu^{(1)}-\mu^{(2)}}}^2_{\psi_2}
& \le &  C \sum_{i=1}^n \norm{\ip{\Z_i, \mu^{(1)}-\mu^{(2)}}}^2_{\psi_2}  \\
& \le &  C n p \gamma (C_0 \max_{i} \twonorm{R_i \mu})^2.
\eens
Then \eqref{eq::proanis} and  \eqref{eq::proanisum} follow from
the sub-gaussian tail bound, for example, Propositions  2.6.1 and 2.5.2 (i)~\cite{Vers18}.
See also the proof for Lemma~\ref{lemma::isotropic}.
\end{proof}

\subsection{Proof of Theorem~\ref{thm::ZHW}}
\label{sec::proofofZHW}
 \label{sec::opcorrelated}
In the rest of this section, we prove Theorem~\ref{thm::ZHW}. The proof may be of
independent interests. 
 We generate $\Z$ according to Definition~\ref{def::WH}:
 \bens
 \Z_j & = & H_j W_j  \in \R^{p}\; \; \text{ and hence}\;\;
\Z  = \sum_{j=1}^n e_j W_j^T H_j^T = \sum_{j=1}^n \diag(e_j)  \BW  H_j^T
\eens
where $W_1^T, \ldots, W_n^T \in \R^{m}$ are independent, mean-zero,
isotropic row vectors of $\BW = (w_{jk})$, where we assume that coordinates $w_{jk}$
are also independent with $\max_{j,k} \norm{w_{jk}}_{\psi_2} \le C_0$.
Throughout this section, let $\mvec{\Z} =
\mvec{X - \E X}$ be formed by concatenating columns of  matrix $\Z$
into a long vector of size $np$.
Denote by $\otimes$ the tensor product. Recall
$e_1, \ldots, e_n$ are the canonical basis of $\R^n$.

\begin{proofof}{Theorem~\ref{thm::ZHW}}
  By Definition~\ref{def::WH},
  \ben
  \nonumber
\mvec{\Z} & = & \sum_{j=1}^n \mvec{\diag(e_j)  \BW H_j^T }  = \sum_{j=1}^n H_j 
\otimes \diag(e_j) \mvec{\BW} \\
  \label{eq::LW}
& =: &  L \mvec{\BW}  \in \R^{n p} \; \; \text{ where } \; L :=\sum_{j=1}^n H_j 
\otimes \diag(e_j) \in \R^{np \times mn}
\een
On the other hand, we have for $\BW \in \R^{n \times m}$, $\Z_j = H_j W_j$ and hence 
\ben
\nonumber
\Z^T & = &  [\Z_1, \ldots, \Z_n]
= \sum_{j=1}^n \Z_j \otimes e_j^T = \sum_{j=1}^n H_j \BW^T \diag(e_j) \; \text{ and hence} \; \;\\
\nonumber
\mvec{\Z^T} 
& =  & \mvec{\sum_{j=1}^n H_j \BW^T \diag(e_j)}  =   \sum_{j=1}^n
\mvec{H_j \BW^T \diag(e_j)} \\
\label{eq::RWT}
& = &\sum_{j=1}^n (\diag(e_j) \otimes H_j) \mvec{\BW^T} 
=: R \mvec{\BW^T}
\een
Then there exist some permutation matrices $P, Q$ such that
\ben
\label{eq::permute}
L   =   \sum_{i=1}^n H_i 
\otimes \diag(e_i)
& =  & P^T \big(\sum_{i=1}^n \diag(e_i)
\otimes H_i  \big)  Q =: P^T R Q
\een
where $H_j H_j^T$ denotes the covariance matrix for each row vector
$\Z_j$, $j \in [n]$.

We now show \eqref{eq::permute} with an
explicit construction. It is well known that there exist permutation matrices $P, Q$ such
that 
\ben
\label{eq::definePQ}
\mvec{\Z^T}   & = & P \mvec{\Z} \; \text{ and } \;  \mvec{\BW^T} = Q
\mvec{\BW} \\
\nonumber
\text{ and hence by}~\eqref{eq::LW}, \;
\mvec{\Z^T} & = & P \mvec{\Z} = P L \mvec{\BW}
\een
On the other hand, we have by \eqref{eq::RWT} and \eqref{eq::definePQ}
\bens
\mvec{\Z^T} & = &
R \mvec{\BW^T}= R Q \mvec{\BW}.
\eens
This shows that for $P^T = P^{-1}$,
\bens
PL = R Q \; \; \text{ and hence} \;\; L =P^T R Q
\eens
and   hence~\eqref{eq::permute} indeed holds.
See Lemma 4.3.1 and Corollary 4.3.10~\cite{HJ91}.

First we rewrite the quadratic form as follows: for any matrix $A = (a_{ij}) \in \R^n$,
   \bens
  \label{eq::quadZZ}
  \abs{\sum_{i=1}^n  \sum_{j \not=i}^n \ip{\Z_{i}, \Z_{j}}  a_{ij}}
&  = &
  \abs{\sum_{i=1}^n  \sum_{j \not=i}^n a_{ij} \sum_{k=1}^p z_{ik} z_{jk}  } \\
  &  = &
  \mvec{\Z}^T \tilde A \mvec{\Z}  =  \mvec{\BW}^T L^T \tilde A L  \mvec{\BW},
  \eens  
  where $\tilde A$  is a block-diagonal matrix with $\diag(\tilde{A}) = 0$,
  and $p$ identical blocks $\tilde{A}^k =\offd(A), \forall k \in [p]$ of size $n 
  \times n$ along the main diagonal, where $\twonorm{\offd(A)} \le
  \twonorm{A} +   \twonorm{\diag(A)} \le 2\twonorm{A}$.
We now compute
\bens 
\twonorm{L^T \tilde{A} L}
  & \le &  \twonorm{\sum_{i=1}^n H_i \otimes \diag(e_i)}^2 
  \twonorm{\tilde A}   \le \max_{i} \twonorm{H_i}^2  \twonorm{\tilde A}   \\
  \text{where  }
   \twonorm{L}   & = &  \twonorm{R} =
  \twonorm{\sum_{i=1}^n \diag(e_i)  \otimes H_i  }  \le   \max_{i} \twonorm{H_i} \\
  \fnorm{L^T \tilde{A} L}
  & \le &   \twonorm{L}^2 \fnorm{\tilde A}  \le  \max_{i} \twonorm{H_i}^2   \fnorm{\tilde A}    \le  \max_{i} \twonorm{H_i}^2 \sqrt{p}   \fnorm{A},
  \eens
  where we use the property of block-diagonal matrix for $R =
  \sum_{i=1}^n \diag(e_i)  \otimes H_i$, which is also known as
  a direct sum over $H_i, i=1, \ldots, n$.

Hence for any $t >0$, by the Hanson-Wright inequality
(Theorem~\ref{thm::HW}),
\bens 
\lefteqn{\prob{\abs{\mvec{\Z}^T \tilde A \mvec{\Z} } > t}
 = 
\prob{\abs{\mvec{\BW}^T L^T \tilde{A}  L\mvec{\BW} } > t} }\\
  & \leq &
  2 \exp \left(- c\min\left(\frac{t^2}{C_0^4 (\max_i \twonorm{H_i}^4)  
        \fnorm{\tilde{A}}^2},
\frac{t}{C_0^2  (\max_i \twonorm{H_i}^2) \twonorm{\tilde{A}}}\right) \right) \\
& \leq & 2 \exp \left(- c\min\left(\frac{t^2}{C_0^4 (\max_{i} \twonorm{H_i}^4) 
      p \fnorm{A}^2}, \frac{t}{C_0^2  (\max_i \twonorm{H_i}^2) \twonorm{A}}\right) \right).
\eens
Thus~\eqref{eq::genDH} holds.
\end{proofof}

\subsection{Proof of Lemma~\ref{lemma::E0correlated}}
\label{sec::proofE0corre}
\begin{proofof2}
  %{Lemma~\ref{lemma::E0correlated}}
  We prove the lemma with the full generality by allowing each row 
  vector to have its own covariance $A_i= H_i H_i^T$, where $A_i \in \R^{p \times p}$, is the
covariance for row vector $\Z_j \in \R^{p}$ for $j \in \C_i$ as shown in \eqref{eq::rowcov}.
Now we also introduce the positive semidefinite matrix $M:= H_i^T H_i
\succeq 0 \in \R^{m \times m}$. First, we bound the $\ell_2$ norm for
each anisotropic vector $\Z_j^T = W^T_j H^T_i \in \R^{p}$, where $W_j \in
\R^m$, and $j \in \C_i$
\ben
\label{eq::Znorm}
\twonorm{\Z_j}^2
& = &  \ip{H_i W_j, H_i W_j} = W_j^T H_i^T H_i W_j =:  W_j^T M W_j \\
\nonumber 
\text{where  } \; \; \tr(H_i H_i^T)
& = & \tr(M) = \fnorm{H_i}^2 \le (m
\wedge p ) \twonorm{H_i}^2,
\een
where $W_{j}$ is an isotropic sub-gaussian random vectors with independent, mean-zero, coordinates, and in \eqref{eq::Znorm}, we use the isotropic property of $W_j$.
Now clearly,
\ben
\label{eq::Mfnorm}
\fnorm{M}  = \fnorm{H_i^T H_i} \le \twonorm{H_i} \fnorm{H_i}  \quad \text{ and }
\twonorm{M} & = & \twonorm{H_i}^2;
\een
Thus we have for any $t > 0$,
\bens 
\lefteqn{\prob{\abs{\twonorm{\Z_j}^2 - \E \twonorm{\Z_j}^2} > t}
  = \prob{\abs{\twonorm{H_i W_j}^2 - \E \twonorm{H_i W_j}^2} > t}  }\\
& = & \prob{\abs{W^T_j M W_j -\fnorm{H_i}^2} > t }  \le 
2 \exp \left(- c\min\left(\frac{t^2}{C_0^4 \fnorm{M}^2},
    \frac{t}{C_0^2 \twonorm{M}} \right)\right) \\
& \le & 2 \exp \left(- c\min\left(\frac{t^2}{C_0^4 
\twonorm{H_i}^2  \fnorm{H_i}^2}, \frac{t}{C_0^2 
\twonorm{H_i}^2}  \right)\right),
\eens
and hence we can also recover the result in 
\eqref{eq::diag2exp} in case $H_i=I$. 
\end{proofof2}

\subsection{Proof of Theorem~\ref{thm::YYcovcorr} }
\label{sec::ZZHproof}
\begin{proofof2}
 First, we choose 
$t_{\diag}  = C (\max_i \twonorm{H_i}^2) C_0^2 (\sqrt{n p} \vee n)$
and finish the calculations. 
First, we have by Lemma~\ref{lemma::E0correlated},
\bens 
\lefteqn{\prob{\max_{i} \max_{j \in \C_i} \abs{\twonorm{\Z_j}^2 - \E \twonorm{\Z_j}^2} > t_{\diag}}
    =: \prob{\E_0}  = }\\
& \le & 2 n \exp \left(- c\min\left(\frac{(\max_i \twonorm{H_i}^4) n
      p}{\max_i (\twonorm{H_i}^2 \fnorm{H_i}^2)}, \frac{(\max_i
      \twonorm{H_i}^2)n}{\max_i (\twonorm{H_i}^2)} \right)\right) \\
& \le &
2 n \exp \left(- c\min\left(\frac{n p}{ (m \wedge p)}, n\right)\right) 
\le  2 \exp(- c' n),
\eens
where for the $p \times m$ matrix $H_i$,  we have $\fnorm{H_i} \le \sqrt{p 
  \wedge m} \twonorm{H_i}$.
We use Theorem~\ref{thm::ZHW} to bound the off-diagonal part.
Hence for all $q, h \in \Sp^{n-1}$, $A(q,h) := \offd(q \otimes h) =(a_{ij})$,
\bens 
\twonorm{{A}(q, h)} \le  \fnorm{A(q,h)} \le 1
\eens 
Let $t_{\offd} = C_{\offd} C_0^2 (\max_{i} \twonorm{H_i}^2) ( \sqrt{p n}
\vee n)$.
For a particular realization of $q, h \in \Sp^{n-1}$ and ${A}(q, h)
=(a_{ij})$ as defined above, and Theorem~\ref{thm::ZHW},
\bens 
\lefteqn{\prob{\abs{\sum_{i=1}^n  \sum_{j \not=i}^n \ip{\Z_{i}, \Z_{j}}
      q_{i} h_{j}} > t_{\offd}}  =  \prob{\sum_{i=1}^n  \sum_{j \not=i}^n \ip{\Z_i, \Z_j} a_{ij} > t_{\offd}}  } \\
& \le &
2 \exp \left(- c\min\left(\frac{C_0^4 (\max_{i} \twonorm{H_i}^4)
      ( \sqrt{p n}\vee n)^2}{C_0^4 (\max_{i}  \twonorm{H_i}^4) p \fnorm{A(q,h)}^2},
    \frac{C_0^2 (\max_{i} \twonorm{H_i}^2) ( \sqrt{p n}
\vee n)}{C_0^2 (\max_{i} \twonorm{H_i}^2) \twonorm{A(q, h)}}  \right)\right)  \\
& \le &
2 \exp \left(- c n\min (C_1^2, C_1)\right) \leq    2 \exp(-c n)
\eens
for some sufficiently large constants $C_1$ and $c > 4 \ln 9$.
Let $\Net$ be as defined in Lemma~\ref{lemma::ZZoporig}.

Taking a union bound over all $\abs{\Net}^2 \le 9^{2n}$ pairs $q, h \in \Net$, the
$\ve$-net of $\Sp^{n-1}$, we conclude that
\bens
\lefteqn{\prob{\max_{q, h \in \Net}
 \abs{\sum_{i=1}^n \sum_{j \not= i}^n q_i  h_j \ip{\Z_{i}, \Z_{j}}}>
 t_{\offd}} =: \prob{\E_8}} \\
&\le &
\abs{\Net}^2 \cdot 2 \exp \left(- c n \min(C_1^2, C_1) 
\right)   \leq    2 \times 9^{2n} \exp \left(- c n \min(C_1^2, C_1)   \right) \\
& \le & 2\exp\left(-c n + 2n \ln 9\right) = 2 \exp\left(-c_3 n \right)
\eens
One can show that \eqref{eq::ZZHop} holds by a standard approximation 
argument under $\E_8^c$: we have 
\ben
\nonumber
  \twonorm{\offd(\Z \Z^T)}
  & =  &
  \sup_{q\in \Sp^{n-1}}   \sum_{i=1}^n \sum_{j \not=i}^n 
  q_i q_j \ip{\Z_{i}, \Z_{j}}  \le
\inv{(1-2\ve)} \sup_{q, h\in \Net}   \sum_{i=1}^n \sum_{j \not=i}^n 
q_i h_j \ip{\Z_{i}, \Z_{j}} \\
\label{eq::Zop2}
& \le & 2  (\max_{i} \twonorm{H_i}^2)C_1 C_0^2 (\sqrt{ np} \vee n)
 \een
 See for example Exercise 4.4.3~\cite{Vers18}.
 Thus we have on event $\E_8^c \cap 
 \E_0^c$,
 \bens
\twonorm{\Z \Z^T- \E (\Z \Z^T) } &\le &
\norm{\diag(\Z \Z^T) -\E \diag(\Z \Z^T)}_{\max} +\twonorm{\offd(\Z \Z^T)} \\
& \le &  C' C_0^2 (\max_{i} \twonorm{H_i}^2)  (\sqrt{ np} \vee n)
\eens
Moreover, we have $\prob{\E_8^c \cap \E_0^c}  \ge 1- 2 \exp(-c n)$
upon adjusting the constants.
\end{proofof2}

\bibliography{final,subgaussian,clustering}

\end{document}